\documentclass[11pt]{article}
\usepackage[margin=1in]{geometry}
\usepackage{amsmath,amssymb,amsthm,mathtools}
\usepackage{graphicx,booktabs,longtable,array,multirow}
\usepackage{placeins}
\usepackage[dvipsnames]{xcolor}
\usepackage{float,algpseudocode}
\floatstyle{ruled}
\newfloat{algorithm}{tbp}{loa}
\floatname{algorithm}{Algorithm}
\usepackage[numbers,sort&compress]{natbib}
\usepackage[colorlinks=true,linkcolor=blue!50!black,citecolor=blue!50!black,urlcolor=blue!50!black]{hyperref}
\usepackage{doi}
\makeatletter
\g@addto@macro\UrlBreaks{\do\-}
\makeatother
\usepackage{microtype}
\microtypesetup{expansion=false}
\usepackage{enumitem}
\graphicspath{{./}}
\usepackage[T1]{fontenc}
\usepackage{lmodern}
\hypersetup{hidelinks,hypertexnames=false}
\usepackage[bottom]{footmisc}
\AtBeginDocument{%
  \setlength{\abovedisplayskip}{7pt plus 2pt minus 3pt}%
  \setlength{\belowdisplayskip}{7pt plus 2pt minus 3pt}%
  \setlength{\abovedisplayshortskip}{4pt plus 2pt minus 2pt}%
  \setlength{\belowdisplayshortskip}{4pt plus 2pt minus 2pt}%
}

\usepackage[compact]{titlesec}

\titlespacing*{\section}
  {0pt}{1.55ex plus 0.35ex minus 0.25ex}{0.75ex plus 0.15ex}

\titlespacing*{\subsection}
  {0pt}{1.25ex plus 0.30ex minus 0.20ex}{0.55ex plus 0.12ex}

\titlespacing*{\subsubsection}
  {0pt}{1.00ex plus 0.25ex minus 0.15ex}{0.45ex plus 0.10ex}

\titleformat{\section}
  {\normalfont\large\bfseries}{\thesection.}{0.5em}{}

\titleformat{\subsection}
  {\normalfont\normalsize\bfseries}{\thesubsection.}{0.5em}{}

\titleformat{\subsubsection}
  {\normalfont\normalsize\itshape}{\thesubsubsection.}{0.5em}{}

\setlist{
  topsep=2pt plus 1pt minus 1pt,
  itemsep=1pt plus 0.5pt,
  parsep=0pt,
  partopsep=0pt,
  leftmargin=*
}

\makeatletter
\def\thm@space@setup{%
  \thm@preskip=6pt plus 2pt minus 2pt
  \thm@postskip=6pt plus 2pt minus 2pt
}
\makeatother

\makeatletter
\renewcommand{\maketitle}{%
  \begin{center}
    {\Large\bfseries \@title\par}
    \vskip 0.45em
    {\normalsize
      \lineskip .5em
      \begin{tabular}[t]{c}
        \@author
      \end{tabular}\par}
    \vskip 0.35em
    {\small \@date\par}
  \end{center}
  \par\vskip 0.35em
  \@thanks
  \vskip 0.7em
}
\makeatother

\renewenvironment{abstract}
  {\small\begin{center}\bfseries Abstract\end{center}\vspace{-0.6em}\small}
  {\par\vspace{0.6em}}

\newtheorem{theorem}{Theorem}[section]
\newtheorem{lemma}[theorem]{Lemma}
\newtheorem{proposition}[theorem]{Proposition}
\newtheorem{corollary}[theorem]{Corollary}
\theoremstyle{definition}
\newtheorem{definition}[theorem]{Definition}

\theoremstyle{remark}

\newcommand{\R}{\mathbb R}
\newcommand{\C}{\mathbb C}
\newcommand{\Z}{\mathbb Z}
\newcommand{\Q}{\mathbb Q}
\newcommand{\cF}{\mathcal F_{\mu,L}}
\newcommand{\cS}{\mathcal S}
\newcommand{\cQ}{\mathcal Q}
\newcommand{\cL}{\mathcal L}
\newcommand{\ip}[2]{\langle #1,#2\rangle}
\newcommand{\norm}[1]{\lVert #1\rVert}
\newcommand{\rh}{\rho_{\rm HB}}
\newcommand{\rs}{\rho_{\rm sp}}
\newcommand{\rhoq}{\rho_{\rm q}}
\newcommand{\sh}{\sinh}
\newcommand{\ch}{\cosh}
\newcommand{\arsinh}{\operatorname{arsinh}}
\newcommand{\arcosh}{\operatorname{arcosh}}
\newcommand{\Rea}{\operatorname{Re}}
\newcommand{\Ima}{\operatorname{Im}}
\newcommand{\supp}{\operatorname{supp}}
\newcommand{\sgn}{\operatorname{sgn}}
\newcommand{\FP}{\mathrm{FP}}
\newcommand{\PP}{\mathrm{P}}
\newcommand{\T}{\mathbb T}
\newcommand{\cl}{\operatorname{cl}}
\newcommand{\ov}[1]{\overline{#1}}

\title{Two Answers on Heavy-Ball Dynamics}
\author{Dmitry Pasechnyuk-Vilensky\thanks{Mohamed bin Zayed University of Artificial Intelligence, Abu Dhabi, UAE. \texttt{dmivilensky1@gmail.com}}, Martin Tak\'{a}\v{c}\thanks{Mohamed bin Zayed University of Artificial Intelligence, Abu Dhabi, UAE. \texttt{Takac.MT@gmail.com}}}
\date{}
\begin{document}
\maketitle

\begin{abstract}
We determine the worst-case behavior of the heavy-ball method on smooth
strongly convex functions, first for Polyak's quadratic-optimal parameters
and then for arbitrary constant parameters. For Polyak's tuning, the exact
dimension-uniform exponential rate is the largest growth factor of an
interpolable mixture of rotating geometric sequences. Four algebraic
families determine this rate at every condition number. Two frequencies
suffice, and one function in dimension five attains the rate at every
horizon up to constant factors. We recover the sharp convergence threshold
$9+4\sqrt5$ of Badithela and Seiler. For general tunings, an exact linear feasibility criterion characterizes
dimension-uniform linear convergence through the angle between the
characteristic polynomials of the two extreme quadratics. This gives a continuous condition-number
threshold and resolves the second heavy-ball riddle of Goujaud, Taylor and
Dieuleveut. The regions between consecutive cycle tongues are Farey gaps,
each containing convergent and divergent tunings. Regular portions of the
threshold surface are analytic plateaus with fixed periodic frequency
support. Every rational rotation number in a gap anchors a single-harmonic
plateau; mixed-harmonic plateaus and an exact two-lag threshold point also
occur. Along tongue edges, explicit multiplier inequalities and
second-order expansions of competing Farey fans bound the transition from
single-harmonic to mixed-harmonic extremals. The rate bounds and the
extremal constructions follow from a common interpolation and separation
argument that preserves initialization with zero momentum.
\end{abstract}

\section{Introduction}\label{sec:intro}

Polyak's heavy-ball method accelerates strongly convex quadratic
minimization by retaining the preceding displacement. On general smooth
strongly convex functions, the same tuning can produce periodic or
exponentially growing trajectories. We determine its exact worst-case rate
under Polyak's tuning and characterize convergence for arbitrary constant
parameters through spectral interpolation.

\paragraph{Problem setting}
Let $0<\mu<L$, $\kappa=L/\mu$, and let $\cF(\R^d)$ denote the class of
differentiable $\mu$-strongly convex functions on $\R^d$ with
$L$-Lipschitz gradients. For a step size $\alpha>0$ and a momentum
$\beta\in[0,1)$, the heavy-ball iteration is
\begin{equation}\label{eq:hb}
 x_{k+1}=x_k-\alpha\nabla f(x_k)+\beta(x_k-x_{k-1}),
 \qquad x_{-1}=x_0.
\end{equation}
Throughout, the initial momentum is zero. Writing $x^\star$ for the unique
minimizer of $f$, define
\begin{equation}\label{eq:WN}
 W_N^{(d)}=
 \sup_{\substack{f\in\cF(\R^d)\\x_0\ne x^\star}}
 \frac{\norm{x_N-x^\star}}{\norm{x_0-x^\star}},
 \qquad W_N=\sup_{d\ge1}W_N^{(d)},
 \qquad \rho(\alpha,\beta)=\limsup_{N\to\infty}W_N^{1/N}.
\end{equation}
Thus $\rho$ is the dimension-uniform worst-case exponential rate;
convergence on the class means $\rho<1$, with a linear bound uniform in
the function and dimension. The quadratic subclass gives a lower bound:
its rate $\rhoq(\alpha,\beta)$ is the largest root modulus of
$z^2-(1+\beta-\alpha\lambda)z+\beta$, $\lambda\in[\mu,L]$.
Convergence on all these quadratics is equivalent to
\begin{equation}\label{eq:quad}
 0<\alpha L<2(1+\beta).
\end{equation}
Polyak's parameters~\cite{Polyak}
\begin{equation}\label{eq:polyak}
 \alpha=\frac4{(\sqrt L+\sqrt\mu)^2},
 \qquad \beta=q^2,
 \qquad q=\frac{\sqrt\kappa-1}{\sqrt\kappa+1}
\end{equation}
balance the extreme quadratic roots and give $\rhoq=q$. Part~I determines
the nonlinear rate and its extremals for every $\kappa>1$, including
$\rho\ge1$; Part~II treats all tunings satisfying \eqref{eq:quad}.

\paragraph{Previous results and the two heavy-ball riddles}
Ghadimi, Feyzmahdavian and Johansson~\cite{GFJ} established global
convergence, with linear bounds in the strongly convex case, for suitable
constant parameters. Lessard, Recht and Packard~\cite{LRP} introduced
integral quadratic constraint (IQC) bounds and exhibited nonconvergence under
Polyak's tuning. Badithela and Seiler~\cite{BS} identified the sharp
threshold $9+4\sqrt5$: the weighted off-by-one bound reaches one and a
scalar three-cycle appears. Interpolation also supports automated searches
for cyclic counterexamples~\cite{GDT23}.

Goujaud, Taylor and Dieuleveut~\cite{GTD} proved non-acceleration in
dimensions at least two using planar root-of-unity cycles, and posed two
further questions~\cite{GTD25}. Fotopoulos, Helcig, Deck and
Alsbach~\cite{FHDA} resolved the first, concerning dimension one, by
rotating and projecting these cycles onto a line. They also treated
coprime rotations, making $j$ turns in $k$ steps, whose binding
interpolation lag $j^{-1}\bmod k$ is the time separation of the
corresponding iterates.

The second riddle concerns gaps between a computed Lyapunov convergence
region and the integer-cycle regions. Evaluating the cycle condition at
fractional periods fills these gaps~\cite{GTD25}, leaving their actual
convergence behavior unresolved. Wang~\cite{Wang26}, also discussed
in~\cite[Appendix~B.1]{WLY26}, certified a convergent parameter box
outside both tests. We determine the gap geometry and give an exact
criterion for its convergent and divergent portions.

Liu et~al.~\cite{Liu26} quantify local attraction under Hessian regularity
and give a gradient-based initialization rule. Hernanz~\cite{Hernanz26}
obtains the optimal asymptotic distance factor
$1-8/\kappa+o(\kappa^{-1})$ after optimizing constant tunings, for
arbitrary initial pairs in dimensions at least two. His extension to starts
from rest in dimension three uses our additional-coordinate construction.
Pokutta and Stellato~\cite{PS26} determine the exact asymptotic
squared-distance factor of Nesterov's method through mixtures of planar
spirals and linear optimization over measures.

\paragraph{Main results}
For Polyak's tuning, Theorems~\ref{thm:spectral}
and~\ref{thm:classification} characterize the exact rate by interpolable
mixtures of rotating geometric sequences and determine it through four
algebraic families: a special spiral, a single-frequency family, and two
two-frequency families. At most two frequencies suffice.
Theorem~\ref{thm:rate} gives matching bounds on $W_N$ up to constants
independent of $N$, attained from below at every horizon by one function
in $\cF(\R^5)$ and one initial point. At the critical condition number a
trajectory from rest approaches a three-cycle; above it an explicit
trajectory grows exponentially (Theorem~\ref{thm:threshold}).

For general tunings, let $\Theta(\omega)$ be the angle between the
characteristic polynomials of the extreme quadratics on the unit circle.
Theorem~\ref{thm:A} gives convergence if $\Theta(\omega)>\omega/2$ for
all $\omega\in(0,\pi]$, and a trajectory approaching a $k$-periodic
orbit if $\Theta(2\pi j/k)\le\pi/k$. The rotation number
$\omega/(2\pi)$ measures turns per step; fractional periods correspond
to contacts with $\omega/2$. The gaps lie between resonances
$1/(K+1)$ and $1/K$.
One strict spectral certificate at one exponent bounds the rate
(Theorem~\ref{thm:oneexp}). At exponent zero, the resulting exact
criterion expresses the divergence set as
$\kappa\ge\kappa_\partial(\alpha L,\beta)$, with continuous
$\kappa_\partial\in(1,\infty]$
(Theorems~\ref{thm:exactthr} and~\ref{thm:global}). Every gap contains
convergent tunings near its lag-one boundary and divergent tunings meeting
each rational tongue; for $K\ge3$, divergence also reaches an arc of
the junction of the bounding tongues
(Theorem~\ref{thm:BK} and Proposition~\ref{prop:corner}).

Regular boundary points lie on analytic plateaus with fixed finite
harmonic support, where finitely many lag constraints decide convergence
(Theorems~\ref{thm:modelock} and~\ref{thm:tightclass}). Every Farey
fraction $j/k$ in every gap anchors a single-harmonic plateau with
$\kappa_\partial=\kappa_{j/k}$
(Theorem~\ref{thm:plateaujk} and Corollary~\ref{cor:everyfarey}).
Two-lag certificates and competing two-frequency Farey fans give explicit
bounds on its edge away from a discrete exceptional set, the latter
through a second-order inverse-period expansion
(Theorems~\ref{thm:edgejk} and~\ref{thm:fanjk}). Exact
examples exhibit convergence in a gap, mixed-harmonic extremals and a
two-lag boundary point (Propositions~\ref{prop:13}
and~\ref{prop:conv-gap}, Theorems~\ref{thm:plateaus}
and~\ref{thm:exact313}).

\paragraph{Methods and relation to control theory}
Smooth strongly convex interpolation tests whether prescribed points,
gradients and values belong to one function~\cite{THG,AM}. A frequency
describes a rotating geometric sequence; a positive measure combines
these sequences. The interpolation inequalities are linear in this
measure and indexed by time lags. For periodic trajectories, this
harmonic linear program is established in~\cite[Theorem~6.13]{GTD} and
rederived in~\cite[Proposition~5.1]{FHDA}, which also gives a mixed
cycle beyond the single-harmonic regions. We extend the formulation to
arbitrary spectra and growth exponents. Separation gives finite dual
certificates whose lag inequalities yield energy bounds for all
trajectories; an additional coordinate realizes extremal spectra from rest.

These lag certificates are discrete-time O'Shea--Zames--Falb
multipliers~\cite{ZF,WB68}, with a Toeplitz frequency
representation~\cite{GS} and energy estimates in the IQC
framework~\cite{MR,LRP}. Related results include quadratic Lyapunov
searches~\cite{TVSL}, periodic counterexamples~\cite{SC}, harmonic
duality bounds~\cite{ZCH} and converse multiplier theorems~\cite{SSCK}.
For slope-restricted feedback systems, Kharitenko and Scherer~\cite{KS24}
prove exactness of the multiplier test for input-output stability uniform
over the lifted plants $G\otimes I_d$ for all finite $d$;
Kharitenko~\cite{Kharitenko26} shows conservatism for a scalar
interconnection. Our criterion concerns convergence uniform over all
finite dimensions, with zero initial momentum and explicit autonomous
realizations.

For Polyak's tuning, the spectral constraints reduce to polynomials in
$\cos\omega$. One family has a rational parametrization; another uses
rescaled variables at the critical three-cycle and a quintic number field.
Polynomial contact systems and rational interval continuation determine
the four families; polynomial positivity and
analytic lag tails complete the classification~\cite{Moore,Arb,Rump2010}.
For general tunings, the harmonic minimum principle extends one-sided
certificates from circle boundaries to disks or annuli, and angle convexity and Farey
relations~\cite{HW} organize the gaps. Strict complementarity and
analytic continuation produce the plateaus; contact curvature and
quadratic minimization over lag residues control their edges.

\paragraph{Formal verification.}
The mathematical results are accompanied by formal proofs in Lean. The package is available at
\href{https://github.com/dmivilensky/PolynomialRealizationsProgram/tree/main/polyak}{\texttt{PolynomialRealizationsProgram/polyak}}\footnote{\url{https://github.com/dmivilensky/PolynomialRealizationsProgram/tree/main/polyak}}.

\paragraph{Source code.}
The code for the numerical experiments, figures, exact certificate generation
and proof-data archives is available at
\href{https://github.com/dmivilensky/PolynomialRealizationsProgram/tree/main/polyak-code}{\texttt{PolynomialRealizationsProgram/polyak-code}}\footnote{\url{https://github.com/dmivilensky/PolynomialRealizationsProgram/tree/main/polyak-code}}.

\section{Polyak's tuning}\label{sec:part1}
We first fix Polyak's parameters~\eqref{eq:polyak} and determine the
worst-case rate as a function of the condition number. A spectral
characterization reduces this problem to four algebraic families with
at most two frequencies. Their realization from rest gives matching
bounds at every horizon and the sharp convergence threshold.

\subsection{Exact worst-case rate under Polyak's tuning}\label{sec:polyak}
For the parameters \eqref{eq:polyak}, write $\rh(\kappa)=\rho(\alpha,\beta)$,
$\varphi=(1+\sqrt5)/2$ and $\kappa_c=9+4\sqrt5=\varphi^6$.
The worst-case rate is determined by four algebraic families of spectra.
There is an optimal spectrum with at most two frequencies; a
single-frequency construction gives the sharp threshold for convergence.
Let $\rs(\kappa)$ denote the unique positive root of
\begin{equation}\label{eq:spiralpoly}
 F(\rho,q)=\rho^2(\rho^2+1)-2q\rho(\rho^2-\rho+1)-q^2(\rho^2+1)=0 .
\end{equation}
\begin{theorem}[Exact rate and realization]\label{thm:rate}
For every $\kappa>1$ the limits
\[
 \rh(\kappa)=\lim_{N\to\infty}W_N(\kappa)^{1/N}=\lim_{N\to\infty}W_N^{(5)}(\kappa)^{1/N}
\]
exist, and $\rh(\kappa)$ is given by the finite algebraic prescription of
Theorem~\ref{thm:classification}. There are constants
$0<c_\kappa\le C_\kappa<\infty$ such that
\begin{equation}\label{eq:bounds}
 c_\kappa\,\rh(\kappa)^N\le W_N^{(5)}(\kappa)\le W_N(\kappa)\le C_\kappa\,\rh(\kappa)^N\quad(N\ge0).
\end{equation}
The lower bound is attained by one function in $\cF(\R^5)$ and one initial
point, independent of $N$. Moreover, one function in $\cF(\R^3)$ and one
initial point satisfy
\begin{equation}\label{eq:spirallower}
 \norm{x_N-x^\star}\ge\frac{\rs(\kappa)^N}{\sqrt5}\,\norm{x_0-x^\star}\quad(N\ge0).
\end{equation}
\end{theorem}
The proof is given in Appendix~\ref{app:finite-uniform}; the explicit three-dimensional construction is in Appendix~\ref{app:finite-lower}.

\begin{theorem}[Sharp convergence threshold]\label{thm:threshold}
The heavy-ball method with Polyak's tuning converges on every
$f\in\cF$ from every initial point if and only if $\kappa<\kappa_c$. More
precisely,
\[
 \rh(\kappa)<1\quad(1<\kappa<\kappa_c),\quad \rh(\kappa_c)=1,\quad
 \rh(\kappa)>1\quad(\kappa>\kappa_c),
\]
and $\rh(\kappa)\le\sqrt{\varphi q}$ for $\kappa<\kappa_c$. At $\kappa_c$ a
function in $\cF(\R^3)$ has a trajectory with
$\norm{x_N-x^\star}\ge\norm{x_0-x^\star}/\sqrt5$ for all $N$, which converges
to a three-cycle; above $\kappa_c$ the same construction diverges
exponentially.
\end{theorem}
The proof is given in Appendix~\ref{app:classification-spiral}.

Badithela and Seiler~\cite{BS} identified $\kappa_c$ and a three-cycle of a
scalar piecewise quadratic function at this condition number. Their
explicit solution of the weighted off-by-one linear matrix inequality
also gives $\rho\le1$ at $\kappa_c$. Theorem~\ref{thm:threshold} recovers
their threshold and combines it with the exact rate and a realization
from rest valid at every horizon, including exponential growth above
the threshold.

\begin{theorem}[Regularity and critical behavior]\label{thm:regularity}
The function $\rh$ is continuous and strictly increasing on $(1,\infty)$ and
real analytic outside the five switching points
$\kappa_W<\kappa_1<\kappa_a<\kappa_b<\kappa_2$ of
Theorem~\ref{thm:classification}. It is analytic at $\kappa_c$, and
\begin{equation}\label{eq:critical}
 \rh(\kappa)=1+\frac{\kappa-\kappa_c}{4\kappa_c}+O\bigl((\kappa-\kappa_c)^2\bigr).
\end{equation}
As $q\downarrow0$,
\[
 \rh=(1+\sqrt2)\,q-\Bigl(2+\frac{3\sqrt2}2\Bigr)q^2+O(q^3),\quad
 \rh-\rs=\frac{10+7\sqrt2}{4}\,q^3+O(q^4),
\]
and as $\kappa\to\infty$, $\rh(\kappa)\to\rho_*=1.71298\ldots$, the positive
root of $\rho^4-2\rho^3+2\rho^2-2\rho-1$.
\end{theorem}
Regularity is proved in Appendix~\ref{app:classification-regularity}, and the asymptotic expansions in Appendix~\ref{app:finite-expansions}.

Two-frequency spectra attain a strictly larger rate than the spiral on two
intervals of condition numbers. The rate equals
$\rs$ exactly on
$[\kappa_W,\kappa_1]\cup\{\kappa_c\}\cup[\kappa_2,\infty)$ and exceeds $q$
everywhere. Figure~\ref{fig:rates} displays the rate and its regimes.
\begin{figure}[tbp]
\centering
\includegraphics[width=\linewidth]{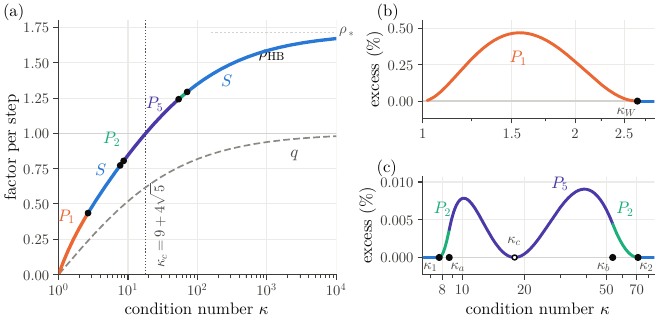}
\caption{Exact Polyak rate and optimal regimes. (a) $\rh$, the quadratic
factor $q$, the threshold $\kappa_c$ and the limit $\rho_*$; dots mark the
five switches. (b), (c) $100(\rh/\rs-1)$ on its two positive windows,
each at its own scale. At $\kappa_c$ the optimal spectrum is the three-cycle.}
\label{fig:rates}
\end{figure}

\subsection{Spectral characterization}\label{sec:spectral}
The reduction applies to every quadratically stable tuning
\eqref{eq:quad}. Translate the minimizer to zero and normalize the function by
\begin{equation}\label{eq:psi}
 \psi(e)=\frac2{L-\mu}\Bigl(\frac{L+\mu}4\norm{e}^2-f(x^\star+e)+f(x^\star)\Bigr).
\end{equation}
Then $\psi(0)=0$, $\nabla\psi(0)=0$ and $\nabla\psi$ is $1$-Lipschitz;
these properties define $\cS$. For centered functions with $f(0)=0$ this is a bijection, and with
$e_k=x_k-x^\star$ and $\sigma_k=\nabla\psi(e_k)$ the iteration becomes
\begin{equation}\label{eq:normrec}
 e_{k+1}=Ae_k+B\sigma_k-\beta e_{k-1},\quad
 A=1+\beta-\frac{\alpha(L+\mu)}2,\quad B=\frac{\alpha(L-\mu)}2 .
\end{equation}
Identify the plane with $\C$. A wave $e_k=z^k$, $z=e^{t+i\omega}$,
rotates by $\omega$ and scales by $e^t$ per step; $t$ is its growth exponent.
The recurrence requires $\sigma_k=m(t,\omega)z^k$, where
\begin{equation}\label{eq:m}
 m=m(t,\omega)=\frac{z+\beta/z-A}{B} .
\end{equation}
For an orthogonal sum of waves with common exponent $t$, a probability
measure $\nu$ on $[0,\pi]$ records their squared initial amplitudes, normalized
to sum to one. Each atom is a mass at one frequency and represents one
wave. Write bars for averages with respect to $\nu$ and put
$g=1-|m|^2$, $s=\Ima m$ and
\begin{equation}\label{eq:abm}
 a_m=g\,(\ch mt-\cos m\omega),\quad b_m=2s\sin m\omega\quad(m\ge1).
\end{equation}
Here the index $m\ge1$ denotes a time lag. The interpolation conditions
of Taylor, Hendrickx and Glineur~\cite{THG}, extended to arbitrary subsets of a Hilbert space by
Azagra and Mudarra~\cite{AM}, give a linear feasibility problem for this
spectrum. The scalar $C$ prescribes its function values; the precise
interpolation statement is Theorem~\ref{thm:interpolation}.
\begin{definition}[Feasible spectrum]\label{def:feasible}
A pair $(\nu,C)$ is feasible at the exponent $t$ if
\begin{equation}\label{eq:feasible}
 \begin{gathered}
 |C|\le\bar g,\\
 \ell_m:=\bar a_m-\bar b_m+C\sh mt\ge0,\quad
 \upsilon_m:=\bar a_m+\bar b_m-C\sh mt\ge0\quad(m\ge1).
 \end{gathered}
\end{equation}
The numbers $\ell_m$ and $\upsilon_m$ are the lower and upper lag-$m$ slacks.
\end{definition}
The first inequality enforces interpolation with the minimizer; a lag
constraint is active when its slack is zero. A finite nonnegative combination
of these inequalities gives a strict dual certificate of infeasibility
when its $C$ coefficient cancels and its remaining function of frequency,
or symbol, is uniformly negative:
\begin{equation}\label{eq:strictcert}
 \begin{gathered}
 \Phi_t(\omega)=\sum_m\bigl\{\alpha_m(a_m-b_m)+\beta_m(a_m+b_m)\bigr\}+(\gamma_++\gamma_-)g\le-\varepsilon
 \quad(0\le\omega\le\pi),\\
 \sum_m(\alpha_m-\beta_m)\sh mt+\gamma_+-\gamma_-=0 .
 \end{gathered}
\end{equation}
Here $\alpha_m,\beta_m,\gamma_\pm\ge0$ have finite support and
$\varepsilon>0$. Such a certificate gives an upper bound valid for every
trajectory. Under the hypotheses below, a finitely supported feasible
spectrum gives a lower bound realized by one trajectory at every horizon.
\begin{theorem}[Spectral characterization]\label{thm:spectralgeneral}
Assume $0<\alpha L<2(1+\beta)$.
\begin{enumerate}[label=\textup{(\alph*)}]
\item The conditions \eqref{eq:feasible} are infeasible at $t$ if and only if a
strict certificate at $t$ exists.
\item If strict certificates exist at every $t\ge t_0$, then $\rho<e^{t_0}$.
\item If $(\nu,C)$ is feasible at $t$, $\nu$ has $r$ atoms, and $|C|<\bar g$ in
case $t>0$, then there are $f\in\cF(\R^{2r+1})$ and $x_0\ne x^\star$ with
$\norm{x_N-x^\star}\ge c_0\,e^{Nt}\norm{x_0-x^\star}$ for all $N\ge0$, where
$c_0>0$. In particular $\rho\ge e^t$.
\end{enumerate}
\end{theorem}
\begin{proof}
For \textup{(a)}, integrating the symbol of a strict certificate against a
feasible measure gives $0\le\int\Phi_t\,d\nu\le-\varepsilon$, a
contradiction. Conversely, the minimizer inequalities restrict $C$ to
$[-1,1]$, and the probability measures on $[0,\pi]$ form a weakly compact
set. Thus infeasibility implies that a finite subsystem, including both
minimizer inequalities,
$\int\phi_i\,d\nu+\varsigma_iC\ge0$, $i\in\mathcal I$, is infeasible
(Lemma~\ref{lem:compact}). The closed convex set
\[
 \operatorname{conv}\{(\phi_i(\omega))_i:0\le\omega\le\pi\}
 +\R\varsigma-\R_+^{\mathcal I}
\]
excludes zero; it is closed because it is the sum of a compact set and a
polyhedral cone. Strict separation gives nonnegative multipliers
$\lambda_i$, with $\lambda\cdot\varsigma=0$ and
$\sum_i\lambda_i\phi_i(\omega)\le-\varepsilon$ uniformly in $\omega$.
These are the coefficients of \eqref{eq:strictcert}.

For \textup{(b)}, let $\rho_\Gamma$ be the asymptotic growth factor of
arbitrary finite trajectory segments, defined in \eqref{eq:Gamma}.
Lemma~\ref{lem:envelope} gives $\rho\le\rho_\Gamma$.
Lemma~\ref{lem:energy} converts a strict certificate into an energy
estimate with only endpoint terms; Lemma~\ref{lem:plateau} consequently
excludes a strict certificate at $\log\rho_\Gamma$. If
$\rho_\Gamma\ge e^{t_0}$, the hypothesis supplies precisely such a
certificate. Hence $\rho\le\rho_\Gamma<e^{t_0}$.

For \textup{(c)}, Lemma~\ref{lem:spectralinterp} gives interpolable wave
data in $2r$ dimensions. Lemma~\ref{lem:startup} appends one decaying
quadratic coordinate and changes the first gradient so that
$x_{-1}=x_0$. Its hypotheses are exactly feasibility and, when $t>0$,
$|C|<\bar g$. The resulting single interpolant satisfies
$\norm{x_N-x^\star}\ge c_0 e^{Nt}\norm{x_0-x^\star}$ for every $N$.
Taking $N$th roots gives $\rho\ge e^t$.
\end{proof}

For Polyak's tuning, $A=0$, $B=2q$, $\beta=q^2$ and
$m=\ch(v+i\omega)$, with
\begin{equation}\label{eq:polyakwave}
 v=t-\log q,\quad u=\sh v,\quad c=\cos\omega,\quad G=\ch t.
\end{equation}
Consequently $g=1-c^2-u^2$, $s=u\sin\omega$, and
\begin{equation}\label{eq:abg}
 a_m(c)=g(c)\bigl(T_m(G)-T_m(c)\bigr),\quad b_m(c)=2u(1-c^2)U_{m-1}(c),
\end{equation}
where $T_m,U_m$ are the Chebyshev polynomials. On changing variables from
$\omega$ to $c=\cos\omega$, we use $\nu$ for the resulting measure on $[-1,1]$.
The relative exponent $v=t-\log q$ measures growth
against the quadratic rate $q$. In the auxiliary program $v$ and $t$ vary
independently, with $u=\sh v$; the actual tuning follows $v=t-\log q$. Set
\begin{equation}\label{eq:VT}
 V(t)=\max\{v\ge0:\ \text{\eqref{eq:feasible} is feasible at }(v,t)\},\quad
 T(q)=\max\{t:\ |t-\log q|\le V(t)\}.
\end{equation}
\begin{theorem}[Spectral rate for Polyak's tuning]\label{thm:spectral}
For every $q\in(0,1)$, $\rh(\kappa)=e^{T(q)}$. The function $V$ is even,
bounded by $\arsinh1$, and the set of $v$ feasible at a given $t$ is
$[-V(t),V(t)]$. At every optimal point with $T(q)\ne0$, every optimal measure
is finitely supported in $(-1,1)$ and satisfies $|C|<\bar g$, and only finitely
many lag constraints are active.
\end{theorem}
The proof is given in Appendix~\ref{app:spectral-polyak}.

The feasible region is $|v|\le V(t)$. The parameters of the method satisfy
$v=t-\log q$; the largest feasible intersection gives
$T=\log\rh$ (Figure~\ref{fig:boundary}). A limiting dual symbol is analytic
for $t\ne0$. Its finitely many interior zeros carry every optimal measure,
and strict slack at the minimizer makes every sufficiently large lag
inactive. These properties reduce the infinitely many interpolation
inequalities to a finite algebraic classification.

\subsection{Complete spectral classification}\label{sec:boundary}
\subsubsection{Single-frequency construction and convergence threshold}\label{sec:spiral}
The three contacts $\ell_1=\upsilon_2=\upsilon_3=0$ determine an explicit
single-frequency spectrum.
\begin{proposition}[The special spiral]\label{prop:spiral}
For every $t\ne0$, the data
\begin{equation}\label{eq:spiraldata}
 u=1-\frac1{2G},\quad c=-\frac1{2G},\quad \nu=\delta_c,\quad C=g\tanh t
\end{equation}
are feasible at $(\arsinh u,t)$, and $\ell_1=\upsilon_2=\upsilon_3=0$. On this
spectrum
\begin{equation}\label{eq:En}
 \frac{\ell_m}g=\frac{E_{m+1}(t)}G,\quad \frac{\upsilon_m}g=\frac{E_{m-1}(t)}G,\quad
 E_n(t)=\ch nt+2G^2\cos n\omega,
\end{equation}
and for every $n\ge3$, $G^{n-2}E_n(t)$ is a polynomial in $G^2-1$ with
nonnegative coefficients, positive for $t\ne0$. The relation
$u=\sh(t-\log q)$ on this curve is \eqref{eq:spiralpoly} with $\rho=e^t$.
\end{proposition}
The proof is given in Appendix~\ref{app:classification-spiral}.

At $t=0$ this spectrum is the three-cycle $\omega=2\pi/3$, $u=1/2$.
Combining the lag-one and minimizer inequalities gives the sharp strip
$|v|\le |t|+\log\varphi$, strictly for $t\ne0$
(Proposition~\ref{prop:strip}). Intersecting it with $v=t-\log q$ proves
$\rh\le\sqrt{\varphi q}<1$ for $q<\varphi^{-1}$; the spiral gives the
matching transition to a nondecaying cycle and then exponential divergence.

\subsubsection{Algebraic description of the boundary}
Each family pairs a feasible measure with a nonpositive dual symbol that
vanishes on its support, as required by complementary slackness. The families $P_1$ and $S$ have one frequency;
$P_2$ and $P_5$ have two selected frequency points, with zero masses omitted at endpoints. On the open
noncritical regimes their active upper lags are respectively $\{3\}$,
$\{2,3\}$, $\{2,3\}$ and $\{2,3,5\}$, together with the lower lag one.
The corresponding polynomial contact systems are
\eqref{eq:p1system}, \eqref{eq:p2system} and \eqref{eq:p5system}; all their
coefficients and branch selections are specified in
Appendix~\ref{app:classification}.
For the two-frequency families let $x<y$ be the cosine coordinates
of the support points, with masses $1-p$ and $p$, and put $d=C\sh t$.
\begin{theorem}[Complete spectral classification]\label{thm:classification}
Let $G_W=(1+\sqrt3)/2$, let $G_1\in(1.03329,1.03330)$ be the root of
$32G^3-36G^2+4G-1$, and let $G_a\in(1.02358,1.02359)$ be the root of
\eqref{eq:Ha}.
Then $V(t)=\arsinh u(\ch t)$, where
\begin{equation}\label{eq:classification}
 u(G)=\begin{cases}
 u_5(G),&1\le G\le G_a,\\
 u_2(G),&G_a\le G\le G_1,\\
 1-(2G)^{-1},&G_1\le G\le G_W,\\
 u_1(G),&G_W\le G<\infty,
 \end{cases}
\end{equation}
and $u_1$, $u_2$, $u_5$ are the values of $u$ on the families $P_1$, $P_2$,
$P_5$. On $P_2$, $G_1=G(\zeta_1)$ and $G_a=G(\zeta_a)$, where $\zeta_1$ is the
real root of $\zeta^3-\zeta^2-2$ and $\zeta_a\in(1.6989,1.699)$ is a root of
the polynomial $N_{12}$ of degree twelve in Lemma~\ref{lem:p2ends}; at $G=1$ the central family reduces
to the three-cycle, $u=\frac12$, $x=-\frac12$, $p=d=0$. The branches are selected by the parametrization of $P_2$ and the unique
analytic continuations of $P_1$ and $P_5$ from the endpoint data of
Appendix~\ref{app:classification}; their values agree at the endpoints. The function
$t\mapsto t-V(t)$ is continuous and strictly increasing, so
\begin{equation}\label{eq:scalarprescription}
 t-\arsinh u(\ch t)=\log q
\end{equation}
has exactly one solution $t=T(q)$, and $\rh=e^{T(q)}$. At most two atoms are
needed.
\end{theorem}
\begin{proof}
For each $G\ge1$ we construct a feasible spectrum and a nonpositive
dual symbol $\Phi(c)=-(G-c)Q(c)$ whose zeros contain its support.
The normalized upper lags and the lower lag one cancel $C\sh t$;
Lemma~\ref{lem:dualfactor} shows that increasing $u\ge0$ decreases the
symbol by at least $(G-c)$ times the increment of $u^2$. Thus such a
primal--dual pair determines the maximal feasible $u$. For $G>1$ the
decrease is uniformly strict. At $G=1$, adding equal positive
minimizer weights makes it strict also at $c=1$, as shown after
Lemma~\ref{lem:dualfactor}.

The factorization \eqref{eq:spiralcertificate} proves this pairing for
$G_1\le G\le G_W$. Proposition~\ref{prop:p1} gives the unique
continued branch for $G\ge G_W$; Theorem~\ref{thm:p2arc} gives the
rationally parametrized branch on $[G_a,G_1]$; and
Proposition~\ref{prop:p5arc} gives the central branch on $[1,G_a]$.
Each result proves feasibility for every lag, dual nonnegativity, and
the stated endpoint identities. In particular, the branch values agree
at $G_W,G_1,G_a$, and Theorem~\ref{thm:p5local} identifies the critical
spectrum at $G=1$. This proves \eqref{eq:classification}, including its
branch selections.

The same results give $|V'|<1$ on every open arc; on the spiral this
follows from $du/dG=1/(2G^2)$. Continuity at the junctions implies that
$t-V(t)$ is strictly increasing. It tends to $-\infty$ and $+\infty$
at the two ends of the real line because $V$ is bounded. Its unique
level $\log q$ is therefore the largest feasible exponent on
$v=t-\log q$, hence equals $T(q)$ by Theorem~\ref{thm:spectral}.
Every constructed measure has at most two atoms.
\end{proof}

For conversion to condition numbers put
$q_\pm(G)=(G\pm\sqrt{G^2-1})(\sqrt{1+u(G)^2}-u(G))$ and
$\kappa(q)=((1+q)/(1-q))^2$. Then
\[
 (\kappa_W,\kappa_1,\kappa_a)
 =\bigl(\kappa(q_-(G_W)),\kappa(q_-(G_1)),\kappa(q_-(G_a))\bigr),
 \quad
 (\kappa_b,\kappa_2)
 =\bigl(\kappa(q_+(G_a)),\kappa(q_+(G_1))\bigr).
\]
The symmetry $(t,C)\mapsto(-t,-C)$ produces the two copies of the central
families in Table~\ref{tab:regimes}.
\begin{table}[t]
\caption{Optimal regimes for Polyak's tuning. At switching points the adjacent
families agree; zero spectral weights are omitted. The critical three-cycle
has the additional active constraints described below.}
\label{tab:regimes}
\centering
\begin{tabular}{lll}
\toprule
Condition number & Family & Zero raw lag slacks\\
\midrule
$(1,\kappa_W)$ & $P_1$ & $\ell_1,\upsilon_3$\\
$[\kappa_W,\kappa_1]$ & $S$ & $\ell_1,\upsilon_2,\upsilon_3$\\
$(\kappa_1,\kappa_a)$ & $P_2$ & $\ell_1,\upsilon_2,\upsilon_3$\\
$[\kappa_a,\kappa_b]\setminus\{\kappa_c\}$ & $P_5$ & $\ell_1,\upsilon_2,\upsilon_3,\upsilon_5$\\
$\{\kappa_c\}$ & Three-cycle & $\ell_m\ (m\equiv0,1),\ \upsilon_m\ (m\equiv0,2)$\\
$(\kappa_b,\kappa_2)$ & $P_2$ & $\ell_1,\upsilon_2,\upsilon_3$\\
$[\kappa_2,\infty)$ & $S$ & $\ell_1,\upsilon_2,\upsilon_3$\\
\bottomrule
\end{tabular}
\par\smallskip
\small
In the critical row $m\ge1$ and congruences are modulo $3$; these are all
zero raw slacks, so the active set is infinite. Switching values, rounded
to five decimal places, are
\[
\kappa_W\approx2.65646,\quad \kappa_1\approx7.73195,\quad \kappa_a\approx8.64845,\quad
\kappa_b\approx53.64303,\quad \kappa_2\approx71.14732.
\]
\end{table}

\begin{figure}[tbp]
\centering
\includegraphics[width=\linewidth]{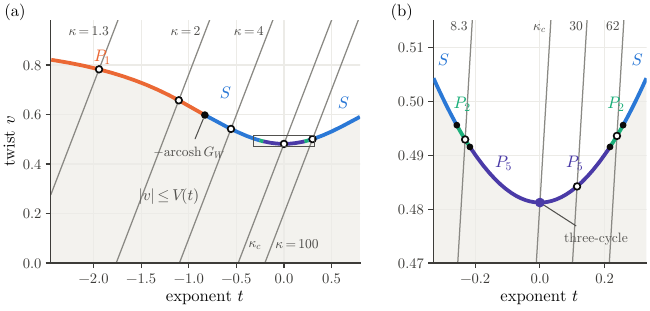}
\caption{The upper spectral boundary $v=V(t)$; the shaded region is
$0\le v\le V(t)$. A line $v=t-\log q$ meets the boundary at the method's
exponent $T=\log\rh$. Panel (b) magnifies the three-cycle and the
switches $\pm\arcosh G_a$, $\pm\arcosh G_1$.}
\label{fig:boundary}
\end{figure}

\subsection{Finite-horizon realization}\label{sec:realize}
\subsubsection{Initialization from rest}\label{sec:startup}
An atom of positive mass supplies a plane rotating by $\omega$ and scaling
by $e^t$ at each step. To initialize these waves with $x_{-1}=x_0$, append
a scalar quadratic coordinate $Yp_k$ and adjust the first gradient. For
Polyak's tuning one may take
$p_k=(-q^2)^{\lceil k/2\rceil}$. The extra coordinate adds
$\tfrac14Y^2(p_i-p_j)^2$ to the interpolation margin. A sufficiently large
$Y$ absorbs the first-step defect uniformly over all later times; the
strict minimizer slack controls the growing case. Hence $r$ frequencies
give one example in dimension $2r+1$, with one start valid for all horizons.
For the special spiral the choice $Y=2$ gives the uniform factor $1/\sqrt5$.

The zeros of the dual polynomial $Q$ select the stationary frequencies.
At $G=G_1$ or
$G=1$ one selected point has zero mass. At $\kappa_c$ the stationary
component is periodic with period three, and the initialized trajectory
approaches this orbit as the extra coordinate decays.
Figure~\ref{fig:spectra} illustrates the selected stationary spectra.

\subsubsection{Uniform bound at the exact rate}
The spectral characterization first yields
$W_N\le C_\delta e^{N(T+\delta)}$. The boundary certificate has a strict
negativity margin of order $\delta$ at exponent $T+\delta$. Real linear
combinations of finitely many consecutive positions turn it into a local
window functional; these functionals remain bounded as $\delta\downarrow0$.
Their limit is nonnegative on every admissible window, grows by at most
$e^{2T}$ per step, and is strictly positive on every stationary optimal
spectrum (Lemma~\ref{lem:exactprice}). Separation and the energy estimate
give a strictly smaller growth factor where this functional vanishes
(Lemma~\ref{lem:kernel}). Compactness extends this contraction to a
neighborhood of that set (Lemma~\ref{lem:nearkernel}). Outside that
neighborhood the limiting functional controls the window norm, which
gives $W_N\le C_\kappa e^{NT}$.
Together with the two-frequency realization this proves
Theorem~\ref{thm:rate}. Appendix~\ref{app:finite} gives this argument and
the asymptotic expansions of Theorem~\ref{thm:regularity}.

\FloatBarrier
\section{General tunings}\label{sec:part2}
We now allow the step size and momentum to vary within the region of
quadratic stability. The spectral characterization gives an exact
criterion for uniform convergence and describes the boundary of its
complement. Angle inequalities, Farey geometry and continuation of
finite spectra determine the global organization and local structure
of this boundary. We first give explicit phase tests, then the exact
measure criterion of Theorem~\ref{thm:exactthr}, and finally the global
threshold and its local plateau structure.

\subsection{Spectral characterization for general tunings}\label{sec:general}
In this part $q$ denotes the parameter triple
$(\kappa,y,\beta)$.
Here $\kappa>1$, $0\le\beta<1$ and $0<y<2(1+\beta)$;
$x=y/\kappa=\alpha\mu$ and $y=\alpha L$.
This is precisely the range \eqref{eq:quad} of stability on quadratics.
Spectra in this part are measures in the frequency variable
$\omega\in[0,\pi]$.
Write $\rho(q)$ for \eqref{eq:WN} and call
$\mathcal D=\{q:\rho(q)\ge1\}$ the divergence set, the complement of
uniform exponential convergence in \eqref{eq:WN}. Its boundary can
support nondecaying periodic behavior; exponential growth requires $\rho>1$.

The endpoint characteristic polynomials are
\begin{equation}\label{eq:ab}
 a(z)=z^2-(1+\beta-x)z+\beta,\qquad
 b(z)=z^2-(1+\beta-y)z+\beta.
\end{equation}
Their largest root moduli $\rho_a,\rho_b$ give
$\rhoq=\max(\rho_a,\rho_b)\le\rho$.
Between the damping curves
\begin{equation}\label{eq:damping}
 x=(1-\sqrt\beta)^2,\qquad y=(1+\sqrt\beta)^2,
\end{equation}
every curvature is underdamped or critically damped and $\rhoq=\sqrt\beta$.
Outside this region an endpoint has a dominant real root.
Polyak's tuning is the intersection of the curves.

The link with the spectral program is the identity
\begin{equation}\label{eq:factor}
 g-2is=(1-m)(1+\ov m)
       =-\frac{a(z)\ov{b(z)}}{B^2|z|^2},\qquad z=e^{t+i\omega}.
\end{equation}
At $t=0$, $s=(1-\beta)\sin\omega/B>0$ on $(0,\pi)$.
The argument of $g-2is$ compares the phases of the endpoint
characteristic polynomials and determines the threshold geometry.
\begin{definition}[Angle function]\label{def:angle}
For $\omega\in(0,\pi]$ let $\Theta(\omega)\in(0,\pi]$ be the angle between
$b(e^{i\omega})$ and $-a(e^{i\omega})$, so that
$g-2is=|g-2is|\,e^{-i\Theta}$ at $t=0$.
\end{definition}
The curve is determined by a concave polynomial $\Xi$ of degree at most two:
$\cot\Theta(\omega)=\Xi(\cos\omega)/\sin\omega$.
Its two endpoint values are negative; the coefficients and endpoint
formulas are in Lemma~\ref{lem:angle}.

Extend $\Theta$ continuously by $\Theta(0)=\pi$.
The spectral upper bound becomes particularly useful because one exponent
already determines a strict rate bound. Let $r_0<1$ be the largest root modulus
of $z^2-Az+\beta$.
\begin{theorem}[One exponent suffices]\label{thm:oneexp}
Assume \eqref{eq:quad}. If a strict certificate exists at some $t$ with
$e^t>r_0$, then $\rho<e^t$. In particular, infeasibility of
\eqref{eq:feasible} at $t=0$ implies $\rho<1$.
\end{theorem}

\begin{proof}
Put $\varrho=e^t$ and rescale a trajectory of \eqref{eq:normrec} by
$\tilde e_k=\varrho^{-k}e_k$. Lemma~\ref{lem:iqc} converts the strict
certificate into a quadratic constraint satisfied by the scaled gradient
of every $\tau\psi$, $0\le\tau\le1$. The corresponding linear filter has
transfer function
\[
 G_t(z)=\frac{B\varrho z}{\varrho^2z^2-A\varrho z+\beta}.
\]
Its poles lie inside the unit circle because $\varrho>r_0$, and it
vanishes at infinity. Moreover,
$G_t(e^{i\omega})=1/m(t,\omega)$ is nonzero on the unit circle, and
\[
 [G_t;1]^*\Pi[G_t;1]=|G_t|^2\Phi_t
 \le-\varepsilon\min_{|z|=1}|G_t(z)|^2<0.
\]
The scaled gradient acts pointwise in time, is $1$-Lipschitz and fixes zero;
scaling $\psi$ by $\tau\in[0,1]$ preserves $\cS$. Thus every hypothesis of
Lemma~\ref{lem:homotopy} is satisfied. Its energy bound is uniform in
the function, dimension and initial pair:
\[
 \sum_{k\ge0}\varrho^{-2k}\|e_k\|^2
 \le C_0(\|e_0\|^2+\|e_{-1}\|^2).
\]
Lemma~\ref{lem:shift} therefore yields
$\rho\le\varrho\theta^{1/4}<\varrho$ for one $\theta<1$.
At $t=0$ the filter is stable by \eqref{eq:quad}, and the spectral
alternative supplies a strict certificate whenever
\eqref{eq:feasible} is infeasible.
\end{proof}

\subsection{Convergence and periodic orbits}\label{sec:regions}
Two elementary sufficient conditions determine large subsets of the
parameter region: a lag-one multiplier proves convergence, while a
single-frequency feasible measure gives a periodic obstruction.
\begin{theorem}[Phase locking]\label{thm:A}
Assume \eqref{eq:quad}.
\begin{enumerate}[label=\textup{(\roman*)}]
\item If $\Theta(\omega)>\omega/2$ for every $\omega\in(0,\pi]$, then the
method converges.
\item If $\Theta(2\pi j/k)\le\pi/k$ for coprime integers $1\le j<k/2$, then
there are $f\in\cF(\R^3)$ and $x_0$ with
$\norm{x_N-x^\star}\ge c\norm{x_0-x^\star}>0$ for all $N$; the trajectory
converges to a $k$-periodic orbit that rotates by $2\pi j/k$ in a plane. If
the inequality is strict, then $\rho>1$.
\item Let $\Omega_1$ be the set of tunings in \textup{(i)}. On its closure
relative to \eqref{eq:quad}, $\rho\le1$. At every boundary point of
$\Omega_1$ at which $\min_\omega(\Theta(\omega)-\omega/2)=0$ is attained at
$\omega=2\pi/K$, the hypothesis of \textup{(ii)} holds with $j=1$ and $k=K$;
there $\rho=1$, and the point lies on the boundary of the convergence region.
\end{enumerate}
\end{theorem}
The proof is given in Appendix~\ref{app:general-multipliers}.

For fixed $j/k$, the tunings satisfying the periodic criterion form
the $j/k$-tongue; the $K$-tongue corresponds to $j/k=1/K$.
The two conditions compare the angle curve with the line $\omega/2$ and
with rational phase values (Figure~\ref{fig:thomae}). At a phase-locked
boundary point, the tangency frequency equals $2\pi/K$.
The cycle criterion agrees with the roots-of-unity regions
of~\cite{GTD,FHDA}; its algebraic form is given in
Corollary~\ref{cor:C}.

\begin{figure}[tbp]
\centering
\includegraphics[width=\linewidth]{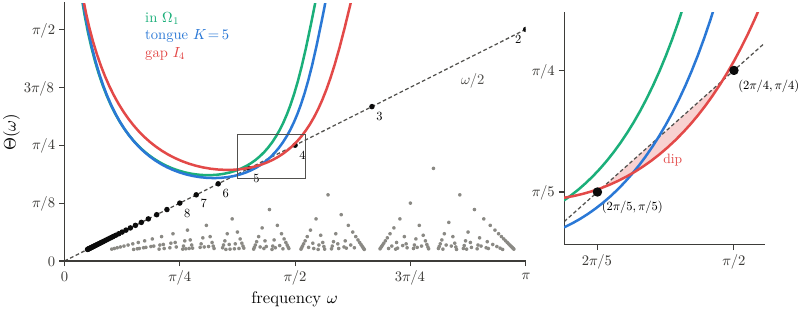}
\caption{The rational points $(2\pi j/k,\pi/k)$ and three angle curves at
$\kappa=100$. The points $j=1$ lie on $\omega/2$. A curve above the line
gives convergence; meeting a rational point gives a periodic obstruction.
The magnified dip between $2\pi/5$ and $2\pi/4$ contains the two-harmonic
example of Proposition~\ref{prop:13}.}
\label{fig:thomae}
\end{figure}

For $c=\cos\omega$, let
$H_r(c)=\Rea[(1-z)a(z)\ov{b(z)}]$, $z=re^{i\omega}$.
It has degree at most three in $c$ (Lemma~\ref{lem:cubic}).
Inside the lag-one region $\Omega_1$, it bounds the rate and
identifies a region $\cQ$ of exact agreement with the quadratic worst case.
\begin{theorem}[Rates]\label{thm:B}
Assume \eqref{eq:quad}.
\begin{enumerate}[label=\textup{(\alph*)}]
\item If $\Theta(\omega)>\omega/2$ on $(0,\pi]$, then the set
$\{r\in(\rho_b,1):H_r(c)\ge0\ \text{for all }c\in[-1,1]\}$ is a nonempty interval
with right endpoint $1$. With $r_1$ its infimum, $\rhoq\le\rho\le r_1<1$.
\item Suppose $\rhoq>\sqrt\beta$ and $\Theta(\omega)>\omega/2$ on $(0,\pi]$. Then
$r_1=\rhoq$ if and only if $H_{\rhoq}(c)\ge0$ for all $c\in[-1,1]$. On the set
$\cQ$ of such tunings, $\rho=\rhoq$.
\item If $\rhoq=\sqrt\beta$, then $\rho>\rhoq$.
\end{enumerate}
\end{theorem}
The proof is given in Appendix~\ref{app:general-multipliers}.

An interior underdamped quadratic mode has interpolation slack and can be
perturbed to grow faster; this explains the strict gap in part~(c).
An overdamped endpoint mode is tight, and the sign of $H_{\rhoq}$ decides
whether the lag-one bound is sharp. The signs reduce to quadratics
(Corollary~\ref{cor:C}); Figure~\ref{fig:phase} shows the resulting regions.
\begin{corollary}[Polyak's tuning and gradient descent]\label{cor:D}
\begin{enumerate}[label=\textup{(\alph*)}]
\item For Polyak's tuning, the method lies in $\Omega_1$ if and only if
$\kappa<\kappa_c$; at $\kappa_c$ the tuning is the phase-locked point $K=3$
of $\partial\Omega_1$, and for $\kappa>\kappa_c$ it lies in the three-tongue.
\item For $\beta=0$ and $0<\alpha L<2$, both $\Omega_1$ and $\cQ$ are the whole
range, and $\rho=\max(|1-\alpha\mu|,|1-\alpha L|)$.
\end{enumerate}
\end{corollary}
The proof is given in Appendix~\ref{app:general-multipliers}.

\begin{figure}[tbp]
\centering
\includegraphics[width=\linewidth]{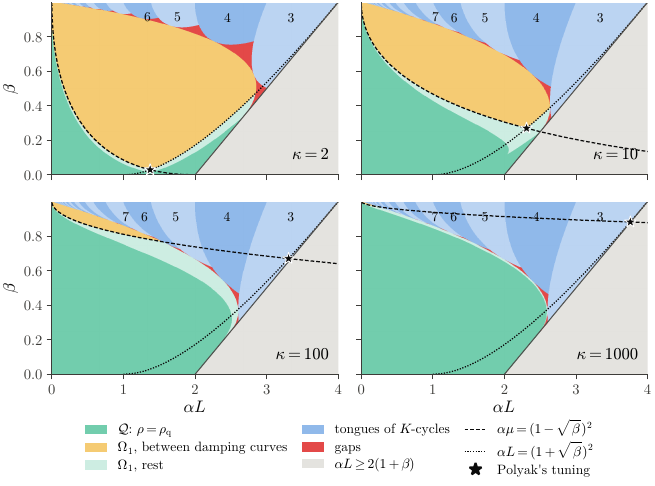}
\caption{Phase diagrams for $\kappa=2,10,100,1000$.
Green: $\cQ$, where $\rho=\rhoq$. Yellow and pale green: the remaining
lag-one convergence region $\Omega_1$. Blue: integer-cycle tongues.
Violet: additional rational-cycle tongues. Red: the remaining displayed
gap regions. Grey: quadratic instability.
The damping curves are dashed and dotted; the star is Polyak's tuning.}
\label{fig:phase}
\end{figure}

The complement of $\Omega_1$ is the union of the fractional-cycle regions
$\Theta(\omega)\le\omega/2$, the continuation of the integer tongues to
real cycle lengths~\cite{GTD25}. The gaps between consecutive integer
tongues contain both convergent tunings and multi-frequency periodic
obstructions. The full spectral program distinguishes these behaviors.

\subsection{The threshold function and the divergence set}\label{sec:levels}
At exponent zero all lag inequalities can be written using the angle alone:
\begin{equation}\label{eq:Lambda}
 \Lambda_n(\omega)=\cos\Theta-\cos(n\omega-\Theta)
 =2\sin\frac{n\omega}2\sin\Bigl(\frac{n\omega}2-\Theta\Bigr),
 \qquad n\in\Z\setminus\{0\}.
\end{equation}
The magnitude $|n|$ is the temporal separation in an interpolation
constraint; positive and negative indices denote the lower and upper
inequalities, respectively.
A curve is a continuous map $\Theta:[0,\pi]\to(0,\pi]$ with
$\Theta(0)=\Theta(\pi)=\pi$.
\begin{definition}[Full program and certificate symbols]\label{def:FP}
Let $\Theta$ be a curve. The full program $\FP(\Theta)$ asks for a nonzero
positive Borel measure $w$ of finite total mass on $[0,\pi]$ with
\[
 \int\Lambda_n\,dw\ge0\quad(n\in\Z\setminus\{0\}),\qquad \int\cos\Theta\,dw\ge0 .
\]
The last condition is the minimizer constraint.
For $\mathcal N\subset\Z\setminus\{0\}$, the program $\FP[\mathcal N](\Theta)$ keeps
the lag constraints with $n\in\mathcal N$. A certificate symbol with lags in
$\mathcal N$ is a function $S=\sum_{n\in N}\lambda_n\Lambda_n+\gamma\cos\Theta$ with
$N\subset\mathcal N$ finite, $\lambda_n\ge0$ and $\gamma\ge0$; it is strict
if $S\le-\varepsilon$ on $[0,\pi]$ for some $\varepsilon>0$.
\end{definition}
This is a linear feasibility problem in $w$. The symbol combines its
constraint functions with nonnegative dual weights: its integral
against a feasible measure is nonnegative, so strict negativity
certifies infeasibility.
\begin{theorem}[Exact threshold criterion]\label{thm:exactthr}
Let $q$ be a tuning.
\begin{enumerate}[label=\textup{(\alph*)}]
\item The conditions \eqref{eq:feasible} are feasible at $t=0$ if and only if
$\FP(\Theta_q)$ is feasible; up to normalization, $\nu\mapsto|g-2is|\,\nu$ is a
bijection between the spectra feasible at $t=0$ with $C=0$ and the solutions of
$\FP(\Theta_q)$. Feasibility at $t=0$ depends on the tuning only through the
curve $\Theta_q$, that is, through the three coefficients of $\Xi$.
\item If $\FP(\Theta_q)$ is feasible, then $W_N\ge c$ for all $N$, with $c>0$
independent of $N$; in particular $\rho(q)\ge1$.
\item $\rho(q)<1$ if and only if $\FP(\Theta_q)$ is infeasible, if and only if
a strict certificate symbol exists. Consequently
$\mathcal D=\{q:\ \FP(\Theta_q)\ \text{is feasible}\}$.
\end{enumerate}
\end{theorem}

\begin{proof}
At exponent zero, feasibility with any admissible $C$ implies
$\bar g\ge|C|\ge0$, so $C=0$ may be used. With
$d(\omega)=|g-2is|=|a(e^{i\omega})b(e^{i\omega})|/B^2$, the lag and
minimizer functions satisfy
\[
 g(1-\cos n\omega)-2s\sin n\omega=d\Lambda_n,
 \qquad g=d\cos\Theta_q.
\]
The function $d$ is continuous and strictly positive on $[0,\pi]$ by
quadratic stability. Consequently $w=d\nu$ and $\nu=d^{-1}w$ preserve
finiteness and positivity and give exactly the correspondence in (a).
Lemma~\ref{lem:alt} gives the equivalence between infeasibility and a
strict finite symbol. Such a symbol is a strict spectral certificate
at $t=0$, hence implies $\rho<1$ by Theorem~\ref{thm:oneexp}.
Conversely, a feasible measure admits the start-up realization of
Appendix~\ref{app:levels-realization}: multiplication by $e^{i\omega}$ in
$L^2(\nu)$ gives a stationary trajectory, and one decaying coordinate
enforces the initial condition with a fixed constant $Y$.
Restricting its interpolation data to the finite span used before a
horizon $N$ gives a finite-dimensional objective with
$\|e_N\|\ge1$ and $\|e_0\|\le\sqrt{1+Y^2}$.
Thus $W_N\ge(1+Y^2)^{-1/2}$ for every $N$, proving (b) and the
remaining implication in (c).
\end{proof}

\subsubsection{Tangencies and gaps between periodicity regions}
The cotangent $h_q=\Xi(\cos\omega)/\sin\omega$ is strictly concave.
Consequently the closed dip, the set of frequencies where the curve
meets or falls below the lag-one convergence line,
$\bar\Delta=\{\omega\in(0,\pi]:\Theta_q(\omega)\le\omega/2\}$
is a compact interval in $(0,\pi)$, possibly empty or a point
(Theorem~\ref{thm:convex} and Corollary~\ref{cor:dip}).
Every feasible measure charges this interval (Lemma~\ref{lem:dip}).
\begin{theorem}[The boundary of $\Omega_1$]\label{thm:tangency}
Let $q$ satisfy $\Theta_q\ge\omega/2$ on $(0,\pi]$ with equality at some $\omega_*$. If
$\omega_*\ne2\pi/K'$ for every $K'\ge3$, then a strict certificate symbol with the lags
$1$ and $-K$ exists, where $2\pi/(K+1)<\omega_*<2\pi/K$, and $\rho(q)<1$. Consequently
$\partial\Omega_1\cap\mathcal D$ consists exactly of the phase-locked points of
Theorem~\ref{thm:A}\textup{(iii)}, where $\rho=1$.
\end{theorem}
The proof is given in Appendix~\ref{sec:convexdip}.
\begin{definition}[Gaps]\label{def:levels}
A tuning lies in the gap $I_K$
of level one, $K\ge2$, if the dip $\{\Theta<\omega/2\}$ is nonempty with rotation numbers
in $(1/(K+1),1/K)$, while $\Theta(2\pi/K)>\pi/K$ and $\Theta(2\pi/(K+1))>\pi/(K+1)$.
\end{definition}
\begin{theorem}[Level zero]\label{thm:levelzero}
Every tuning $q$ falls into exactly one of the following cases.
\begin{enumerate}[label=\textup{(\Roman*)}]
\item $\bar\Delta=\emptyset$: $q\in\Omega_1$, and $\rho(q)<1$.
\item $2\pi/K'\in\bar\Delta$ for some $K'\ge3$: $\Theta(2\pi/K')\le\pi/K'$, the tuning lies in the
closed $K'$-tongue, $\delta_{2\pi/K'}$ solves $\FP(\Theta_q)$, and $\rho(q)\ge1$.
\item $\emptyset\ne\bar\Delta\subset(2\pi/(K+1),2\pi/K)$ for a unique $K\ge2$. Either
$\bar\Delta$ is a point, $q$ is a point of $\partial\Omega_1$ that is not phase-locked, and
$\rho(q)<1$; or $\bar\Delta$ has interior, and $q\in I_K$.
\end{enumerate}
\end{theorem}
The proof is given in Appendix~\ref{sec:convexdip}.

Thus only the open gaps remain after lag-one convergence and integer
cycles have been classified. For $c_K=\cos(2\pi/K)$, each gap has the
coordinates
\[
 1+c-\Xi_q(c)=a(c-r_1)(c-r_2),\qquad
 c_K<r_1<r_2<c_{K+1},\qquad
 a>a_0(r_1,r_2):=\frac{2}{(1-r_1)(1-r_2)}.
\]
Proposition~\ref{prop:chart} proves that this parametrization is a
homeomorphism and, in particular, that $I_K$ is connected.

\subsubsection{Continuity and periodic approximation}
At exponent $t=0$, introduce a separate parameter $s$ by shifting
the cotangent vertically:
\begin{equation}\label{eq:shifted}
 \Theta_{q,s}=\operatorname{arccot}(h_q+s)\quad(0<\omega<\pi),
 \qquad \Theta_{q,s}(0)=\Theta_{q,s}(\pi)=\pi.
\end{equation}
Increasing $s$ lowers the curve and preserves feasibility.
In the monotone form of Proposition~\ref{prop:monoform}, this is the
identity
\[
 \ell_n[h+s]=\ell_n[h]+s(1-\cos n\omega),\qquad
 \ell_n[h]=h(1-\cos n\omega)-\sin n\omega.
\]
There is a finite attained threshold $\sigma(q)$:
$\FP(\Theta_{q,s})$ is feasible exactly for $s\ge\sigma(q)$,
and $\mathcal D=\{\sigma\le0\}$ (Proposition~\ref{prop:threshold}).

Feasible measures at this smallest shift are called threshold
extremals; extremality here refers to the attained shift.
For $L\ge3$, let
$\Omega_L=\{2\pi j/L:0<2\pi j/L\le\pi\}$.
The $L$-periodic program $\PP_L$ restricts the measure to
$\Omega_L$, whose waves all repeat after $L$ iterations; lag
constraints then depend only on $n\bmod L$.
The harmonic reduction for cycles with free initial states is given
in~\cite[Theorem~6.13]{GTD}; Proposition~\ref{prop:periodic} enforces
the initialization in~\eqref{eq:hb}.
Write $\mathcal T_L$ for its feasible tunings, $\sigma_L$ for its
shift threshold, and $\sigma_{\mathcal N}$ for the threshold using
lags in $\mathcal N$. Then
$\sigma_{\mathcal N}\le\sigma\le\sigma_L$.
\begin{theorem}[Periodic sampling]\label{thm:sampling}
For every tuning $q$, $\sigma(q)=\inf_{L\ge3}\sigma_L(q)$. If $q\in\mathcal D$, then either
$q\in\mathcal T_n$ for some $n$, and then $(\kappa',\alpha L,\beta)\in\mathcal T_n$ for all $\kappa'\ge\kappa$, or for
every $\kappa'>\kappa$ there is $L_0$ with $(\kappa',\alpha L,\beta)\in\mathcal T_L$ for all $L\ge L_0$. In
particular $\mathcal D=\cl\bigcup_L\mathcal T_L$.
\end{theorem}
The proof is given in Appendix~\ref{app:levels-sampling}.
\begin{theorem}[Global structure]\label{thm:global}
\begin{enumerate}[label=\textup{(\alph*)}]
\item For every $L\ge3$ the periodic threshold $\sigma_L$ is finite and continuous on
the set of tunings, and the threshold $\sigma$ is continuous.
\item $\sigma$ is strictly decreasing in $\kappa$ at fixed $(\alpha L,\beta)$, and
$\sigma\to+\infty$ as $\kappa\downarrow1$.
\item Put
\[
 \kappa_\partial(\alpha L,\beta)=\inf\{\kappa>1:\ \sigma(\kappa,\alpha L,\beta)\le0\}\in(1,\infty]
\]
on the strip $\{\beta\in[0,1),\ 0<\alpha L<2(1+\beta)\}$. Then
\[
 \mathcal D=\{(\kappa,\alpha L,\beta):\ \kappa\ge\kappa_\partial(\alpha L,\beta)\},
\]
and $\kappa_\partial$ is continuous as a map into $(1,\infty]$.
\end{enumerate}
\end{theorem}
The proof is given in Appendix~\ref{app:levels-sampling}.
Writing $\kappa_L$ for the condition-number threshold of
$\PP_L$, we also have $\kappa_\partial=\inf_{L\ge3}\kappa_L$
(Corollary~\ref{cor:kappaL}).

Here $(1,\infty]$ carries its order topology; an infinite boundary means
convergence for every finite condition number. This occurs throughout
the zero-momentum face.

The essential
approximation step is to preserve all signed lag constraints
simultaneously. A strict cotangent shift gives positive slack unless
the measure is already supported on a periodic frequency grid. Lemmas~\ref{lem:A-smooth}
and~\ref{lem:A-gain} turn that slack into a uniform margin after
smoothing; Lemma~\ref{lem:A-alias} then preserves it on every
sufficiently fine periodic grid. Compactness and continuity of the
finite-grid thresholds give continuity of $\sigma$. Strict decrease
in $\kappa$ follows from Lemma~\ref{lem:kappa} and yields the boundary
graph, including its possible value $+\infty$.

On $\partial\mathcal D$, continuity gives $\sigma=0$, so threshold
extremals solve the original program. Their supports are constrained
simultaneously by one limiting dual symbol.
By complementary slackness, the symbol vanishes on every
threshold-extremal support, and each positive dual coefficient forces
its corresponding constraint to hold with equality
(Theorem~\ref{thm:limitsymbol}). Its zeros will determine the periodic
structure of the boundary.

\subsection{Farey geometry and finite spectral measures}\label{sec:farey}
The gap $I_4$ already exhibits the distinction between one harmonic and
the full program.
\begin{proposition}[A two-harmonic thirteen-cycle]\label{prop:13}
Let $\kappa=100$, $\alpha L=70471/31800$ and $\beta=6169/9900$. Then the tuning
lies in the gap $I_4$, $\Theta>\pi/6$ on $(0,\pi]$, and no one-atom measure solves
$\FP(\Theta)$. The spectrum $\nu$ with the weights $149/797$ at
$2\cdot2\pi/13$ and $648/797$ at $3\cdot2\pi/13$, with $C=0$, is feasible at $t=0$;
equivalently, $w=|g-2i\Ima m|\,\nu$ satisfies the hypothesis of
Theorem~\ref{thm:fan-div}(a) with $p=3$, $k=13$.
Hence a function in $\cF(\R^5)$ has a trajectory that converges to a
thirteen-periodic orbit with the two harmonics $2/13$ and $3/13$, and $\rho\ge1$.
\end{proposition}
The proof is given in Appendix~\ref{app:farey-examples}.
\begin{proposition}[Convergence in a gap]\label{prop:conv-gap}
Let $\kappa=100$, $\alpha L=111/50$ and $\beta=31/50$. Then the tuning lies in the
gap $I_4$, and the two-lag inequality \eqref{eq:twolag} holds with $K=4$ and the
rational multipliers listed in Appendix~\ref{app:farey}. Hence $\rho<1$.
\end{proposition}
The proof is given in Appendix~\ref{app:farey-examples}.

Thus $I_4$ contains both a two-frequency periodic obstruction and a
tuning certified to converge by the lags $1$ and $-4$.

\subsubsection{Farey parents and binding constraints}
Introduce
\begin{equation}\label{eq:lattice}
 \xi=\frac{\omega}{2\Theta},\qquad \eta=\frac\pi\Theta.
\end{equation}
The image of the angle curve is $\mathcal C$, and its radial direction
$\xi/\eta=\omega/(2\pi)$ is the rotation number, the fraction of a
full turn per iteration. A lag function is negative exactly in the open
strips $0<n\xi-\ell\eta<1$, $\ell\in\Z$;
the rational contact points become the lattice points $(j,k)$
(Theorem~\ref{thm:lattice}).

For $0<j/k<1/2$ in lowest terms, write
$a/b<j/k<c/d$ for its Farey parents: $j=a+c$, $k=b+d$ and
$bc-ad=1$~\cite[Chapter~III]{HW}.
A single atom at $2\pi j/k$ is feasible exactly when its phase
is at most $\pi/k$. As the phase crosses $\pi/k$, its first violated
lags are $b$ and $-d$, the denominators of the two Farey parents;
they form one residue class modulo $k$ (Theorem~\ref{thm:farey}).

The parent denominators identify the binding lags. Each periodic
obstruction in $I_K$ has an anchor, an atom in the closed dip,
at a reduced fraction $j/k$ of its window; necessarily $k\ge2K+1$
(Proposition~\ref{prop:farey}).
For later use put
\[
 \omega_0=2\pi j/k,\quad c_0=\cos\omega_0,\quad s_0=\sin\omega_0,
 \qquad c'=\pi/k,\quad h_0=s_0\cot c',\quad b\equiv j^{-1}\pmod k.
\]
The single-harmonic tongue is the explicit graph region
$\kappa\ge\kappa_{j/k}(y,\beta)$ of Lemma~\ref{lem:tongue}.
Hence $\kappa_\partial\le\inf_{j/k}\kappa_{j/k}$:
the single-harmonic envelope bounds the full threshold from above,
and periodic mixtures may lower it (Proposition~\ref{prop:13}).
By Theorem~\ref{thm:plateaujk}, each Farey fraction carries an open
plateau of equality. Figure~\ref{fig:tongues} shows the competing tongue
graphs across $I_4$ and their lower envelope. The fraction attaining this
envelope changes in steps; the undamped quadratic rotation number of
Proposition~\ref{prop:location} gives the limiting locations of the
single-harmonic plateaus.

\begin{figure}[tbp]
\centering
\includegraphics[width=\linewidth]{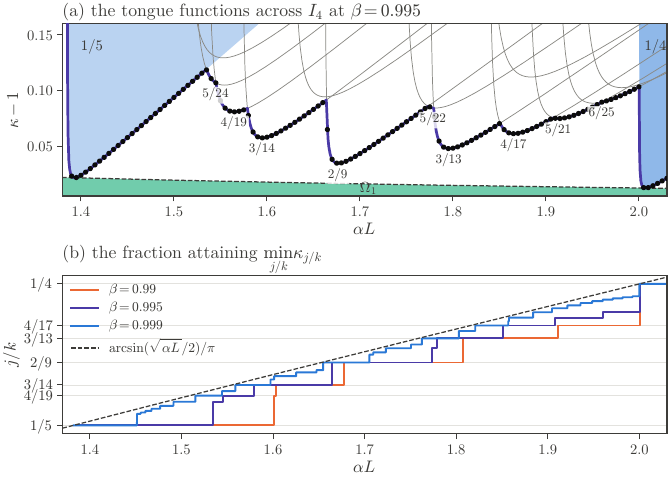}
\caption{Tongue graphs and their lower envelope across $I_4$.
(a) At $\beta=0.995$: individual graphs $\kappa_{j/k}-1$ with $k\le30$
(grey), their envelope over $k\le60$ (violet), and periodic upper bounds
on the full threshold, using periods at most $60$ (dots). The dashed curve
bounds $\Omega_1$; blue regions are the tongues of $1/5$ and $1/4$.
(b) The fraction attaining the envelope over $k<80$, at three values of
$\beta$. The dashed curve is the undamped quadratic rotation number
$\arcsin(\sqrt{\alpha L}/2)/\pi$ (Proposition~\ref{prop:location}).}
\label{fig:tongues}
\end{figure}

\subsubsection{Compensation and pinning}
The sequences of neighboring fractions
$ (a+mj)/(b+mk)$ and $(c+mj)/(d+mk)$, called Farey fans,
approach $j/k$ from its two sides.
For the endpoint $1/K$ the fan has
\begin{equation}\label{eq:fanpoints}
 k_p=pK+1,\qquad p/k_p,\quad p\ge1.
\end{equation}
Its common upper parent lag is $-K$; the associated lag family is
\begin{equation}\label{eq:fanlags}
 \mathcal N_{\rm fan}(K)=\{mK+1:m\ge0\}\cup\{-mK:m\ge1\}.
\end{equation}
Under the phase bounds of Theorem~\ref{thm:fan-div}, the lag functions
at an anchor $p/k_p$ and a second frequency, its companion $(p-1)/k_p$,
have opposite signs for the lags $1$ and $-K$; positive masses can
therefore compensate the violated constraints.
Theorem~\ref{thm:fan-div} gives the exact determinant criterion;
Theorem~\ref{thm:twoatom} gives the complete ratio interval for every
anchor and companion, selected by Farey parents one level below the anchor.
The underlying condition excludes lattice points from a specified
box: every lag function negative at the anchor must be positive
at its companion.

The dual construction uses the same lags. With Chebyshev polynomials
$T_n,U_n$, the simplest certificate is
\begin{equation}\label{eq:twolag}
 \Xi(c)\bigl[(1-c)+r(1-T_K(c))+\gamma\bigr]
 <(1-c^2)\bigl(1-rU_{K-1}(c)\bigr),\qquad -1\le c\le1,
\end{equation}
where $r\ge0$ and $\gamma>0$.
It proves $\rho<1$ by Theorem~\ref{thm:fan-conv}.
Allowing the other fan lags gives the general polynomial criterion
\eqref{eq:fanlag}.

Finally, lags in one residue class agree on $\Omega_k$ but have different
derivatives there. Their exact difference is the pinning identity
of Lemma~\ref{lem:pinid}; for the fan, $-pK\equiv1$ and
$(p-1)K+1\equiv-K\pmod{k_p}$ (Lemma~\ref{lem:pin}).
Redistributing dual weight within these classes preserves contact values
and controls contact slopes. This permits continuation of a threshold measure with fixed periodic
support, as established in Subsection~\ref{sec:boundaryK}.

\subsection{Regularity of the divergence boundary}\label{sec:boundaryK}
In a gap, convergence and divergence meet on the surface
\[
 \mathcal B_K=\partial\mathcal D\cap I_K=\{q\in I_K:\sigma(q)=0\}.
\]
The surface is the graph of $\kappa_\partial$ over its projection to the
$(\alpha L,\beta)$ plane. Off this surface, either a finite-lag certificate or a periodic
measure determines the behavior. On the surface, the support and
regularity of a threshold measure determine the local geometry.
Call a solution of $\FP(\Theta_q)$ at $q\in\mathcal B_K$ an extremal.
Write $\mathcal P=\{q\in\mathcal B_K:\sigma_L(q)=0\text{ for some }L\}$,
and let $\mathcal R$ be its regular points, in the sense of
Definition~\ref{def:regular}.

\begin{theorem}[Classification]\label{thm:class}
Let $q\in I_K$. Exactly one of the following holds.
\begin{enumerate}[label=\textup{(\roman*)}]
\item $\kappa<\kappa_\partial(\alpha L,\beta)$: $\sigma(q)>0$, a strict certificate symbol with finitely many lags
exists, and $\rho(q)<1$.
\item $\kappa>\kappa_\partial(\alpha L,\beta)$: $\sigma(q)<0$, there is $L$ with $q\in \mathcal T_{mL}$ for all $m\ge1$,
and $\rho(q)\ge1$.
\item $\kappa=\kappa_\partial(\alpha L,\beta)$: $q\in\mathcal B_K$, $\rho(q)\ge1$, and $q$ lies in exactly one of
$\mathcal R$, $\mathcal P\setminus\mathcal R$ and $\mathcal B_K\setminus\mathcal P$.
\end{enumerate}
In (ii) and (iii), every solution of every feasible $\PP_L(\Theta_q)$ has an anchor at a
Farey fraction $j/k\in(1/(K+1),1/K)$ with $k\mid L$ and $k\ge2K+1$.
\end{theorem}
The proof is given in Appendix~\ref{app:boundary-heavy}.

Thus convergence off the boundary admits a finite-lag certificate,
and divergence off the boundary admits a periodic spectral measure. Boundary points divide into regular periodic
points, singular periodic points, and points with no threshold extremal
supported on a finite periodic grid. The last class is approached by the thresholds of every
sufficiently large period; plateau endpoints belong to the second class
(Theorem~\ref{thm:bdry}).

\subsubsection{Analytic continuation of regular spectra}
At an $L$-regular point, $S\subset\Omega_L\cap(0,\pi)$ is the finite
frequency support, $N$ is a finite set of certificate lags, and
$w^*,\lambda^*$ are their positive normalized weights. Write
$\mathcal A=(N\setminus L\Z)\bmod L$ for the active nonzero residue
classes, whose lag constraints hold with equality at $w^*$, and
$\Lambda_n^{q,s}=\Lambda_n[\Theta_{q,s}]$.
Definition~\ref{def:regular} requires a one-dimensional kernel for
the active constraint matrix, strict inactive and minimizer inequalities, negative
second derivatives of the dual symbol at its zeros, a positive endpoint
moment and positive pinning factors, and full rank of the pinning equations.
These conditions permit the masses and the certificate to
vary while the frequencies stay fixed.

\begin{theorem}[Mode locking]\label{thm:modelock}
Let $(q_*,s_*)$ be $L$-regular with data $S,N,w^*,\lambda^*$. Then $s_*=\sigma(q_*)=\sigma_L(q_*)$,
and there are a neighborhood $U$ of $q_*$ and real analytic maps $s:U\to\R$,
$w:U\to\R^S_{>0}$ with $s(q_*)=s_*$, $w(q_*)=w^*$, such that for every $q\in U$:
\begin{enumerate}[label=\textup{(\alph*)}]
\item $\sigma(q)=\sigma_L(q)=s(q)$, and $\det\bigl(\Lambda^{q,s(q)}_\varrho(\omega)\bigr)_{\varrho\in\mathcal A,\omega\in S}=0$;
\item $w(q)$ solves $\FP(\Theta_{q,s(q)})$;
\item the solutions of $\FP(\Theta_{q,\sigma(q)})$ are exactly the measures $c\,w(q)+a\,\delta_0$ with $c>0$ and
$0\le a\le c\int\cos\Theta_{q,\sigma(q)}\,dw(q)$; the threshold extremal on $(0,\pi]$ is unique up to
scaling and carried by $S$;
\item $\nabla s(q)=-\sum_{\omega\in S}\varpi_\omega(q)\,\nabla_qh_q(\omega)$ with weights $\varpi_\omega(q)>0$,
$\sum_\omega\varpi_\omega=1$; in particular $\partial_\kappa s<0$;
\item $(q,s(q))$ is $L$-regular with the data $S,N$.
\end{enumerate}
Consequently, if $s_*=0$, then near $q_*$ the boundary $\partial\mathcal D$ is the analytic graph
$\kappa=\kappa_\partial(\alpha L,\beta)$ of the function defined by $s=0$, and along it the
threshold extremal keeps the support $S$.
\end{theorem}
The proof is given in Appendix~\ref{app:boundary-regularity}.

The determinant
equation determines the shift and the positive primal masses by the
analytic implicit function theorem. The independent slope equations
then determine positive dual weights whose symbol has the same
contacts. The strict curvature and complementary slackness conditions
ensure that these contacts remain its only nonzero zeros. They give
both optimality of the continued shift and the stated classification
of every threshold measure.

The regular boundary is relatively open. On each connected component its
support, least period, and tight residue classes are constant
(Corollary~\ref{cor:plateau}); these components are the plateaus.
Thus mode locking fixes the frequencies, while their masses and the
boundary height $\kappa_\partial$ may vary analytically.
The finite lag set of a regular certificate gives an exact local decision
procedure, and competing grids incur a quadratic loss
(Theorems~\ref{thm:tightclass} and~\ref{thm:qlock}).

The least period determines all feasible threshold periods. The endpoint
fan realizes a plateau precisely when its least period is coprime to the
gap index (Corollary~\ref{cor:periods}).

\subsubsection{Mixed-spectrum plateaus}
Regular threshold measures may require more than one nonzero
frequency. The following two boundary points in $I_4$ have supports
of sizes two and three, respectively.

\begin{theorem}[Two mixed-spectrum plateaus]\label{thm:plateaus}
For each row of Table~\ref{tab:plateaus}, take the displayed decimal values
of $(\alpha L,\beta)$ exactly. There is a unique boundary value
$\tilde\kappa$ in the stated open interval. The tuning
$(\tilde\kappa,\alpha L,\beta)$ lies in $I_4$ and is $L$-regular with
support $S$. It lies on an open analytic
plateau whose extremals are exactly the positive multiples of a measure
with positive mass at every point of $S$, together with the admissible
atom at $0$ of Theorem~\ref{thm:modelock}(c). Throughout its intersection
with $I_4$, its unique anchor is the displayed fraction.
The least period of the first plateau is $18$, and none of its threshold
periods has the form $4p+1$. Every extremal on the second plateau has
three nonzero frequencies.
\end{theorem}
The proof is given in Appendix~\ref{app:boundary-examples}.

\begin{table}[tbp]
\centering\small
\caption{Mixed-spectrum boundary points and rational enclosures.}\label{tab:plateaus}
\begin{tabular}{cccc}
\toprule
anchor & $(\alpha L,\beta)$ & interval for $\tilde\kappa$ & $S\cdot L/(2\pi)$\\
\midrule
$2/9$ & $(2.201327,0.6268)$ & $(100.5508210,100.5508214)$ & $\{3,4\}$, $L=18$\\
$5/21$ & $(2.240775,0.6152)$ & $(106.8055118,106.8055121)$ & $\{3,4,5\}$, $L=21$\\
\bottomrule
\end{tabular}
\end{table}

At some points a two-lag relaxation already has the threshold of the full
program. Such a point occurs when its two contact frequencies lie on a
common grid and the remaining residue constraints are strict
(Theorem~\ref{thm:exact}).

Theorem~\ref{thm:exact313} gives such an exact point with anchor $3/13$
at $\kappa\approx83.5343$. Its threshold spectrum is
$\{2,3\}\cdot2\pi/13$, up to scaling and an admissible atom at $0$;
the lags $1,-4$ certify the threshold at this point, and regularity
places it on an analytic plateau.

The geometry behind this example is a double tangency between the
quadratic $\Xi_q$ and a rational Chebyshev profile
(Proposition~\ref{prop:twolagtangency}). It reduces the contact equations
to one scalar equation and the global symbol inequality to the sign of
a linear polynomial; see Appendix~\ref{app:boundary}.

Figure~\ref{fig:staircase} illustrates a finite-period approximation to
$\mathcal B_4$ at $\kappa=150$, across the gap between the tongues of
$1/5$ and $1/4$. The displayed frequency intervals are indexed by Farey
fractions, with the widest interior interval at $3/13$. These intervals
illustrate the mode locking of Theorem~\ref{thm:modelock}; the mixed
examples above describe possible spectral supports within plateaus.

\begin{figure}[tbp]
\centering
\includegraphics[width=\linewidth]{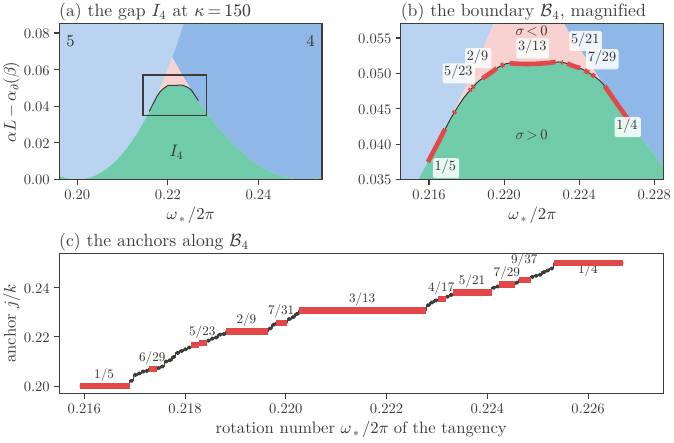}
\caption{A finite-period approximation to the section at $\kappa=150$,
using programs with $9\le L\le120$. The horizontal coordinate is the
rotation number $\omega_*/2\pi$ of the tangency on $\partial\Omega_1$;
the vertical coordinate is $\alpha L-\alpha_\partial(\beta)$, where
$(150,\alpha_\partial(\beta),\beta)\in\partial\Omega_1$.
(a) The gap between the tongues of $1/5$ and $1/4$, with green and pale
red below and above the computed curve. (b) The boxed region magnified;
red segments mark constant selected frequencies. (c) The rotation number
of the largest atom of the minimizing periodic measure in cotangent form,
$\nu=\sin\Theta\,w$ (Proposition~\ref{prop:monoform}). The constant-frequency
intervals illustrate periodic mode locking.}
\label{fig:staircase}
\end{figure}

\subsection{Single-frequency plateaus}\label{sec:plateau}
Every reduced fraction $j/k$ in the window of a gap also supports a
single-harmonic plateau. Its threshold lies on the edge of its tongue,
and its entire regular component admits an exact description.
The edge and fan criteria apply to every reduced $0<j/k<1/2$,
$k\ge3$.
Use the notation $\omega_0=2\pi j/k$, $c_0=\cos\omega_0$,
$s_0=\sin\omega_0$, $c'=\pi/k$, $h_0=s_0\cot c'$ of
Subsection~\ref{sec:farey}. Let $b$ be the inverse of $j$ modulo $k$ and set
$\mathcal N_{j/k}=(b+k\Z)\cup(k\Z\setminus\{0\})$.
For every tuning define
\[
 s_*(q)=\cot c'-h_q(\omega_0),\qquad
 \Theta_*=\Theta_{q,s_*(q)},\qquad \bar n(q)=2\Theta_*'(\omega_0).
\]
The single atom $\delta_{\omega_0}$ is feasible at this shift, so
\begin{equation}\label{eq:sstarjk}
 \sigma(q)\le\sigma_k(q)\le s_*(q).
\end{equation}
On the tongue edge $s_*=0$. For a fraction $j/k\in(1/(K+1),1/K)$,
parametrize that edge by
\[
 \Xi^{a,D}(c)=h_0+(1+D)(c-c_0)-a(c-c_0)^2,\qquad a>0.
\]
Whenever $\Xi^{a,D}(\pm1)<0$ it defines a unique tuning $q(a,D)$.
The slope index records twice the derivative of the angle curve
at the contact:
\begin{equation}\label{eq:nbaredge}
 \bar n(D)=2\sin^2c'\bigl(1+D+\cot c'\cot\omega_0\bigr).
\end{equation}
Let $a_E(D)$ be the explicit piecewise affine threshold for membership
in $I_K$ (Lemma~\ref{lem:edgedomain}). Let $a_F(D)$ be the least curvature
required by a finite single-harmonic certificate, defined by the
one-variable minimax problem in Definition~\ref{def:admissible}.
Its admissible finite nonnegative lag weights cancel the contact slope;
$a_F$ minimizes the global maximum of their curvature profiles.
When the slope index lies strictly between two consecutive lags
in $b+k\Z$, the adjacent-lag symbol is the convex combination
with zero derivative at $\omega_0$.

\begin{definition}[Edge symbol]\label{def:edgesymbol}
Let $q$ be a tuning with $\bar n=\bar n(q)\notin b+k\Z$. Let $n_1<\bar n<n_2=n_1+k$ be the
members of $b+k\Z$ adjacent to $\bar n$, put $\vartheta=(\bar n-n_1)/k\in(0,1)$, and, with the lag functions
of $\Theta_*$,
\[
 T_q=(1-\vartheta)\Lambda_{n_1}+\vartheta\Lambda_{n_2},\qquad F_q=-\tfrac12T_q .
\]
\end{definition}

Write $\lambda^{\rm e}_{n_1}=1-\vartheta$ and
$\lambda^{\rm e}_{n_2}=\vartheta$. The associated profile is
\[
 \Psi_{\rm e}(c)=\frac{\sin\omega\sum_{i=1}^2\lambda^{\rm e}_{n_i}\sin(n_i\omega)}
 {\sum_{i=1}^2\lambda^{\rm e}_{n_i}(1-\cos(n_i\omega))},\qquad c=\cos\omega,
\]
\[
 R_{\lambda^{\rm e}}(c)=\frac{h_0+(1+D)(c-c_0)-\Psi_{\rm e}(c)}{(c-c_0)^2}.
\]
Values at the endpoints and at $c_0$ are understood by continuous
extension. The curvature threshold $a_2(D)=R_{\lambda^{\rm e}}(c_0)$
is where the contact curvature vanishes (Lemma~\ref{lem:ray});
Definition~\ref{def:admissible} gives the full class of profiles defining $a_F$.

\begin{theorem}[The single-harmonic plateau of $j/k$]\label{thm:plateaujk}
Let $K\ge2$ and let $j/k\in(1/(K+1),1/K)$ be a Farey fraction. Put
\[
 \Pi_{j/k}=\bigl\{q(a,D):\ D\in\R,\ a>\max\{a_E(D),a_F(D)\}\bigr\}.
\]
\begin{enumerate}[label=\textup{(\alph*)}]
\item $\Pi_{j/k}$ is a plateau contained in $\mathcal B_K$, and it consists of all regular points of $\partial\mathcal D$ whose
threshold extremal is carried by $\{\omega_0\}$.
\item On $\Pi_{j/k}$ the threshold extremal is $\delta_{\omega_0}$, up to scaling and an atom at $0$, its unique anchor is $j/k$, and $\kappa_\partial=\kappa_{j/k}$ on the projection
of $\Pi_{j/k}$ to the $(\alpha L,\beta)$ plane.
\item If $\bar n(D)\notin b+k\Z$, then $a_2(D)\le a_F(D)\le\max_{[-1,1]}R_{\lambda^{\rm e}}$, and every defined $q(a,D)$ with $a<a_2(D)$ lies in
the interior of $\mathcal D$, where the fan pairs of $j/k$ have a smaller threshold than the single harmonic (Theorem~\ref{thm:fanjk}).
\end{enumerate}
\end{theorem}
The proof is given in Appendix~\ref{app:plateau-domain}.

Lemma~\ref{lem:ray} identifies strict regular
certificates with $a>\max R_\lambda$, and
Theorem~\ref{thm:largea} supplies such a certificate at every sufficiently
large curvature. A common continuous upper bound for these sufficient
curvatures makes the parameter region connected. The converse uses the
uniqueness of a regular extremal to exclude contact with an integer
tongue. The curvature comparison in (c) then follows from the exact
second-order fan asymptotics below.

\begin{corollary}[A plateau at every Farey fraction]\label{cor:everyfarey}
For every $K\ge2$, every Farey fraction $j/k\in(1/(K+1),1/K)$ is the unique anchor of the plateau $\Pi_{j/k}\ne\emptyset$ of
$\mathcal B_K$, whose threshold extremal is the single harmonic at $2\pi j/k$ and on which $\kappa_\partial$ is the rational
function $\kappa_{j/k}$. In particular $\mathcal B_K$ contains regular points for every $K\ge2$.
\end{corollary}
The proof is given in Appendix~\ref{app:plateau-domain}.

\begin{proposition}[Location of $\Pi_{j/k}$]\label{prop:location}
For fixed $D$, as $a\to\infty$,
\[
 \beta\to1,\qquad a(1-\beta)^2\to\frac{2\tan(\pi/(2k))}{\sin\omega_0},\qquad \alpha L\to4\sin^2\frac{\pi j}k,\qquad \kappa\to1
\]
along $q(a,D)$. Thus the projection of $\Pi_{j/k}$ to the $(\alpha L,\beta)$ plane accumulates at the point
$(4\sin^2(\pi j/k),1)$, at which the characteristic polynomial $z^2-(2-\alpha L)z+1$ of the undamped quadratic has
the roots $e^{\pm2\pi ij/k}$.
\end{proposition}
The proof is given in Appendix~\ref{app:plateau-domain}.

\begin{figure}[tbp]
\centering\includegraphics[width=\linewidth]{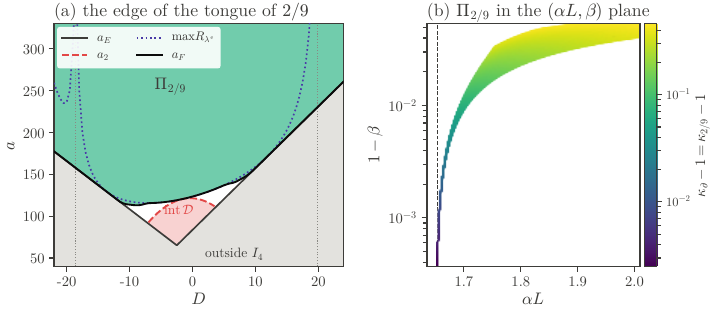}
\caption{The single-harmonic plateau of $2/9$.
(a) In edge coordinates: the gap boundary $a_E$, the curvature
obstruction $a_2$ (dashed), and a numerical evaluation of the adjacent-lag
bound $\max R_{\lambda^{\rm e}}$ (dotted). The solid curve approximates
$a_F$ by finite-lag programs. (b) The image in $(\alpha L,1-\beta)$,
with logarithmic vertical scale; its limiting abscissa is
$4\sin^2(2\pi/9)$. Color indicates $\kappa_\partial-1$.}\label{fig:plateau}
\end{figure}

\subsubsection{Second-order asymptotics of Farey fan pairs}
The same curvature that certifies the single harmonic controls the
nearest rational competitors. For the edge symbol of
Definition~\ref{def:edgesymbol},
\[
 F_q''(\omega_0)=\sin c'\,\Theta_*''(\omega_0)
   -\tfrac12\cos c'\,k^2\vartheta(1-\vartheta).
\]
Positive curvature, together with the global sign and endpoint
conditions, gives regularity. Nonpositive curvature obstructs every
regular single-harmonic certificate at that frequency
(Theorem~\ref{thm:edgejk}).

\begin{definition}[Fan pairs of $j/k$]\label{def:fanpairjk}
Let $\varsigma\in\{\pm1\}$ and let $L>k$ satisfy $jL\equiv\varsigma\pmod k$. Put $J=(jL-\varsigma)/k$,
\[
 \omega_a=\frac{2\pi J}L,\qquad \omega_c=\frac{2\pi(J+\varsigma)}L,\qquad \delta=\omega_a-\omega_0=-\frac{2\pi\varsigma}{kL}.
\]
Then $\gcd(J,L)=1$, $\omega_c=\omega_0-(k-1)\delta$ and $(k-1)\omega_a+\omega_c=2\pi j$. The fan pair of order $L$
is $\{\omega_a,\omega_c\}$; $\omega_a$ is its point nearer to $\omega_0$.
\end{definition}

\begin{theorem}[Fan pairs at a Farey fraction]\label{thm:fanjk}
Let $q$ be a tuning with $\bar n(q)\notin b+k\Z$, let $\varsigma\in\{\pm1\}$, and put
$e_*(q)=-(k-1)F_q''(\omega_0)/(2\sin^3c')$.
\begin{enumerate}[label=\textup{(\alph*)}]
\item For every $e<e_*(q)$ there are $z\in\R$ and $L_0$ such that for every $L\ge L_0$ with
$jL\equiv\varsigma\pmod k$ the measure with mass $k-1+z\delta$ at $\omega_a$ and mass $1$ at $\omega_c$ solves
$\FP(\Theta_{q,s})$ at the shift $s=s_*(q)-e\delta^2$.
\item For every $e>e_*(q)$ there is $L_0$ such that for every $L\ge L_0$ with $jL\equiv\varsigma\pmod k$ no
nonzero measure carried by $\{\omega_a,\omega_c\}$ solves $\FP(\Theta_{q,s})$ at $s=s_*(q)-e\delta^2$.
\item The threshold $s^{\rm pair}_L(q)$ of $\FP$ restricted to measures carried by the fan pair of
order $L$ satisfies
\[
 s^{\rm pair}_L(q)=s_*(q)+\frac{(k-1)F_q''(\omega_0)}{2\sin^3c'}\,\delta^2+o(\delta^2)\qquad\bigl(L\to\infty,\ jL\equiv\varsigma\ (\mathrm{mod}\ k)\bigr).
\]
\item If $F_q''(\omega_0)<0$, then $\sigma(q)<s_*(q)$. In particular a point of the edge with
$F_q''(\omega_0)<0$ lies in the interior of $\mathcal D$.
\item If $F_q>0$ on $(0,\pi]\setminus\{\omega_0\}$, $F_q''(0)>0$ and $F_q''(\omega_0)>0$, then for $L\to\infty$ with $k\nmid L$
\[
 \sigma_L(q)-s_*(q)\ \ge\ \frac{F_q''(\omega_0)}{2\sin^3c'}\Bigl(\frac{2\pi}{kL}\Bigr)^2(1+o(1)),
\]
and for $jL\equiv\pm1\pmod k$ also $\sigma_L(q)-s_*(q)\le s^{\rm pair}_L(q)-s_*(q)$.
\end{enumerate}
The constants are locally uniform: on every compact set of tunings on which $\bar n$ stays in one
interval $(n_1,n_1+k)$, and for every $\eta_0>0$, \textup{(a)} holds with $e=e_*(q)-\eta_0$ and \textup{(b)}
with $e=e_*(q)+\eta_0$ for all $L\ge L_0$, with $L_0$ independent of $q$.
\end{theorem}
The proof is given in Appendix~\ref{app:plateau-fans}.

The proof reduces all lag
constraints to a trigonometric polynomial on the finite grid of
phases attained by the periodic frequencies. The two
zeros of its limiting polynomial require separate local estimates.
At one zero the grid spacing controls the linear perturbation; at the
other the leading constraints form a convex quadratic on $b+k\Z$.
Balancing its two adjacent values gives $e_*$. Uniform Taylor bounds
near the zeros and a strict minimum on their compact complement
control every lag, including those increasing with $L$.

Negative contact curvature therefore gives a two-frequency threshold
strictly smaller than the one-frequency threshold. Under the stated
global sign conditions, positive curvature yields a quantitative
$L^{-2}$ lower bound for grids that omit the threshold frequency.

\subsection{Gap boundaries and tongue junctions}\label{sec:ends}
The edges of $I_K$ are the closing dip on $\partial\Omega_1$ and the
adjacent tongue edges at $2\pi/(K+1)$ and $2\pi/K$; for $K=2$ the
latter endpoint is $\alpha L=2(1+\beta)$. Near a tongue tip the
single-harmonic spectrum persists even at a resonant slope index.
Fix $K'\ge3$, put $\omega_*=2\pi/K'$, $c'=\pi/K'$, and use
$s_*(q)=\cot c'-h_q(\omega_*)$.
A tip satisfies $\Theta_q\ge\omega/2$ with equality only at
$\omega_*$.

\begin{theorem}[Tips of the tongues]\label{thm:tips}
Let $K'\ge3$ and let $q_0$ be a tip of the $K'$-tongue. There is a neighborhood $V$ of
$q_0$ such that for every $q\in V$:
\begin{enumerate}[label=\textup{(\alph*)}]
\item $\sigma(q)=\sigma_{K'}(q)=\cot(\pi/K')-h_q(2\pi/K')$; in particular $\sigma$ is real analytic on $V$;
\item the threshold extremal is $\delta_{2\pi/K'}$, up to scaling and an atom at $0$;
\item $\mathcal D\cap V=\{q\in V:\ \Theta_q(2\pi/K')\le\pi/K'\}$, the closed tongue, and
$(I_{K'-1}\cup I_{K'})\cap V\subset\{\sigma>0\}$;
\item $(q,s_*(q))$ is $K'$-regular with support $\{2\pi/K'\}$ and the lag set $\{1,1+K',1-K'\}$.
\end{enumerate}
\end{theorem}
The proof is given in Appendix~\ref{app:edges-tips}.

At the other points of $\partial\Omega_1$, a strict finite certificate
persists. Consequently $\cl\mathcal B_K$ stays away from
$\partial\Omega_1$ (Corollary~\ref{cor:omega1}). In the dip coordinates,
convergence occupies a neighborhood of the diagonal side, while every
curvature ray eventually becomes divergent.

\begin{theorem}[The boundary in every gap]\label{thm:BK}
Let $K\ge2$.
\begin{enumerate}[label=\textup{(\alph*)}]
\item The sets $I_K\cap\{\sigma>0\}$ and $I_K\cap\{\sigma<0\}$ are nonempty and open, and
$\mathcal B_K=I_K\cap\{\sigma=0\}$ separates them: every path in $I_K$ from one to the other meets
$\mathcal B_K$. In particular $\mathcal B_K\ne\emptyset$.
\item In the coordinates $(a,r_1,r_2)$ of Proposition~\ref{prop:chart}, every point $(a,r,r)$ with
$c_K<r<c_{K+1}$ and $a>a_0(r,r)$ has a neighborhood in which $\sigma>0$; near the corners
$r_1=r_2=c_{K+1}$ and, for $K\ge3$, $r_1=r_2=c_K$, $\sigma>0$ on $I_K$ as well (Theorem~\ref{thm:tips}(c));
and every ray $\{(a,r_1,r_2):a>a_0(r_1,r_2)\}$ over a point of the triangle contains points with
$\sigma<0$.
\item For every $r$ with $c_K<r\le c_{K+1}$, and also for $r=c_K$ if $K\ge3$, there is $\delta>0$ such that for all
$c_K<r_1<r_2<c_{K+1}$ with $|r_1-r|,|r_2-r|<\delta$ the ray over $(r_1,r_2)$ meets $\mathcal B_K$. Thus the projection of
$\mathcal B_K$ to the triangle $c_K<r_1<r_2<c_{K+1}$ contains a relative neighborhood of its diagonal side,
including the corners for $K\ge3$ and the corner $r_1=r_2=c_3$ for $K=2$.
\end{enumerate}
\end{theorem}

\begin{proof}
Fix $c_K<r<c_{K+1}$ and $a>a_0(r,r)$. By
Proposition~\ref{prop:chart}, the point $(a,r,r)$ is a nonlocked
point of $\partial\Omega_1$.
Corollary~\ref{cor:omega1} gives a neighborhood with $\sigma>0$.
Theorem~\ref{thm:tips}(c) gives the same conclusion at the stated
corners, with the stability endpoint excluded when $K=2$.
This proves the convergent-neighborhood assertions.

On a ray with $r_1<r_2$, choose a reduced rational frequency
$\omega_0=2\pi j/k$ with $r_1<\cos\omega_0<r_2$. Then
\[
 \Xi^{a,r_1,r_2}(\cos\omega_0)
 =1+\cos\omega_0
   +a(\cos\omega_0-r_1)(r_2-\cos\omega_0)\longrightarrow+\infty.
\]
For all sufficiently large $a$, the single-frequency condition
$\Theta_{q,s}(\omega_0)\le\pi/k$ therefore holds at some $s<0$.
Theorem~\ref{thm:farey} gives $\sigma(q)\le s<0$.
Both signs consequently occur in $I_K$, and both sign sets are
open by Theorem~\ref{thm:global}. Continuity forces a zero on
every path joining them, proving (a) and the remaining part of (b).

Finally fix a point of the diagonal side or one of the allowed
corners, and choose a finite $a$ strictly above its lower boundary.
For all nearby ordered roots $(r_1,r_2)$, the same $a$ is admissible
and belongs to the convergent neighborhood. Increasing $a$ along
each such ray eventually gives the negative sign just proved.
The intermediate value theorem therefore gives a point of
$\mathcal B_K$ on every such ray, which is (c).
\end{proof}

For $K\ge3$, divergence also enters the gap from a nonempty arc where
the two adjacent tongues meet. Proposition~\ref{prop:corner} locates
this arc explicitly by the negative curvature of the edge symbol.
Together, these results describe the boundary in each gap: the
closing-dip side has a convergent neighborhood, some tongue-junction
arcs have divergent neighborhoods, and the separating surface contains
an analytic single-frequency plateau for every Farey fraction in the
gap window.

\FloatBarrier
\section{Conclusion}\label{sec:conclusion}

A common spectral interpolation problem gives both universal rate
bounds and realizable obstructions. For Polyak's tuning it yields the
exact rate and matching finite-horizon bounds, attained by one
five-dimensional function initialized from rest. For general tunings it
gives a continuous condition-number threshold, while Farey geometry and
analytic continuation describe its regular plateaus. Single-harmonic
edge certificates and competing fan spectra bound their extent; the
mixed-harmonic examples exhibit the additional supports that a complete
boundary classification must accommodate.

Several questions remain. Are the regular points dense in each threshold
surface $\mathcal B_K$? Which mixed supports occur, and how do the
plateaus accumulate along their relative boundaries? A more explicit
description of the variational function $a_F$ in
Theorem~\ref{thm:plateaujk} would sharpen the two available edge bounds.
The exact point of $3/13$ also raises a classification problem: which
Farey fractions admit a threshold extremal with a two-lag certificate?
Proposition~\ref{prop:twolagtangency} reduces this question to an
explicit tangency equation and the remaining global sign conditions.

The boundary may also contain regimes outside periodic mode locking.
Does some point of $\mathcal B_K$ have no periodic extremal? By
Proposition~\ref{prop:heavy}, a connected boundary set contains uncountably
many such points if, after removing the zero atom and normalizing, every
extremal has an atom of mass at least one fixed $m_0>1/2$ and the common
location of these atoms is nonconstant.
Further questions concern the gap geometry after imposing additional lag
constraints, the smallest dimension of
mixed-harmonic obstructions, and the number of frequencies needed to
attain the rate inside the lag-one convergence region. In particular,
does the tuning of Proposition~\ref{prop:13} admit a nonconvergent scalar
trajectory, and is the rate in $\Omega_1\setminus\cQ$ always attained by
at most two frequencies?

\FloatBarrier
\appendix
\section{Interpolation and spectral characterization}\label{app:spectral}
\subsection{Normalization and smooth interpolation}
Throughout,
$\alpha>0$, $\beta\in[0,1)$ and \eqref{eq:quad} hold, $A$ and $B$ are as in
\eqref{eq:normrec}, and all Hilbert spaces are real. Since $B>0$ and
$A\pm B=1+\beta-\alpha\mu,\ 1+\beta-\alpha L$, condition \eqref{eq:quad} gives
\begin{equation}\label{eq:Abound}
 |A|+B=\max\bigl(|1+\beta-\alpha\mu|,|1+\beta-\alpha L|\bigr)<1+\beta .
\end{equation}

\begin{lemma}[Normalization]\label{lem:normalization}
The map $f\mapsto\psi$ is a bijection between functions in $\cF$ with
minimizer $0$ and value $0$ there, and $\cS$. With $e_k=x_k-x^\star$ and
$\sigma_k=\nabla\psi(e_k)$, the iteration \eqref{eq:hb} becomes
\eqref{eq:normrec}.
For Polyak's tuning $A=0$, $B=2q$ and $\beta=q^2$.
\end{lemma}
\begin{proof}[Proof of Lemma~\ref{lem:normalization}]
Put $M=L-\mu$ and $f_0=f-\mu\norm\cdot^2/2$. The convex $M$-smooth
function $f_0$ has a $1/M$-cocoercive gradient~\cite{THG}; hence
$\nabla\psi=\mathrm{id}-2\nabla f_0/M$ is nonexpansive. Conversely,
$\nabla f=(L+\mu)\mathrm{id}/2-M\nabla\psi/2$ is $\mu$-strongly monotone
and $L$-Lipschitz by Cauchy--Schwarz, so it defines a member of $\cF$.
The two transformations are inverse. Substitution gives
\eqref{eq:normrec}; \eqref{eq:polyak} gives $A=0$, $B=2q$ and $\beta=q^2$.
\end{proof}

\begin{theorem}[Interpolation on arbitrary sets]\label{thm:interpolation}
Let $\{(x_i,\sigma_i,\psi_i):i\in I\}$ be data in a real Hilbert space
containing the datum $(x_\star,\sigma_\star,\psi_\star)=(0,0,0)$ of an index
$\star\in I$. A function $\psi\in\cS$ with $\nabla\psi(x_i)=\sigma_i$
and $\psi(x_i)=\psi_i$ exists if and only if, for all $i,j\in I$,
\begin{equation}\label{eq:interpolation}
 |T_{ij}|\le R_{ij},\quad
 T_{ij}=\psi_j-\psi_i-\tfrac12\ip{\sigma_i+\sigma_j}{x_j-x_i},\quad
 R_{ij}=\tfrac14\bigl(\norm{x_j-x_i}^2-\norm{\sigma_j-\sigma_i}^2\bigr).
\end{equation}
\end{theorem}
\begin{proof}[Proof of Theorem~\ref{thm:interpolation}]
Set $F_i=\psi_i+\norm{x_i}^2/2$ and $G_i=\sigma_i+x_i$.
Expanding \eqref{eq:interpolation} gives precisely
\begin{equation}\label{eq:convexinterp}
 F_i\ge F_j+\ip{G_j}{x_i-x_j}+\tfrac14\norm{G_i-G_j}^2\quad(i,j\in I).
\end{equation}
This is the convex $C^{1,1}$ extension criterion of Azagra and
Mudarra~\cite[Theorem~1.3]{AM}, with gradient Lipschitz constant two.
It applies to arbitrary subsets of a real Hilbert space. Repeated sites
give identical values and gradients by the same inequality. Thus there is
a convex $2$-smooth interpolant $F$; cocoercivity makes
$\nabla F-\mathrm{id}$ nonexpansive, and $\psi=F-\norm\cdot^2/2$ is the
required interpolant. Necessity follows by applying the same convex
smoothness inequality to $F=\psi+\norm\cdot^2/2$.
An explicit conjugate construction is
\begin{equation}\label{eq:extension}
 H(y)=\sup_{i\in I}\{\ip{x_i}{y}-F_i+\tfrac14\norm{y-G_i}^2\},
 \qquad \psi=H^*-\tfrac12\norm\cdot^2.
\end{equation}
Indeed \eqref{eq:convexinterp} makes the $j$th quadratic attain the
supremum at $G_j$, with subgradient $x_j$; strong convexity of $H$ gives
the claimed smooth conjugate.
\end{proof}

\begin{lemma}[Exact spectral interpolation]\label{lem:spectralinterp}
For the orthogonal sum of waves with spectrum $\nu$ at the exponent $t$ and
function values $\psi_k=\frac14\bigl(C+2\int\Rea m\,d\nu\bigr)e^{2kt}$,
\begin{equation}\label{eq:waveslacks}
 R_{j,j+m}\pm T_{j,j+m}=\tfrac12e^{(2j+m)t}\bigl(\bar a_m\mp\bar b_m\pm C\sh mt\bigr),\quad
 R_{j\star}\pm T_{\star j}=\tfrac14e^{2jt}(\bar g\pm C).
\end{equation}
Consequently the data are interpolated by a function in $\cS$ if and only if
$(\nu,C)$ is feasible at $t$.
\end{lemma}
\begin{proof}[Proof of Lemma~\ref{lem:spectralinterp}]
For one wave and $j=0$, the defining expressions reduce to
$T_{0n}=\tfrac C4(e^{2nt}-1)-s\,\Ima z^n$ and
$R_{0n}=\tfrac g4|z^n-1|^2$. Insert
$e^{2nt}-1=2e^{nt}\sh nt$ and
$|z^n-1|^2=2e^{nt}(\ch nt-\cos n\omega)$.
Shifting by $j$ multiplies both expressions by $e^{2jt}$, and a minimizer
pair gives $T_{\star j}=Ce^{2jt}/4$, $R_{j\star}=ge^{2jt}/4$.
Orthogonal sums replace these expressions by their averages.
For a general measure use the real Hilbert space
$L^2([0,\pi],\nu;\C)$ and $e_k(\omega)=e^{k(t+i\omega)}$.
Theorem~\ref{thm:interpolation} now gives both directions.
\end{proof}

\subsection{Separation and trajectory growth}
The next lemma collects the topological facts used below.

\begin{lemma}[Compactness]\label{lem:compact}
For fixed $t$, the set of pairs $(\nu,C)$ feasible at $t$ is compact in the
product of the weak topology on probability measures on $[0,\pi]$ and the
real line. For Polyak's tuning the set of $(v,t,\nu,C)$ feasible in the sense
of \eqref{eq:VT} with $(v,t)$ in a compact set is compact, and feasibility
implies $|\sh v|\le1$.
\end{lemma}
\begin{proof}
Probability measures on $[0,\pi]$ are weakly compact.
The minimizer constraint gives $|C|\le\bar g\le1$, and every fixed-lag
constraint is closed because its integrand is continuous, jointly with the
parameters on compact sets. In the Polyak case,
$0\le\bar g=1-\int c^2d\nu-\sh^2v\le1-\sh^2v$.
\end{proof}

For the upper bound we allow arbitrary initial states. A segment consists of positions
$e_0,\dots,e_n$, gradients and values at $e_1,\dots,e_{n-1}$, and the
minimizer datum; it satisfies all interpolation inequalities among these data
and the recurrence \eqref{eq:normrec} for $1\le k<n$. Its state is
$S_k=(e_{k-1},e_k)$. Let
\begin{equation}\label{eq:Gamma}
 \Gamma(p)=\sup\frac{\norm{S_{k+p}}}{\norm{S_k}},\quad\Gamma(0)=1,
\end{equation}
over all segments, admissible $k$ and nonzero $S_k$.

\begin{lemma}[Growth envelope]\label{lem:envelope}
$\Gamma$ is finite and submultiplicative, the limit
$\rho_\Gamma=\lim_p\Gamma(p)^{1/p}=\inf_{p\ge1}\Gamma(p)^{1/p}$ exists,
$\rho_\Gamma\ge\rhoq>0$, and $W_N\le\sqrt2\,\Gamma(N)$ for $N\ge0$. If
$(\nu,C)$ is feasible at $t$, then $\Gamma(p)\ge e^{pt}$ for all $p$.
\end{lemma}
\begin{proof}
The minimizer inequality gives $\norm{\sigma_k}\le\norm{e_k}$, hence
\[
 \binom{\norm{e_{k+p-1}}}{\norm{e_{k+p}}}
 \le\begin{pmatrix}0&1\\ \beta&|A|+B\end{pmatrix}^{p}
 \binom{\norm{e_{k-1}}}{\norm{e_k}}.
\]
This proves finiteness. Restricting and concatenating ratios proves
submultiplicativity; a zero state remains zero. The submultiplicative
lemma gives the limit and infimum. Quadratic trajectories give
$\rho_\Gamma\ge\rhoq>0$, and the prescribed initial state has norm
$\sqrt2\norm{e_0}$. A feasible spectrum has
$\norm{S_k}^2=e^{2kt}(1+e^{-2t})$, proving the last assertion.
\end{proof}

\begin{lemma}[Energy estimate]\label{lem:energy}
Suppose \eqref{eq:strictcert} holds at the exponent $t$ with lags at most $M'$.
There are $K>0$ and an integer $w\ge M'+2$ such that for every segment and
$w\le i\le j\le n-w$,
\begin{equation}\label{eq:energy}
 \sum_{k=i}^j\mathsf b_k\le K(E_i+E_j),\quad \mathsf b_k=e^{-2kt}\norm{e_k}^2,\quad
 E_k=\sum_{|l-k|\le w}\mathsf b_l .
\end{equation}
\end{lemma}
\begin{proof}
Let $I^\pm_{k,n}=R_{k,k+n}\pm T_{k,k+n}$,
$I^\pm_{k\star}=R_{k\star}\pm T_{\star k}$ and form
\[
 Z_k=\sum_{n\le M'}2e^{-nt}
 (\alpha_nI^+_{k,n}+\beta_nI^-_{k,n})
 +4\gamma_+I^+_{k\star}+4\gamma_-I^-_{k\star}\ge0.
\]
After $f_k=e^{-kt}e_k$, $h_k=e^{-2kt}\psi_k$ and the recurrence,
$e^{-2kt}Z_k=\mathcal Q_k(f)+\mathcal H_k(h)$.
The quadratic form couples positions at most $M'+2$ steps apart and is
invariant under time shifts; its value on $f_k=e^{ik\omega}$ is
$\Phi_t(\omega)$. The value coefficients sum to zero by the second line of
\eqref{eq:strictcert}, so their sum has only endpoint terms.
The minimizer inequality gives $|h_k|\le\norm{f_k}^2/2$.
For every finitely supported Hilbert-valued sequence $F$, Parseval yields
\[
 \sum_k\mathcal Q_k(F)=\frac1{2\pi}\int_{-\pi}^{\pi}
 \Phi_t(|\theta|)\norm{\widehat F(\theta)}^2\,d\theta
 \le-\varepsilon\sum_k\norm{F_k}^2.
\]
Apply this to $F=f\mathbf1_{[i-1,j+M'+1]}$.
Every truncation error lies in fixed-width endpoint windows;
$|\ip{f_a}{f_b}|\le(\norm{f_a}^2+\norm{f_b}^2)/2$ bounds their sum by
$K_1(E_i+E_j)$. The value terms have the same bound, giving
\eqref{eq:energy}.
\end{proof}

\begin{lemma}[No strict certificate at the growth exponent]\label{lem:plateau}
There is no strict certificate at $t=\log\rho_\Gamma$.
\end{lemma}
\begin{proof}
Suppose there is one and put $r=\rho_\Gamma$,
$\gamma(p)=\Gamma(p)r^{-p}$, $\bar\gamma(p)=\max_{0\le l\le p}\gamma(l)$. By
Lemma~\ref{lem:envelope}, $\gamma(p)\ge1$ and $\bar\gamma(p)^{1/p}\to1$. For
$l\ge k\ge1$, $\norm{e_l}\le\norm{S_l}\le\Gamma(l-k)\norm{S_k}$, hence
\begin{equation}\label{eq:forward}
 \mathsf b_l\le c_r\gamma(l-k)^2(\mathsf b_{k-1}+\mathsf b_k),\quad c_r=\max(1,r^{-2}).
\end{equation}
Let $K,w$ be as in Lemma~\ref{lem:energy}, $\theta=1+1/(2K)$, $a=2w+1$, and
choose an integer $p_0$ with
\begin{equation}\label{eq:choosep}
 \theta^{p_0}>1+p_0ac_r\bar\gamma(p_0a+w)^2,
\end{equation}
which is possible by subexponential growth of $\bar\gamma$. Write $L'=p_0a$.

First, if $i\ge w$, $j\ge i+2w$ and $j\le n-L'$, then
\begin{equation}\label{eq:shortgrowth}
 B_i(j):=\sum_{k=i}^j\mathsf b_k\le2KE_i .
\end{equation}
Otherwise \eqref{eq:energy} gives $E_l>B_i(l)/(2K)$ for $j\le l\le n-w$. Since
$E_l=B_i(l+w)-B_i(l-w-1)$, this gives $B_i(l+w)\ge\theta B_i(l-w-1)$, and
iterating at $l=j+w+1+pa$, $0\le p<p_0$, yields $B_i(j+L')\ge\theta^{p_0}B_i(j)$.
On the other hand \eqref{eq:forward} with $k=j-w$ gives
$B_i(j+L')\le B_i(j)\{1+L'c_r\bar\gamma(L'+w)^2\}$, which contradicts
\eqref{eq:choosep}.

Second, put $n'=n-L'$ and $\mathsf T_i=\sum_{k=i}^{n'}\mathsf b_k$. For $w\le i\le n'-2w$,
\eqref{eq:shortgrowth} gives $\mathsf T_i\le2KE_i\le2K(\mathsf T_{i-w}-\mathsf T_{i+w+1})$, hence
$\mathsf T_{i+w+1}\le\theta^{-1}\mathsf T_{i-w}$, and $\mathsf T_{pa}\le\theta^{-p}\mathsf T_0$ as long as the
windows fit. Moreover $\mathsf T_0\le(1+2K)E_w\le C_w\norm{S_1}^2$, with $C_w$
depending only on $\Gamma(l)$, $l\le2w$, and $r$. The last such window ends at
distance at most $L'+4w+2$ from $n$; applying \eqref{eq:forward} over these
finitely many steps gives $\norm{S_n}^2\le Cr^{2n}\theta^{-n/a}\norm{S_1}^2$ with
$C$ independent of the segment and of $n$. Hence
$\rho_\Gamma\le r\theta^{-1/(2a)}<r$, a contradiction.
\end{proof}

\subsection{Initialization from rest}
\begin{lemma}[One-coordinate initialization]\label{lem:startup}
Let $(\nu,C)$ be feasible at $t$, with atoms $\omega_1,\dots,\omega_r$ and
weights $\nu_j>0$, and assume $|C|<\bar g$ if $t>0$. Then there are
$f\in\cF(\R^{2r+1})$ and $x_0$ with $x_{-1}=x_0$ such that
$\norm{x_k-x^\star}\ge c\,e^{kt}\norm{x_0-x^\star}$ for all $k\ge0$, with $c>0$
independent of $k$. The trajectory is the orthogonal sum of the $r$ waves in
$2r$ coordinates and a sequence $y_k\to0$ in the last coordinate.
\end{lemma}
\begin{proof}
Choose $\tau\in(-1,1)$ and let
$p_{-1}=p_0=1$, $p_{k+1}=(A+B\tau)p_k-\beta p_{k-1}$.
By \eqref{eq:Abound}, $p_k\to0$. Each $p_k$ is a nonconstant polynomial
in $\tau$ with leading coefficient $B^k$; avoid its countably many roots
of $p_k=1$. Then $\delta_0=\inf_{k\ge1}(1-p_k)^2>0$.
In $\C^r\oplus\R$ prescribe
\[
 \begin{aligned}
 e_0&=((\sqrt{\nu_j})_j,Y),&
 \sigma_0&=((\sqrt{\nu_j}(z_j-A+\beta)/B)_j,\tau Y),\\
 e_k&=((\sqrt{\nu_j}z_j^k)_j,p_kY),&
 \sigma_k&=((\sqrt{\nu_j}m_jz_j^k)_j,\tau p_kY)\quad(k\ge1),
 \end{aligned}
\]
with $\psi_0=\ip{\sigma_0}{e_0}/2$ and the stationary wave values plus
$\tau p_k^2Y^2/2$ for $k\ge1$. These data satisfy the recurrence from rest.
The scalar quadratic contributes $T=0$ and
$R=(1-\tau^2)(y_i-y_j)^2/4$ to every pair.
All pairs excluding zero already interpolate. The pair $(\star,0)$ gains
$(1-\tau^2)Y^2/4$; every $(0,k)$ gains at least
$(1-\tau^2)\delta_0Y^2/4$.
For $t\le0$ their wave deficits are uniformly bounded. For $t>0$, the
stationary margin is at least
$(\bar g-|C|)e^{2kt}/4-O(e^{kt})$, while altering the initial gradient
changes it by only $O(e^{kt})+O(1)$. Thus their negative parts remain
uniformly bounded in both cases. A single sufficiently large $Y$ makes
every pair interpolate. Theorems~\ref{thm:interpolation} and
\ref{lem:normalization} yield one $f$ and one trajectory, with
$\norm{e_k}\ge e^{kt}$ and $\norm{e_0}=\sqrt{1+Y^2}$.
\end{proof}

\subsection{Feasibility for Polyak's tuning}\label{app:spectral-polyak}
We turn to Polyak's tuning, treating the relative exponent $v$, also called
the twist, and the exponent $t$ as independent parameters. Here $(\nu,C)$
is feasible at $(v,t)$ if \eqref{eq:feasible} holds
with the data \eqref{eq:abg}.

\begin{lemma}[Symmetries and monotonicity in the twist]\label{lem:vmono}
For Polyak's tuning:
\begin{enumerate}[label=\textup{(\roman*)}]
\item $(t,C)\mapsto(-t,-C)$ at fixed $(v,\nu)$ preserves feasibility;
\item $(v,C)\mapsto(-v,-C)$ at fixed $(t,\nu)$ interchanges the lower and upper
constraints and preserves feasibility;
\item if $(\nu,C)$ is feasible at $(v,t)$ and $0\le v'\le v$, then $(\nu,\theta C)$
with $\theta=\sh v'/\sh v$ (and $\theta=1$ if $v=0$) is feasible at $(v',t)$.
\end{enumerate}
Consequently the set of $v$ feasible at $t$ is $[-V(t),V(t)]$, $V$ is even, and
$V\le\arsinh1$.
\end{lemma}
\begin{proof}
(i) $a_m$ and $g$ are even in $t$, $b_m$ does not depend on $t$, and
$\sh mt$ is odd. (ii) $a_m$ and $g$ are even in $u=\sh v$ and $b_m$ is odd.
(iii) The case $v=0$ is immediate. Otherwise, with $u'=\sh v'=\theta u$, $b_m(v')=\theta b_m(v)$ and
\[
 a_m(v')=a_m(v)+(u^2-u'^2)\bigl(T_m(G)-T_m(c)\bigr)\ge a_m(v),
\]
because $T_m(G)\ge1\ge T_m(c)$. Hence
$\ell_m(v')=\theta\ell_m(v)+(1-\theta)\bar a_m(v)+(\bar a_m(v')-\bar a_m(v))$, and
$\bar a_m(v)=\frac12(\ell_m+\upsilon_m)\ge0$; the same holds for $\upsilon_m$. The
minimizer slack satisfies $\bar g(v')\ge\bar g(v)\ge|C|\ge\theta|C|$. The
consequences follow from (ii), (iii), closedness, and
Lemma~\ref{lem:compact}.
\end{proof}

\begin{lemma}[Finite support and inactive tails]\label{lem:finiteatoms}
For Polyak's tuning let $T=T(q)\ne0$ and $v=T-\log q$. Every optimal pair
$(\nu,C)$ at $(v,T)$ has $\nu$ finitely supported in $(-1,1)$ and
$|C|<\bar g$, and only finitely many lag constraints are active.
\end{lemma}
\begin{proof}
The feasible spiral satisfies $\rs>q$ by Proposition~\ref{prop:spiral}
and the root calculation following its proof. Consequently
$T\ge\log\rs>\log q$, so $v=T-\log q>0$.
Divide the lag functions $a_m\mp b_m$ by $U_{m-1}(\ch t)$, so that $C$ enters
every normalized lag constraint with the coefficient $\pm\sh t$, and represent
the minimizer constraints, multiplied by $\sh|t|$, as the limits $m=\infty$:
for $t\ne0$ the normalized lag functions converge uniformly on $[-1,1]$, as
$m\to\infty$, to $g\sh|t|$. For $t>T$ the points $(t-\log q,t)$ are infeasible
by the definition of $T(q)$, and Theorem~\ref{thm:spectralgeneral}(a) gives
strict certificates. Normalized to total weight one, the certificates at a
sequence $t_n\downarrow T$ have a weakly convergent subsequence on the compact
index set $\{\mathsf L_m,\mathsf U_m:m\in\mathbb N\cup\{\infty\}\}$. The limit is a
nonpositive symbol $\Phi_T$, with equal total lower and upper weight, and
complementary slackness places every optimal measure in the zero set of
$\Phi_T$.

In the complex strip $|\Ima\omega|<|T|$ the trigonometric parts of the
normalized lag functions decay geometrically in $m$, so $\Phi_T$ is analytic
there. At $c=\pm1$ every normalized lag functional is at most
$-\sh^2v\,\sh|T|\tanh(|T|/2)<0$, so $\Phi_T$ is not identically zero; it has
finitely many zeros in $[-1,1]$, all interior. Thus $\nu$ is finitely supported
in $(-1,1)$.

Suppose $|C|=\bar g$. Part (i) of Lemma~\ref{lem:vmono} maps $(T,C)$ to
$(-T,-C)$, so we may take $T>0$. Put $u=\sh v$ and
$z_j=-\nu_j(g_j+2iu\sin\omega_j)$, so that
\[
 f(m):=\Rea\sum_jz_je^{im\omega_j}=\sum_j\nu_j\{2u\sin\omega_j\sin m\omega_j-g_j\cos m\omega_j\}.
\]
If $C=\bar g$, the upper constraints give $f(m)\ge-\bar ge^{-mT}$; if
$C=-\bar g$, the lower constraints give the same bound for the function
obtained by negating the sine term, that is, with $z_j$ replaced by its
conjugate, and the argument below applies verbatim. The atoms are distinct
points of $(0,\pi)$, so the Ces\`aro means satisfy
$N^{-1}\sum_{m\le N}f(m)\to0$ and
$N^{-1}\sum_{m\le N}f(m)^2\to\frac12\sum_j|z_j|^2$. Here
$|z_j|^2=\nu_j^2(g_j^2+4u^2\sin^2\omega_j)>0$: for $u\ne0$ the second term is
positive, and for $u=0$ one has $g_j=\sin^2\omega_j>0$. The lower bound makes
the Ces\`aro mean of the negative part $f^-$ tend to zero; together with
$N^{-1}\sum f(m)\to0$, the means of $|f|$ tend to zero, and since $f$ is
bounded the means of $f^2$ tend to zero, a contradiction. Hence
$|C|<\bar g$.

Finally, with $\eta=\bar g-|C|>0$ every raw lag slack is at least
\begin{equation}\label{eq:tail}
 \eta\ch(m|T|)-(1+u)^2,
\end{equation}
since $|g(c)|\le1+u^2$ and $|2u\sin\omega\sin m\omega|\le2u$; so only finitely
many lags are active.
\end{proof}

\begin{proof}[Proof of Theorem~\ref{thm:spectral}]
Lemma~\ref{lem:vmono} gives the description of the feasible $v$ and the
bound on $V$. Lemma~\ref{lem:compact} shows that the maximum defining $T(q)$
is attained; the spiral of Proposition~\ref{prop:spiral} is feasible, so
the set is nonempty. For $t>T(q)$, $|t-\log q|>V(t)$, so the actual point
$(t-\log q,t)$ is infeasible and a strict certificate exists by
Theorem~\ref{thm:spectralgeneral}(a); Theorem~\ref{thm:spectralgeneral}(b)
gives $\rho_\Gamma<e^{t_0}$ and $\rh<e^{t_0}$ for every $t_0>T(q)$, hence
$\rh\le e^{T(q)}$; together with Lemma~\ref{lem:envelope},
\begin{equation}\label{eq:freeidentity}
 \rho_\Gamma=e^{T(q)} .
\end{equation}
At $t=T(q)$
the actual point is feasible. If $T(q)\ne0$, Lemma~\ref{lem:finiteatoms} gives
a finitely supported optimal spectrum with $|C|<\bar g$, and
Lemma~\ref{lem:startup} gives $\rh\ge e^{T(q)}$. If $T(q)=0$, then
$q=\varphi^{-1}$ by Propositions~\ref{prop:strip} and~\ref{prop:spiral}, and the three-cycle of
Proposition~\ref{prop:spiral} is a one-atom feasible spectrum at $t=0$, so
Lemma~\ref{lem:startup} applies again. The remaining assertions are
Lemma~\ref{lem:finiteatoms}.
\end{proof}

\section{The spectral classification for Polyak's tuning}\label{app:classification}
Throughout,
the parameters are those of Polyak's tuning, $T_m$ and $U_m$ denote Chebyshev
polynomials, $\mathsf U_m$ denotes the normalized upper functional of
\eqref{eq:LU}, and $G=\ch t\ge1$, $u=\sh v$, $c=\cos\omega$. The raw slacks of
a spectrum are $\ell_m=\bar a_m-\bar b_m+d\,U_{m-1}(G)$ and
$\upsilon_m=\bar a_m+\bar b_m-d\,U_{m-1}(G)$, where $d=C\sh t$, so that
$d\,U_{m-1}(G)=C\sh mt$. All finite decimals in this appendix denote exact
rationals. The computer-assisted steps are exact polynomial identities,
Sturm counts and signs of rational numbers, and, for two arcs, interval
inclusions in outward-rounded arithmetic~\cite{Moore,Arb}. 

\subsection{Duality and the four contact systems}
Normalize the lag functionals and the constant $C$ by
\begin{equation}\label{eq:LU}
 \mathsf L_m(c)=\frac{a_m(c)-b_m(c)}{U_{m-1}(G)},\quad
 \mathsf U_m(c)=\frac{a_m(c)+b_m(c)}{U_{m-1}(G)},\quad d=C\sh t,
\end{equation}
so that the constraints read $\bar{\mathsf L}_m+d\ge0$ and $\bar{\mathsf U}_m-d\ge0$.
For nonnegative weights $w_m$, $m\in J$, with $\sum_{m\in J}w_m=1$, the
combination
\begin{equation}\label{eq:dualpoly}
 \Phi(c)=\mathsf L_1(c)+\sum_{m\in J}w_m\mathsf U_m(c),\quad Q(c)=-\frac{\Phi(c)}{G-c},
\end{equation}
cancels $d$, and $Q$ is a polynomial in $c$. Its structure is the basis of
the classification.

\begin{lemma}[Dual factorization]\label{lem:dualfactor}
For $G\ge1$ and $-1\le c\le1$,
\begin{equation}\label{eq:alphabeta}
 \Phi(c)=(G-c)\bigl\{(1-c^2-u^2)\,\alpha_J-2u(1-c^2)\,\beta_J\bigr\},
\end{equation}
where $\alpha_J\ge1$ and $\beta_J\ge0$ are divided differences of Chebyshev
polynomials that depend on $G$, $c$ and the weights and are independent of
$u$. In particular, for $u'\ge u\ge0$,
\[
 \Phi_{G,u'}(c)\le\Phi_{G,u}(c)-(G-c)(u'^2-u^2).
\]
\end{lemma}
The proof is given in Appendix~\ref{app:classification-spiral}.

Thus the dual symbol decreases strictly with $u=\sh v$. A feasible measure
supported on the zeros of a nonnegative $Q$ at the same point $(G,u)$ is
therefore optimal: increasing $u$ makes the certificate strict. This
primal--dual pairing reduces the boundary $u=\sh V(t)$ to finitely many
polynomial equations once the active lags are known.

\subsection{Polynomial contact systems}
Each family is described by polynomial contact equations: active
constraints are equalities, and interior zeros of the nonnegative
polynomial $Q$ also satisfy $Q'=0$.

\begin{enumerate}[label=\textup{(\roman*)},leftmargin=*]
\item The spiral family $S$, given by \eqref{eq:spiraldata}. Its active
constraints are $\mathsf L_1,\mathsf U_2,\mathsf U_3$ when $t\ne0$.
\item The single-frequency family $P_1$. Take $J=\{3\}$ in \eqref{eq:dualpoly}
and solve
\begin{equation}\label{eq:p1system}
 Q(c)=0,\quad Q'(c)=0,
\end{equation}
on the analytic branch with $u\to1$ and $c\to0$ as $G\to\infty$. Set
$\nu=\delta_c$ and $d=-\mathsf L_1(c)$.
\item The two-frequency family $P_2$. Take $J=\{2,3\}$ with weights $(W,1-W)$
and $\nu=(1-p)\delta_x+p\delta_y$, $x<y$, and solve
\begin{equation}\label{eq:p2system}
 \bar{\mathsf L}_1+d=0,\quad \bar{\mathsf U}_2-d=0,\quad \bar{\mathsf U}_3-d=0,\quad
 Q(y)=Q'(x)=Q'(y)=0 .
\end{equation}
Its values of $(G,u)$ lie on the curve
\begin{equation}\label{eq:p2poly}
 \begin{aligned}
 P_2(G,u)={}&256G^4(u^2-1)-512G^3u^3+768G^3u\\
 &-32G^2u^4-480G^2u^2-16Gu^3-u^4-16u^2=0,
 \end{aligned}
\end{equation}
which is rational:
\begin{equation}\label{eq:p2param}
 G=\frac{\zeta^4+4}{8\zeta(\zeta^2-2)},\quad u=\frac{(\zeta^2-2\zeta+2)^2}{\zeta^4-4}.
\end{equation}
\item The central family $P_5$. Take $J=\{2,3,5\}$ with weights
$(w_2,w_3,1-w_2-w_3)$ and solve
\begin{equation}\label{eq:p5system}
 \bar{\mathsf L}_1+d=0,\quad \bar{\mathsf U}_m-d=0\ (m=2,3,5),\quad Q(y)=Q'(x)=Q'(y)=0 .
\end{equation}
\end{enumerate}

In the two-frequency systems the condition $Q(x)=0$ follows from the primal
equations and $p<1$. Every coefficient is obtained from
\eqref{eq:abg}--\eqref{eq:dualpoly}.

\subsection{The spiral and the sharp threshold}\label{app:classification-spiral}
For one frequency the equations
$\ell_1=\upsilon_2=\upsilon_3=0$ are homogeneous and linear in
$(g,C,2u(1-c^2))$. With $\varrho=e^t$, their determinant factors as
\[
 \varrho^{-4}(\varrho^2-1)(\varrho^2-2c\varrho+1)^2
 \{c(\varrho^2+1)+\varrho\}.
\]
For $t\ne0$ it follows that $c=-1/(2G)$.
The first two equations give $C=g\tanh t$ and $g=u/G$.
Substitution into $g=1-c^2-u^2$ selects the positive root $u=1-1/(2G)$:
\begin{equation}\label{eq:spiralsimple}
 u=1-\frac1{2G},\quad c=-\frac1{2G},\quad g=\frac{2G-1}{2G^2},\quad C=g\tanh t .
\end{equation}

\begin{lemma}[Lag identities on the spiral]\label{lem:spiralids}
On the data \eqref{eq:spiralsimple}, for every $m\ge1$,
$\ell_m=gE_{m+1}(t)/G$ and $\upsilon_m=gE_{m-1}(t)/G$, where
$E_n(t)=\ch nt+2G^2\cos n\omega$. Moreover $E_0=1+2G^2$ and $E_1=E_2=0$.
\end{lemma}
\begin{proof}
The relations $g=u/G$, $c=-1/(2G)$ give
$b_m/g=\sqrt{4G^2-1}\sin m\omega$. Apply the addition formulas to
$\ch mt\pm\tanh t\sh mt=\ch(m\pm1)t/G$ and
$-\cos m\omega\mp\sqrt{4G^2-1}\sin m\omega
=2G\cos(m\pm1)\omega$. The values at $n=0,1,2$ follow immediately.
\end{proof}

\begin{lemma}[Positivity of the spiral kernel]\label{lem:spiralkernel}
Let $s=G^2-1$ and $F_n=G^{n-2}E_n$. For every $n\ge3$, $F_n$ is a polynomial
in $s$ with nonnegative coefficients, $F_n(0)=3$ if $3\mid n$ and $F_n(0)=0$
otherwise, and $F_n(s)>0$ for $s>0$. In particular $E_n(t)>0$ for $n\ge3$ and
$t\ne0$.
\end{lemma}
\begin{proof}
Put $X=G^2=1+s$. The generating function
$\sum_{n\ge0}T_n(x)r^n=(1-xr)/(1-2xr+r^2)$, applied with $x=G$ and $x=-1/(2G)$
and the variable $Gr$, gives
\[
 \sum_{n\ge0}F_nr^n=\frac{1-Xr}{X(1-2Xr+Xr^2)}+\frac{2+r}{1+r+Xr^2} .
\]
Since $F_1=F_2=0$ and $F_0=(1+2X)/X$, a direct computation gives
\[
 \sum_{n\ge3}F_nr^n=\frac{(2X+1)\,r^3\,(2X-1-Xr)}{(1+r+Xr^2)(1-2Xr+Xr^2)}
 =\frac{(2s+3)\,r^3\,\{(1-r)+s(2-r)\}}{D},
\]
where $D=(1+r+Xr^2)(1-2Xr+Xr^2)=D_0-sD_1-s^2D_2$ with
\[
 D_0=(1-r)(1-r^3),\quad D_1=r(2+3r^2-2r^3),\quad D_2=r^3(2-r).
\]
As formal power series in $r$ with coefficients in $\Q[s]$,
\[
 \frac{(1-r)+s(2-r)}{D}=\sum_{k\ge0}\Bigl(s\frac{D_1}{1-r}+s^2\frac{D_2}{1-r}\Bigr)^k
 \frac1{(1-r^3)^{k+1}}\Bigl(1+s\frac{2-r}{1-r}\Bigr),
\]
and the series converges coefficientwise because $D_1$ and $D_2$ have no
constant term. Each factor has nonnegative coefficients:
$D_1/(1-r)=2r+2r^2+5r^3+3\sum_{j\ge4}r^j$,
$D_2/(1-r)=2r^3+\sum_{j\ge4}r^j$, $(2-r)/(1-r)=2+\sum_{j\ge1}r^j$, and
$1/(1-r^3)=\sum_jr^{3j}$. Hence every $F_n$, $n\ge3$, has nonnegative
coefficients in $s$. The terms with $k\ge1$ are divisible by $s$. The term
$k=0$ contributes to $F_n$ the polynomial $(2s+3)\{[3\mid n]+s\,c_n\}$, where
$c_n$ is the coefficient of $r^{n-3}$ in $(2-r)/((1-r)(1-r^3))$, and $c_n\ge1$.
This gives the value at $s=0$ and strict positivity for $s>0$. Finally
$s=\sh^2t$.
\end{proof}

\begin{proof}[Proof of Proposition~\ref{prop:spiral}]
By Lemmas~\ref{lem:spiralids} and~\ref{lem:spiralkernel},
$\ell_1=\upsilon_2=\upsilon_3=0$, $\upsilon_1=g(1+2G^2)/G>0$, and all other
raw slacks are positive multiples of $E_n$, $n\ge3$, hence positive for
$t\ne0$. Moreover $g>0$ and $|C|=g|\tanh t|<\bar g=g$. This is
\eqref{eq:En} and the feasibility. On the spiral,
$1-\frac1{2G}=(\varrho^2-\varrho+1)/(\varrho^2+1)$ with $\varrho=e^t$, and
$\sh(t-\log q)=\frac12(\varrho/q-q/\varrho)$; equating the two and multiplying
by $2q\varrho(\varrho^2+1)$ gives $F(\varrho,q)=0$.
\end{proof}

The positive root of \eqref{eq:spiralpoly} is unique. Write
$F(\rho,q)=(\rho-q)^2(\rho^2+1)-2q\{q-(1-q)\rho^2\}$. For $\rho>q$ the first term
increases and the subtracted expression does not increase; at $\rho=q$ the
value is negative and at infinity it is positive. For $0<\rho\le q$ the first
term is smaller than $q^2(1+q^2)\le2q^2(1-q+q^2)$, which does not exceed the
subtracted expression. Thus the positive root $\rs$ is unique, exceeds $q$,
and satisfies $F_\rho>0$ and $F_q<0$, so $\rs$ is analytic and strictly
increasing in $q$. At $\rho=1$ the equation reads $2(1-q-q^2)=0$, so
$\rs=1$ exactly at $q=\varphi^{-1}$, that is, at $\kappa=\kappa_c$.

The next lemma controls all large lags at once.

\begin{lemma}[Amplitude identity and tail bound]\label{lem:tailsharp}
For $c=\cos\omega\in[-1,1]$ and $u\in\R$,
$g(c)^2+(2u\sin\omega)^2=(1-c^2+u^2)^2$. Consequently, for every spectrum and
every $m\ge1$,
\begin{equation}\label{eq:tailsharp}
 \ell_m,\ \upsilon_m\ \ge\ (\bar g-|C|)\ch mt-(\bar g+2u^2).
\end{equation}
\end{lemma}
\begin{proof}
With $k=1-c^2$, $(k-u^2)^2+4u^2k=(k+u^2)^2$. For one frequency,
$a_m\mp b_m=g\ch mt-\{g\cos m\omega\pm2u\sin\omega\sin m\omega\}$, and by the
identity and the Cauchy--Schwarz inequality the braces are at most $1-c^2+u^2$
in absolute value. Integrating gives
$\bar a_m\mp\bar b_m\ge\bar g\ch mt-(\bar g+2u^2)$, and
$|C\sh mt|\le|C|\ch mt$.
\end{proof}

\begin{proof}[Proof of Lemma~\ref{lem:dualfactor}]
Divided differences give
\[
 \alpha_J=1+\sum_mw_m\frac{T_m(G)-T_m(c)}{(G-c)U_{m-1}(G)},
 \qquad
 \beta_J=\sum_mw_m\frac{U_{m-1}(G)-U_{m-1}(c)}{(G-c)U_{m-1}(G)}.
\]
The quotients by $G-c$ are polynomials. Since
$T_m(G)\ge1\ge T_m(c)$ and $U_{m-1}(G)\ge m\ge U_{m-1}(c)$,
they have the required signs on $G\ge1$, $c\in[-1,1]$.
Subtracting the two symbols proves the monotonicity inequality.
\end{proof}

A feasible measure supported on the zeros of $Q\ge0$ is optimal.
Indeed increasing $u$ gives a strictly negative symbol when $G>1$.
At $G=1$ its upper bound is $-(1-c)S(c)$ with $\min S>0$.
Adding equal minimizer weights $\epsilon$ yields
\[
 -(1-c)S+2\epsilon g
 \le-(1-c)(\min S-4\epsilon)-2\epsilon u'^2<0
 \quad(0<4\epsilon<\min S),
\]
which gives a strict certificate there as well.

For the spiral and $J=\{2,3\}$, the weight that makes the zero of the symbol
at $c=-1/(2G)$ double is
\[
 W(G)=\frac{2(1+2G-2G^2)}{4G+1},\quad 1-W(G)=\frac{4G^2-1}{4G+1},
\]
and an exact factorization gives
\begin{equation}\label{eq:spiralcertificate}
 \Phi(c)=-\frac{(G-c)(2Gc+1)^2}{2G^3(4G+1)}\Bigl\{2G\Bigl(c-\frac{2G-1}{4G}\Bigr)^2+\frac{32G^3-36G^2+4G-1}{8G}\Bigr\}.
\end{equation}
Equivalently,
\[
 \Phi(c)=-\frac{(G-c)(2Gc+1)^2}{2G^3(4G+1)}\bigl\{2Gc^2-(2G-1)c+4G(G-1)\bigr\},
\]
with
$2Gc^2-(2G-1)c+4G(G-1)=2G\bigl(c-\frac{2G-1}{4G}\bigr)^2+\frac{32G^3-36G^2+4G-1}{8G}$.
The weight is nonnegative exactly for $G\le G_W=(1+\sqrt3)/2$, and the residual
quadratic is nonnegative on $[-1,1]$ exactly for $G\ge G_1$, the root of
$32G^3-36G^2+4G-1$ in $(1.03329,1.03330)$; its vertex $(2G-1)/(4G)$ lies in
$(0,1)$. Hence the spiral is optimal at fixed $G$ for $G_1\le G\le G_W$. At
$G_1$ the certificate acquires a second double zero, at
$G_W$ the weight of $\mathsf U_2$ vanishes; these are the two transitions
continued below.

\begin{proposition}[The spiral outside its regime]\label{prop:spiralsub}
Let $q\in(0,1)$, let $t_{\rm sp}=\log\rs(\kappa)\ne0$, and let $G=\ch t_{\rm sp}$.
If $G>G_W$ or $1<G<G_1$, then $\rh>\rs$.
\end{proposition}
\begin{proof}
Let $f=(\ell_1,\upsilon_2,\upsilon_3)$ at the spiral and
$\mathsf J=[\partial_\omega f,\partial_Cf]$.
The vector $y=(1,W/U_1(G),(1-W)/U_2(G))$ annihilates both columns.
Their independence follows from
\[
 \partial_\omega(\upsilon_2/U_1-\upsilon_3/U_2)
 =-\sin\omega\frac{(4G+1)(2G^2+1)}{2G^3(2G+1)}\ne0;
\]
thus $\operatorname{im}\mathsf J=y^\perp$.
For $G>G_W$, $W<0$, so the entries of $y$ have both signs. Choose
$r>0$ with $y^Tr=y^T\partial_tf$ and solve
$\mathsf J\Delta=r-\partial_tf$. Increasing $t$ and changing
$(\omega,C)$ in direction $\Delta$ makes all three contacts positive.
For $1<G<G_1$, $y>0$, but \eqref{eq:spiralcertificate} is positive at
$c_0=(2G-1)/(4G)$. Adding mass $\epsilon\tau$ at $c_0$ gives a direction
$d_0$ with $y^Td_0=\Phi(c_0)>0$. Choose $\epsilon$ so that
$y^T(\partial_tf+\epsilon d_0)>0$ and solve the same equation with a
positive $r$ and right side $r-\partial_tf-\epsilon d_0$.
All other fixed-lag and minimizer slacks stay positive. Since $t\ne0$,
\eqref{eq:tailsharp} handles the remaining lags uniformly in a
neighborhood. The perturbed measure is feasible at a larger exponent;
Theorem~\ref{thm:spectral} gives the strict inequality.
\end{proof}

The threshold is located by an inequality that combines two constraints.

\begin{proposition}[Threshold strip]\label{prop:strip}
Every point $(v,t)$ at which \eqref{eq:feasible} is feasible satisfies
\begin{equation}\label{eq:thresholdstrip}
 |v|\le|t|+\log\varphi,
\end{equation}
strictly if $t\ne0$. At $t=0$ the three-cycle $\nu=\delta_{-1/2}$, $C=0$, is
feasible at $v=\log\varphi$.
\end{proposition}
\begin{proof}
By symmetry take $v\ge0$ and put $\varrho=e^{|t|}$.
Add $|\sh t|$ times the appropriate minimizer inequality to $\ell_1$.
The result is $\int h_\varrho\,d\nu\ge0$, where
\[
 h_\varrho(c)=(1-c^2-u^2)(\varrho-c)-2u(1-c^2).
\]
It decreases with $u\ge0$. At
$u_\varrho=\sh(|t|+\log\varphi)$, substitute
$c=1-2X$, $\varrho=1+Y$. Write
$-4\varphi^2\varrho^2h_\varrho=A_3X^3+A_2X^2+A_1X+A_0$.
Expansion gives $A_3=16(3+\sqrt5)(1+Y)^2>0$;
$A_0$ and $4A_1A_3-A_2^2$ are polynomials in $Y$ with zero constant
term and positive remaining coefficients. Therefore
\[
 A_3X^3+A_2X^2+A_1X+A_0
 =A_3X\left(X+\frac{A_2}{2A_3}\right)^2
  +\frac{4A_1A_3-A_2^2}{4A_3}X+A_0
\]
is positive for $Y>0$, $X\in[0,1]$. This excludes $u\ge u_\varrho$
when $t\ne0$. For $t=0$, $u=1/2$ gives
$h_1(c)=-(1-c)(c+1/2)^2$; increasing $u$ makes it negative on
$[-1,1)$, forcing $\nu=\delta_1$, contrary to $\bar g\ge0$.
Finally $\delta_{-1/2}$, $C=0$ is feasible because its slacks are
$gE_n$, with $g=1/2$ and $E_n=1+2\cos(2\pi n/3)\in\{0,3\}$.
\end{proof}

\begin{proof}[Proof of Theorem~\ref{thm:threshold}]
If $q<\varphi^{-1}$, then $\log q<-\log\varphi$, and every $t\ge0$ has
$t-\log q>t+\log\varphi$, so by Proposition~\ref{prop:strip} the actual point
$(t-\log q,t)$ is infeasible. Hence $T(q)<0$. At $t=T(q)$,
$v=T-\log q\le|T|+\log\varphi=-T+\log\varphi$ gives $T\le\frac12\log(\varphi q)$,
so $\rh=e^{T(q)}\le\sqrt{\varphi q}<1$ by Theorem~\ref{thm:spectral}. If
$q=\varphi^{-1}$, the three-cycle is feasible at the actual point
$(\log\varphi,0)$, and for $t>0$ the actual point has $v=t+\log\varphi$, which
Proposition~\ref{prop:strip} excludes; thus $T(q)=0$ and $\rh=1$. If
$q>\varphi^{-1}$, then $\rs>1$, and the spiral is feasible at the actual point
with $t=\log\rs>0$, so $\rh\ge\rs>1$. The trajectory statements are the case
$\rs\ge1$ of \eqref{eq:spirallower}, proved in Appendix~\ref{app:finite}: at
$\kappa_c$ the spiral is the three-cycle, and the realized trajectory
consists of the cycle and a coordinate $2(-q^2)^{\lceil k/2\rceil}\to0$.
\end{proof}

\subsection{Continuation of the algebraic families}
The remaining three families have algebraic descriptions.
For $P_1$ write $c_1$ for the atom, $\nu=\delta_{c_1}$, and use the coordinates
$z=G^{-1}$ and
\begin{equation}\label{eq:p1blowup}
 u=1-\frac z2+z^2\mathcal A,\quad c_1=z^2\mathcal B .
\end{equation}
For $J=\{3\}$, multiplying $Q$ by its positive parameter denominator gives the
quartic
\begin{equation}\label{eq:p1quartic}
 Q_1(c)=-\bigl\{(1-c^2-u^2)\bigl[2+zc+z^2(c^2-1)\bigr]-2uz(1-c^2)(1+zc)\bigr\},
\end{equation}
and \eqref{eq:p1system} states that $c_1$ is a double root of $Q_1$. After
\eqref{eq:p1blowup}, both $Q_1(c_1)/z^2$ and $Q_1'(c_1)/z^2$ are polynomials in
$(z,\mathcal A,\mathcal B)$. At $z=0$ their solution is
$(\mathcal A,\mathcal B)=(\frac18,-\frac14)$, and the Jacobian in
$(\mathcal A,\mathcal B)$ is nonsingular there. For $P_5$ put $h=G-1$ and
\begin{equation}\label{eq:blowup}
 u=\tfrac12+h\mathcal A,\quad x=-\tfrac12+h\mathcal B,\quad p=hP,\quad d=hD,
\end{equation}
with unknowns $(\mathcal A,\mathcal B,y,P,D,w_2,w_3)$. Dividing each of the four
primal equations of \eqref{eq:p5system} by $h$ before clearing denominators,
and keeping the three dual equations, gives a polynomial system $F(h,Z)=0$,
$Z=(\mathcal A,\mathcal B,y,P,D,w_2,w_3)$, which is regular at $h=0$.

For the two continuations use the parametric Krawczyk operator
\cite[\S8.2]{Moore}. For a polynomial system $F(s,Z)=0$, parameter
interval $I$, centered box $X$, and nonsingular preconditioner $R$, check
\begin{equation}\label{eq:krawczyk}
 Z_0-RF(I,Z_0)+(\mathrm{Id}-RF_Z(I,X))(X-Z_0)
 \subset\operatorname{int}X,
 \qquad \|\mathrm{Id}-RF_Z(I,X)\|<1.
\end{equation}
The norm is a weighted infinity norm. The preconditioned fixed-point map
is a contraction on $X$ for every $s\in I$; its unique root is analytic.
At a shared endpoint, inclusion of the left root enclosure in the right
box identifies the branches. Derivatives follow from
$Z'=(\mathrm{Id}-\mathcal M)^{-1}(-RF_s)$, with
$\mathcal M=\mathrm{Id}-RF_Z$; the Neumann remainder is bounded by
$\|\mathcal M\|\|RF_s\|/(1-\|\mathcal M\|)$.

\begin{table}[ht]
\caption{Domains of the verified continuations. Each row uses $4000$
adjacent parameter slabs; the small extension beyond $G_a$ allows exact
endpoint identification.}
\label{tab:continuation}\centering
\begin{tabular}{llll}
\toprule
Family & Parameter interval & Unknowns & Equations\\
\midrule
$P_1$ & $0\le z\le733/1000$ & $(\mathcal A,\mathcal B)$ & \eqref{eq:p1system}, \eqref{eq:p1blowup}\\
$P_5$ & $0\le h\le0.023582$ & $(\mathcal A,\mathcal B,y,P,D,w_2,w_3)$ & \eqref{eq:p5system}, \eqref{eq:blowup}\\
\bottomrule
\end{tabular}
\end{table}

\begin{proposition}[The family $P_1$]\label{prop:p1}
The continued branch of $P_1$ on $0<z\le z_W=\sqrt3-1=G_W^{-1}$ consists of
feasible and optimal spectra: $\nu=\delta_{c_1}$ with $-z/2\le c_1<0$,
$\eta=g-|C|\ge z^3/4$, $Q_1(c)=(c-c_1)^2(q_2c^2+q_1c+q_0)$ with $q_2=z^2$ and
$q_0-|q_1|>0$, and all lag slacks nonnegative. At $z_W$ the branch equals the
spiral, $\mathcal A=0$, $\mathcal B=-1/(2z_W)$. Moreover $u$ increases in $G$,
$\frac12<u<1$, and $|V'|<1$ along the branch.
\end{proposition}
\begin{proof}
The continuation verifies, on every slab,
$\mathcal B<0$, $g/z>1/2$, $u<1$, $u_z<0$,
$(1/2+z\mathcal B)'<0$, $q_0-|q_1|>0$ and
$u_z^2z^2(1-z^2)<1+u^2$. Exact substitution at $z_W$ identifies the
spiral endpoint; the derivative signs give $-z/2\le c_1<0$ and
$1/2<u<1$. Put
$\mathsf D=2+zc_1+z^2(c_1^2-1)$,
$\mathsf E=1+zc_1+z^2(2c_1^2-1)$.
The contact equation gives
$g=2uz(1-c_1^2)(1+zc_1)/\mathsf D$ and
$d=2u(1-c_1^2)\mathsf E/\mathsf D$.
Thus $|C|\le g\sqrt{1-z^2}$ and
$\eta\ge gz^2/2\ge z^3/4$.
After multiplication by positive factors the raw slacks become
\begin{equation}\label{eq:p1slackpoly}
 z^{m-1}\{z(1+zc_1)[T_m(1/z)-T_m(c_1)]
 \pm\mathsf DU_{m-1}(c_1)\mp\mathsf EU_{m-1}(1/z)\}.
\end{equation}
The contacts vanish identically. For $m=2$ the upper expression is
$-2z(2c_1+z)(c_1z-1)\ge0$.
For the other slacks with $m\le12$, substitute $c_1=-\theta z/2$ and
remove $z^2$ from upper slacks with $m\ge4$; their tensor Bernstein
coefficients on $[0,3/4]\times[0,1]$ are positive. The basis functions
are nonnegative and sum to one, so this certifies the sign throughout the box.
For $m\ge13$, \eqref{eq:tail} gives the lower bound
$z^{3-m}/4-4>0$.
The nonnegative dual factor proves optimality, and
$V'=(-z^2u_z)\sh t/\sqrt{1+u^2}$ gives $|V'|<1$.
\end{proof}

\subsection{The rational two-frequency family}
The family $P_2$ is treated in closed form: its curve is rational, and all
primal and dual data are rational functions of one parameter.

\begin{lemma}[The rational curve]\label{lem:p2curve}
Let $G(\zeta)$ and $u(\zeta)$ be given by \eqref{eq:p2param}. Then
$P_2(G(\zeta),u(\zeta))\equiv0$, where $P_2$ is the polynomial \eqref{eq:p2poly},
and $u/G=8\zeta(\zeta^2-2\zeta+2)/\{(\zeta^2+2)(\zeta^2+2\zeta+2)\}$.
\end{lemma}
\begin{proof}
Substitution and clearing denominators prove both identities.
One way to recover the parametrization is to put $u=\sigma G$ in
$P_2=0$. The remaining equation is quadratic in $G^2$; its discriminant
has square-free part $\sigma^2-24\sigma+16$.
Parametrize this conic, then the resulting conic for $G$ by a line through
a rational point. Simplification gives \eqref{eq:p2param}.
\end{proof}

We use $\zeta^4+4=(\zeta^2-2\zeta+2)(\zeta^2+2\zeta+2)$ and the abbreviations
\[
 q_\pm=\zeta^2\pm2\zeta+2,\quad r_\pm=\zeta^2\pm2,\quad
 e_\pm=\zeta^2\pm2\zeta-2,\quad K_6=\zeta^6-4\zeta^5+4\zeta^4+8\zeta^2-16\zeta+8 .
\]

\begin{definition}[Data of the arc]\label{def:p2data}
For $\zeta\in\mathcal J=[1.6955,1.699]$ use \eqref{eq:p2param} and put
\[
 \begin{aligned}
 W&=\frac{q_-^2q_+(7\zeta^2-18\zeta+14)}{8r_+K_6},
 &1-W&=\frac{e_-^2e_+^2}{8r_+K_6},\\
 \mathsf a&=\frac{q_-^3}{2\zeta r_-r_+},
 &\mathsf b&=-\frac{2\zeta(3\zeta^4-10\zeta^3+16\zeta^2-20\zeta+12)}
 {r_-^2r_+}.
 \end{aligned}
\]
Let $x<y$ be the roots of $H(c)=c^2+\mathsf ac+\mathsf b$.
Write $L_0+L_1c$ and $U_0+U_1c$ for the remainders of
$\mathsf L_1(c)$ and $\mathsf U_2(c)$ modulo $H$. Set
\[
 \bar c=-\frac{L_0+U_0}{L_1+U_1},\qquad
 d=-L_0-L_1\bar c,\qquad
 p=\frac{\bar c-x}{y-x},\qquad \nu=(1-p)\delta_x+p\delta_y.
\]
The denominator $L_1+U_1$ stays nonzero on $\mathcal J$.
For $t>0$ put $C=d/\sqrt{G^2-1}$; the reflected branch has $-C$.
\end{definition}

For a polynomial $\phi(c)$, reduction modulo $c^2+\mathsf ac+\mathsf b$ gives
$\phi\equiv\phi_0+\phi_1c$ and $\int\phi\,d\nu=\phi_0+\phi_1\bar c$. Hence every raw
slack of $\nu$ is a rational function of $\zeta$.

\begin{lemma}[Algebraic identities on the arc]\label{lem:p2ids}
Where the rational expressions of Definition~\ref{def:p2data} are defined:
\begin{enumerate}[label=\textup{(\alph*)}]
\item $K_6=\zeta^4(\zeta-2)^2+8(\zeta-1)^2>0$; hence $0<W<1$ whenever $e_-e_+\ne0$.
\item Let $M=G(4G^2-1)\alpha_J$ and $N=G(4G^2-1)\beta_J$ for $J=\{2,3\}$ and weights
$(W,1-W)$, that is, $M=8G^3+4G^2c+4Gc^2-4G-W(4Gc^2-2G+c)$ and
$N=4G^2+4Gc-W(4Gc+1)$. Then
\[
 -\bigl\{(1-c^2-u^2)M-2u(1-c^2)N\bigr\}=4G(1-W)\,(c^2+\mathsf ac+\mathsf b)^2,
\]
and consequently $Q(c)=\frac{4(1-W)}{4G^2-1}(c-x)^2(c-y)^2\ge0$.
\item $\ell_1=\upsilon_2=\upsilon_3=0$ identically.
\end{enumerate}
\end{lemma}
\begin{proof}
Expand $\zeta^4(\zeta-2)^2+8(\zeta-1)^2$ to obtain $K_6>0$.
All remaining factors of $W$ are positive, and the formula for $1-W$
gives its sign. Clearing the parameter denominators in (b) gives a
polynomial identity in $(\zeta,c)$. For (c), reduce each slack modulo
$c^2+\mathsf ac+\mathsf b$ and replace $c$ by $\bar c$ in the affine
remainder. The three resulting rational functions vanish identically.
\end{proof}

The formulas of Definition~\ref{def:p2data} are obtained as follows. The
coefficient identities expressing that the quartic $G(4G^2-1)Q$ is a constant
times a square are two equations, linear in $W$, with the common root $W$ on
the curve of Lemma~\ref{lem:p2curve}. Then $\mathsf a,\mathsf b$ follow from the
cubic and quadratic coefficients, and $\bar c,d$ from the linear equations
$\ell_1=\upsilon_2=0$. The third primal equation holds automatically, by
complementary slackness: $\int\Phi\,d\nu$ equals $\ell_1+W\upsilon_2/U_1(G)+(1-W)\upsilon_3/U_2(G)$.

\begin{lemma}[Endpoints]\label{lem:p2ends}
\begin{equation}\label{eq:Ha}
 \begin{aligned}
 H_a(G)={}&162529280000G^{12}-238685519872G^{10}+77668708352G^8\\
 &-6246545408G^6-58821872G^4+12933112G^2+229441 .
 \end{aligned}
\end{equation}
Let $\zeta_1$ be the real root of $\zeta^3-\zeta^2-2$, and let
\begin{align*}
 N_{12}={}&25\zeta^{12}-314\zeta^{11}+1764\zeta^{10}-6092\zeta^9+14988\zeta^8-28736\zeta^7\\
 &+44832\zeta^6-57472\zeta^5+59952\zeta^4-48736\zeta^3+28224\zeta^2-10048\zeta+1600 .
\end{align*}
Then $1.6956<\zeta_1<1.6957$, and $N_{12}$ has exactly one root $\zeta_a$ in
$\mathcal J=[1.6955,1.699]$, with $1.6989<\zeta_a<1.699$; both roots are simple.
Moreover:
\begin{enumerate}[label=\textup{(\alph*)}]
\item $G(\zeta_1)=G_1$ and $u(\zeta_1)=1-1/(2G_1)$. With $G=G(\zeta)$,
\begin{gather*}
32G^3-36G^2+4G-1=\frac{(\zeta^3+2\zeta+4)(\zeta^3-\zeta^2-2)\,S_6}{16\zeta^3r_-^3},\\
u-1+\frac1{2G}=\frac{8(\zeta^3-\zeta^2-2)^2}{r_-r_+q_-q_+}\ \ge0,
\end{gather*}
with $S_6=\zeta^6-8\zeta^5-2\zeta^4+16\zeta^3+4\zeta^2-32\zeta-8$; the curve $P_2$
touches the spiral curve $u=1-1/(2G)$ at $\zeta_1$ and lies above it.
\item $G_a:=G(\zeta_a)$ is the root of $H_a$ in $(1.02358,1.02359)$: $N_{12}$
divides the numerator of $H_a(G(\zeta))$.
\end{enumerate}
\end{lemma}
\begin{proof}
Rational endpoint evaluations and monotonicity isolate $\zeta_1$.
For $N_{12}$, its values at $1.6989$ and $1.699$ have opposite signs,
and Lemma~\ref{lem:p2certs} gives $N_{12}'<0$ on $\mathcal J$.
The displayed factorizations prove (a). For (b), polynomial division
gives $N_{12}\mid\operatorname{num}H_a(G(\zeta))$; Sturm's theorem
isolates the corresponding root in $(1.02358,1.02359)$.
\end{proof}

\begin{lemma}[Sign certificates]\label{lem:p2certs}
Put $\zeta_c=1.69725$, $\delta=0.00175$.
The rational tests of Table~\ref{tab:p2signs} have their indicated
strict signs on $\mathcal J=[\zeta_c-\delta,\zeta_c+\delta]$.
After cancellation of the displayed endpoint factors, each numerator
and denominator polynomial $R$ satisfies the exact criterion
\[
 |R(\zeta_c)|>
 \sum_{k\ge1}|R^{(k)}(\zeta_c)/k!|\delta^k.
\]
\end{lemma}
\begin{proof}
Taylor's formula bounds $|R(\zeta)-R(\zeta_c)|$ by the displayed sum.
The criterion is a finite rational comparison for each numerator and
denominator obtained from Definition~\ref{def:p2data}; no approximation
to an algebraic root enters it.
\end{proof}

\begin{table}[t]
\caption{Rational tests for $P_2$. All quantities in the second column
are positive on $\mathcal J$. Moments are obtained by reducing modulo
$H(c)$ and replacing its affine remainder by its value at $\bar c$.}
\label{tab:p2signs}\centering
\begin{tabular}{ll}
\toprule
Property & Rational tests\\
\midrule
Parameter bounds & $-G_\zeta$, $G-1.0233$, $1.0337-G$, $u-0.511$, $0.517-u$\\
Root bounds & $H(-0.505)$, $-H(-0.48)$, $-H(0.25)$, $H(0.285)$\\
Positive mass & $-(\bar c+\mathsf a/2)$, $-R_p$\\
Dual weights & $W-0.362$, $0.366-W$\\
Minimizer margin & $d$, $\bar g-0.28$, $(\bar g-0.28)^2(G^2-1)-d^2$\\
Slope and endpoint & $du/dG$, $1-du/dG$, $-N_{12}'$\\
Finite lag slacks & The residual quantities of Table~\ref{tab:p2lags}\\
\bottomrule
\end{tabular}
\end{table}

\begin{theorem}[The arc $P_2$]\label{thm:p2arc}
For $\zeta\in[\zeta_1,\zeta_a]$ the data of Definition~\ref{def:p2data} have the
following properties.
\begin{enumerate}[label=\textup{(\roman*)}]
\item $G$ decreases from $G_1$ to $G_a$; $1.0233<G<1.0337$ and $0.511<u<0.517$.
\item $-0.505<x<-0.48$ and $0.25<y<0.285$.
\item $0\le p<1$, with $p=0$ exactly at $\zeta_1$.
\item $0.362<W<0.366$, and $Q\ge0$ on $[-1,1]$ with zeros exactly at $x$ and $y$.
\item $\bar g-|C|>0.28$.
\item $\ell_1=\upsilon_2=\upsilon_3=0$; $\upsilon_5>0$ for $\zeta<\zeta_a$ and
$\upsilon_5(\zeta_a)=0$; every other raw slack is positive.
\item $0<du/dG<1$.
\end{enumerate}
Consequently $V(t)=\arsinh u(\ch t)$ for $G_a\le\ch t\le G_1$, the measure $\nu$ is
optimal, and $|V'(t)|<1$ there.
\end{theorem}
\begin{proof}
Reduce polynomial moments modulo $c^2+\mathsf ac+\mathsf b$ and use
Lemma~\ref{lem:p2certs} for their rational signs.
The derivatives satisfy
\[
 G_\zeta=\frac{r_+e_-e_+}{8\zeta^2r_-^2}<0,
 \qquad
 \frac{du}{dG}=\frac{32\zeta^2q_-(\zeta^4-2\zeta^3-4\zeta+4)}
 {r_+^3e_-e_+}\in(0,1).
\]
The signs of the root polynomial at $-0.505,-0.48,0.25,0.285$ isolate
$x,y$. The identity
$\bar c^2+\mathsf a\bar c+\mathsf b
=(\zeta^3-\zeta^2-2)R_p$, with $R_p<0$, and
$\bar c+\mathsf a/2<0$ give $x\le\bar c<y$, with equality exactly
at $\zeta_1$.
The bounds on $G,u,W$ are the corresponding rational endpoint sign
tests. For the minimizer margin verify $d>0$, $\bar g>0.28$ and
$(\bar g-0.28)^2(G^2-1)>d^2$.
Lemma~\ref{lem:p2ids} proves the contacts and dual nonnegativity.
For $m\le9$, Table~\ref{tab:p2lags} and the signs of
Table~\ref{tab:p2signs} give every remaining slack; $\upsilon_5$ has
the sign of $N_{12}$. For $m\ge10$, \eqref{eq:tailsharp} gives
$\ell_m,\upsilon_m\ge0.28T_{10}(1.0233)-1.05>0.17$.
Complementary slackness proves optimality, and
$|V'|=|(du/dG)\sh t|/\sqrt{1+u^2}<1$.
\end{proof}

\begin{table}[t]
\caption{Finite-lag sign reduction on $P_2$. Generate every expression
from \eqref{eq:abg}, reduce modulo $H$, and apply
Lemma~\ref{lem:p2certs} after canceling the indicated factor.}
\label{tab:p2lags}\centering
\begin{tabular}{lll}
\toprule
Quantity & Factor removed & Sign of remainder on $\mathcal J$\\
\midrule
$\bar c^2+\mathsf a\bar c+\mathsf b$ & $\zeta^3-\zeta^2-2$ & $R_p<0$\\
$\ell_1,\upsilon_2,\upsilon_3$ & --- & Identically zero\\
$\upsilon_5$ & $N_{12}$ & Positive\\
Other slacks, $1\le m\le9$ & $1$ & Positive\\
\bottomrule
\end{tabular}
\end{table}

Theorem~\ref{thm:p2arc} covers both copies of $P_2$, one for each sign of $t$.
At $G_1$ the data of the arc reduce to the spiral,
\begin{equation}\label{eq:p2endpoint}
 \begin{gathered}
 u=1-\frac1{2G_1},\quad x=-\frac1{2G_1},\quad y=\frac{2G_1-1}{4G_1},\quad p=0,\\
 d=\frac{(2G_1-1)(G_1^2-1)}{2G_1^3},\quad W=\frac{2(1+2G_1-2G_1^2)}{4G_1+1},
 \end{gathered}
\end{equation}
and the six equations of \eqref{eq:p2system} vanish exactly modulo
$32G_1^3-36G_1^2+4G_1-1$. At $G_a$ the slack $\upsilon_5$ vanishes, and the data
of the arc, with $w_2=W$, $w_3=1-W$, $w_5=0$, solve the central system as well;
after the rescaling \eqref{eq:blowup}, a rigorous enclosure of this point is
contained in the continuation box of $P_5$, which identifies the two
branches.

\subsection{The central two-frequency family}
The central family is solved exactly at the critical point.

\begin{theorem}[Closed form at $\kappa_c$]\label{thm:p5h0}
At $h=0$ the system $F(0,Z)=0$ decouples. The three dual equations involve
only $(y,w_2,w_3)$, are linear in $(w_2,w_3)$, and consist of $-3w_2+22w_3=12$
and two equations whose $3\times3$ consistency determinant is
$900(2y+1)^3q_5(y)$, $q_5(y)=32y^5-40y^4+16y^3-6y^2-12y+3$. The primal block is linear in
$(\mathcal A,\mathcal B,D,P)$: its four equations are the coefficients
of $h$ in the four primal equations of \eqref{eq:p5system}, after
\eqref{eq:blowup}.
Let $y_0$ be the root of $q_5$ in
$(0.23188,0.23189)$. The primal coordinates are the unique solution of the four coefficient
equations just described at $y=y_0$; the dual coordinates are
\begin{align*}
 43243231\,w_5^{(0)}&=7851355-25581300y_0-5903400y_0^2-17737600y_0^3-41558400y_0^4,\\
\end{align*}
and $w_2^{(0)}=2/5-22w_5^{(0)}/25$, $w_3^{(0)}=3/5-3w_5^{(0)}/25$. Numerically $y_0\approx0.2318891$,
$\mathcal A_0\approx0.5146605$, $\mathcal B_0\approx-0.2585419$, $D_0\approx2.1051110$,
$P_0\approx2.5790145$, $w_2^{(0)}\approx0.3743478$, $w_3^{(0)}\approx0.5965020$ and
$w_5^{(0)}\approx0.0291502$. The Jacobian $F_Z(0,Z_0)$ is block triangular, with
the dual block in $(y,w_2,w_3)$ and the primal block in $(\mathcal A,\mathcal B,D,P)$,
and both block determinants are nonzero elements of $\Q(y_0)$.
\end{theorem}
\begin{proof}
Setting $h=0$ in $F$ gives seven explicit polynomials; the dual ones do not
contain $\mathcal A,\mathcal B,D,P$. The rest is linear algebra over
$\Q[y]/(q_5)$. Signs of elements of $\Q(y_0)$ are decided from
$q_5(0.23188)>0>q_5(0.23189)$, $q_5'<0$ near $y_0$, and the Taylor-dominance
criterion of Lemma~\ref{lem:p2certs} on $[0.23188,0.23189]$.
\end{proof}

\begin{proposition}[Dual nonnegativity at $\kappa_c$]\label{prop:p5h0dual}
At $h=0$ the dual polynomial factors as
$Q=(c+\frac12)^2(c-y_0)^2(q_2c^2+q_1c+q_0)$ with $q_i\in\Q(y_0)$,
$q_2\approx0.09328$, $q_1\approx-0.05001$, $q_0\approx0.65052$; in particular
$q_0-|q_1|-|q_2|>0$, so $Q\ge0$ on $[-1,1]$ with zeros exactly at $-\frac12$ and
$y_0$.
\end{proposition}
\begin{proof}
Division of $Q$ by $(c+\frac12)^2(c-y_0)^2$ in $\Q(y_0)[c]$ leaves no remainder,
and the signs are decided as in the proof of Theorem~\ref{thm:p5h0}.
\end{proof}

At $h=0$ the raw slacks of the three-cycle are $0$ or $\frac32$: for
$\theta_m=2\pi m/3$ the slack of sign $\sigma$ ($\sigma=-1$ lower, $\sigma=+1$
upper) equals $\frac12(1-\cos\theta_m)+\sigma\frac{\sqrt3}2\sin\theta_m$. It
vanishes for $m\equiv0$, for $m\equiv1$ with $\sigma=-1$, and for $m\equiv2$ with
$\sigma=+1$, the congruences being modulo $3$; we call these pairs $(m,\sigma)$
critical.

\begin{proposition}[First-order slacks]\label{prop:p5divslack}
For a critical pair $(m,\sigma)$ let $s_m^\sigma$ be the limit, as $h\downarrow0$,
of the raw slack divided by $h$. With
$Y_m^\sigma(y)=(\frac34-y^2)(1-T_m(y))+\sigma(1-y^2)U_{m-1}(y)$,
\begin{align*}
 m\equiv0:&\quad s_m^\sigma=\tfrac{m^2}2-\sigma m(\mathcal B_0+D_0)+P_0Y_m^\sigma(y_0),\\
 m\equiv1,\ \sigma=-1:&\quad s_m^-=\tfrac{m^2}2+(\mathcal B_0-3\mathcal A_0)-m\mathcal B_0+mD_0+P_0Y_m^-(y_0),\\
 m\equiv2,\ \sigma=+1:&\quad s_m^+=\tfrac{m^2}2+(\mathcal B_0-3\mathcal A_0)+m\mathcal B_0-mD_0+P_0Y_m^+(y_0).
\end{align*}
These vanish exactly for the active pairs $(1,-),(2,+),(3,+),(5,+)$ and are
positive for every other critical pair; the four smallest positive values,
in increasing order and rounded to two decimal places, are
$s_8^+\approx11.22$, $s_6^+\approx11.49$, $s_3^-\approx14.91$ and
$s_4^-\approx18.40$.
\end{proposition}
\begin{proof}
Differentiate the raw slack at $h=0$, using
$T_m(1+h)=1+m^2h+O(h^2)$,
$\omega_x=2\pi/3-2h\mathcal B_0/\sqrt3+O(h^2)$,
$g(x)=1/2+h(\mathcal B_0-\mathcal A_0)+O(h^2)$ and
$dU_{m-1}(G)=hD_0m+O(h^2)$.
The addition formulas in the three residue classes give the displayed
expressions. Reduction in $\Q[y_0]/(q_5)$ proves the four identities and
all remaining signs for $m\le7$.
For $m\ge8$, $|Y_m^\sigma|\le5/2$,
$|\mathcal B_0|+D_0<2.37$,
$|\mathcal B_0-3\mathcal A_0|<1.81$ and $P_0<2.58$ give
$s_m^\sigma\ge m^2/2-2.37m-8.26>0$.
The same bound for $m\ge10$, together with the exact comparisons for
$m\le9$, orders the four smallest positive slacks.
\end{proof}

Large lags along the central family are controlled by one estimate.

\begin{lemma}[Large lags near the three-cycle]\label{lem:p5tail}
Let $0<h\le0.0236$, $t=\arcosh(1+h)$, and suppose that a two-atom spectrum
with the parameters \eqref{eq:blowup} satisfies
\begin{equation}\label{eq:centralbounds}
 \begin{gathered}
 0<\mathcal A<0.55,\quad -0.27<\mathcal B<0,\quad 0<P<3,\quad 0<D<2.2,\\
 u<\tfrac{11}{20},\quad \bar g>0.49,\quad \bar g-|C|>\tfrac14,\quad
 |g(x)-\tfrac12|<h,\quad |\bar g-\tfrac12|<2h,\\
 |2u\sin\omega_x-\tfrac{\sqrt3}2|<2h,\quad |\arccos x-\tfrac{2\pi}3|<\tfrac h3 .
 \end{gathered}
\end{equation}
Then $\ell_m,\upsilon_m>0$ for all $m\ge24$.
\end{lemma}
\begin{proof}
For $mt\ge3$, \eqref{eq:tail} gives
$\ell_m,\upsilon_m\ge\ch3/4-(31/20)^2>0$.
For $mt\le3$, compare with the nonnegative three-cycle slacks.
The gain $\bar g(\ch mt-1)$ is at least $0.49m^2h$.
The main trigonometric term changes by at most $3h+mh/2$,
the second atom by at most $12h$, and the constant term by at most
$2h$. Finally
$|d|U_{m-1}(G)\le2.2hm\sh3/3<7.7mh$.
The resulting lower bound
$h(0.49m^2-9m-17)$ is positive for $m\ge24$.
The numerical inequalities follow from $20<e^3<21$.
\end{proof}

\begin{theorem}[Central family near $\kappa_c$]\label{thm:p5local}
There are $h_*>0$ and a real-analytic branch $Z(h)$ on a neighborhood
of $[-h_*,h_*]$, with
$Z(0)=Z_0$ and $F(h,Z(h))=0$, locally unique near $(0,Z_0)$.
For $0\le h\le h_*$ the corresponding spectrum
is feasible and optimal, so that
$V(t)=\arsinh\bigl(\frac12+(\ch t-1)\mathcal A(\ch t-1)\bigr)$ for $\ch t\le1+h_*$. In
particular $V$ is even and real-analytic at $t=0$, and $V(0)=\log\varphi$.
\end{theorem}
\begin{proof}
The nonsingular critical Jacobian gives the unique analytic branch by
the implicit function theorem. Positivity of $P_0,D_0$ and the dual
weights persists, as does the positive residual factor of $Q$.
The defining contacts stay zero. For every other $m\le23$, the slack
either tends to $3/2$ or, after division by $h$, tends to the positive
quantity in Proposition~\ref{prop:p5divslack}.
For the uniform tail, the divided limits in
\eqref{eq:centralbounds} are respectively
\[
 \mathcal B_0-\mathcal A_0,\quad
 \mathcal B_0-\mathcal A_0+P_0(1/4-y_0^2),\quad
 \sqrt3\mathcal A_0+\mathcal B_0/\sqrt3,\quad
 -2\mathcal B_0/\sqrt3.
\]
Their absolute values are less than $1,2,2,1/3$, respectively, by exact
comparison in $\Q(y_0,\sqrt3)$. Also $u\to1/2$, $\bar g\to1/2$ and
$C=O(\sqrt h)$. Thus \eqref{eq:centralbounds} holds on one common
interval, and Lemma~\ref{lem:p5tail} covers all $m\ge24$.
Duality proves optimality. Evenness in $t$ and $G=\ch t$ give the
analytic formula for $V$, with $V(0)=\arsinh(1/2)=\log\varphi$.
\end{proof}

\begin{proposition}[The central family]\label{prop:p5arc}
The continued branch of $P_5$ on $0\le h\le G_a-1$ consists of feasible and
optimal spectra, and $|V'|<1$ along it. At
$h=G_a-1$ it coincides with the endpoint of $P_2$.
\end{proposition}
\begin{proof}
The exact critical solution initializes the continuation, and
Theorem~\ref{thm:p5local} identifies it near zero. Each slab verifies
\eqref{eq:centralbounds},
$0.23<y<0.29$, $0.36<w_2<0.38$, $0.59<w_3<0.65$,
$w_5<0.031$, $w_5'<0$, $|u_G|<1$ and
$\sqrt{G^2-1}<\sqrt{1+u^2}$.
The quotient $Q/((c-x)^2(c-y)^2)$ satisfies
$q_0-|q_1|-|q_2|>0$. Inclusion of the exact endpoint of $P_2$ in the
last continuation box gives $w_5(G_a)=0$, hence $w_5>0$ before it.
All nonactive slacks with $m<24$ are positive. Near zero evaluate them
after division by $h$, retaining the exact constant at $h=0$ and
propagating the divided part through sums and products. For the main
phase, the mean value theorem encloses
$(\arccos x-2\pi/3)/h=-\mathcal B/\sqrt{1-\xi^2}$; addition formulas
and the entire function $\operatorname{sinc}$ then enclose its divided
trigonometric differences. Lemma~\ref{lem:p5tail} covers $m\ge24$.
Duality gives optimality, and
$V'=u_G\sh t/\sqrt{1+u^2}$ gives the derivative bound.
\end{proof}

\subsection{Regularity of the optimal rate}\label{app:classification-regularity}
\begin{proof}[Proof of Theorem~\ref{thm:regularity}, regularity]
The inverse of the continuous strictly increasing map $t-V(t)$ is
continuous and strictly increasing. On each open analytic arc its
derivative $1-V'$ is positive, so the analytic inverse function
theorem applies. The critical point belongs to the interior analytic
branch of Theorem~\ref{thm:p5local}; the other junctions give precisely
the five listed switches.
Proposition~\ref{prop:spiralsub} gives strict improvement over the
spiral everywhere else, and $\rs>q$.
The expansions follow in Appendix~\ref{app:finite}.
\end{proof}

\begin{figure}[tbp]
\centering
\includegraphics[width=\linewidth]{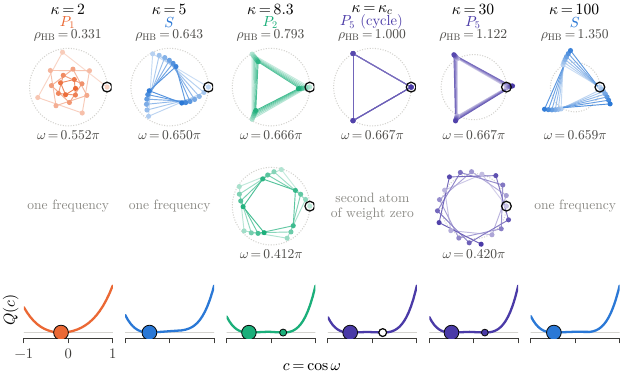}
\caption{Extremal spectra by regime. Top: the first nineteen points
in each frequency plane, with unit initial amplitude and displayed
radius $\rh^{k/12}$. Bottom: the dual polynomial $Q(c)\ge0$ and its
selected double zeros; filled marker size indicates spectral mass. An open marker denotes the zero-mass site at $\kappa_c$.}
\label{fig:spectra}
\end{figure}

\section{Finite-horizon bounds and asymptotic expansions}\label{app:finite}
This appendix completes the proofs of Theorems~\ref{thm:rate}
and~\ref{thm:regularity}. Throughout, $q\in(0,1)$ is fixed, the parameters are
those of Polyak's tuning, $T=T(q)=\log\rh$, and the recurrence is
$e_{k+1}=2q\sigma_k-q^2e_{k-1}$.

\subsection{A uniform lower bound from rest}\label{app:finite-lower}
We first prove the uniform lower bound \eqref{eq:spirallower}. All variables
below refer to the special spiral: $\rho=\rs$, $G=(\rho+\rho^{-1})/2$,
$\cos\omega=-\rho/(1+\rho^2)$, $z=\rho e^{i\omega}$,
$m=(z+q^2/z)/(2q)$, $g=1-|m|^2$, $C=g\tanh(\log\rho)$, and
$\mathsf h=(C+2\Rea m)/4$. In $\C\oplus\R$ prescribe
\begin{equation}\label{eq:explicitdata}
 \begin{gathered}
 x_0=(1,2),\quad\sigma_0=(m+\Delta,0),\quad\psi_0=\mathsf h,\\
 x_k=(z^k,y_k),\quad \sigma_k=(mz^k,0),\quad\psi_k=\mathsf h\rho^{2k}\quad(k\ge1),\\
 y_k=2(-q^2)^{\lceil k/2\rceil},\quad \Delta=\tfrac q2(1-z^{-1}).
 \end{gathered}
\end{equation}
Since $m+\Delta=(z+q^2)/(2q)$, the first step from $x_{-1}=x_0$ gives
$x_1$, and the later steps follow from $2qmz=z^2+q^2$ and
$y_{k+1}=-q^2y_{k-1}$. Once interpolation is proved,
$\norm{x_0}=\sqrt5$ and $\norm{x_k}\ge\rho^k$ give \eqref{eq:spirallower}.

Pairs of positive indices and their minimizer pairs satisfy
\eqref{eq:interpolation} by Proposition~\ref{prop:spiral} and
Lemma~\ref{lem:spectralinterp}, since the last coordinate carries zero
gradient and only adds $\frac14(y_i-y_j)^2$ to $R_{ij}$. Only the pairs with the
index $0$ remain. Write $X=\rho^k$, $Z=z^k$, and
\[
 \begin{gathered}
 \mathsf p_\pm=\tfrac12\bar\Delta(m\pm1),\quad
 \mathsf a_\mp=\frac{g\rho^{\mp1}}{4G},\quad
 \mathsf b_-=gGe^{-i\omega}+\mathsf p_+,\quad \mathsf b_+=gGe^{i\omega}+\mathsf p_-,\\
 \mathsf c_\mp=\frac{g\rho^{\pm1}}{4G}-\Rea\mathsf p_\pm-\frac{|\Delta|^2}4,\quad
 \mathsf d=(1-q^4)^2 .
 \end{gathered}
\]

\begin{lemma}[Slacks of the initial pairs]\label{lem:startupslack}
For $k\ge1$,
\begin{equation}\label{eq:startupslack}
 R_{0k}\mp T_{0k}=\mathsf a_\mp X^2+\Rea(\mathsf b_\mp Z)+\mathsf c_\mp
 +\bigl(1-(-q^2)^{\lceil k/2\rceil}\bigr)^2,
\end{equation}
and $|T_{\star0}|=|\frac12\Rea\Delta-\frac C4|$, $R_{0\star}=\frac14(5-|m+\Delta|^2)$. On
$F(\rho,q)=0$,
\begin{equation}\label{eq:startupidentities}
 \begin{gathered}
 g=\frac{2\rho(\rho^2-\rho+1)}{(\rho^2+1)^2}=\frac12-\frac{(\rho-1)^4}{2(\rho^2+1)^2},\\
 \Rea\Delta=\frac q2\,\frac{\rho^2+2}{\rho^2+1}\le q,\quad
 |\Delta|^2=\frac{q^2}4\,\frac{\rho^4+4\rho^2+1}{\rho^2(\rho^2+1)} .
 \end{gathered}
\end{equation}
In particular $0<g\le\frac12$ and $|C|<\frac12$.
\end{lemma}
\begin{proof}
Insert \eqref{eq:explicitdata} into $R_{0k}$ and $T_{0k}$, use
$\ip uv=\Rea(u\bar v)$ and collect the coefficients of
$X^2$, $Z$ and $1$. The last coordinate contributes
$(1-(-q^2)^{\lceil k/2\rceil})^2$.
The remaining identities follow by clearing denominators and reducing
modulo $F(\rho,q)$, which is quadratic in $q$.
\end{proof}

\begin{table}[ht]
\caption{Rational sign certificates for the initialized spiral. Every
inequality holds for both signs on the entire closed interval in $\rho$.}
\label{tab:startup}\centering\small
\begin{tabular}{p{.19\textwidth}p{.73\textwidth}}
\toprule
Interval in $\rho$ & Certified inequalities\\
\midrule
$[33/100,1]$ & $\mathsf c_\mp+\mathsf d>0$ and $(\mathsf c_\mp+\mathsf d)^2>|\mathsf b_\mp|^2$.\\[3pt]
$[1,6/5]$ & $|\mathsf b_\mp|^2<1$, $\mathsf a_\mp-\epsilon>0$, and $P_r^\mp(1)>0$.
For $(-,0),(-,1),(+,2)$, $(P_r^\mp)'(1)>0$; for the other three pairs
$(\sigma,r)$, $4(\mathsf a_\mp-\epsilon)(\mathsf c_\mp+\mathsf d)>(\mathsf b^\mp_r+\epsilon)^2$.\\[3pt]
$[6/5,1713/1000]$ & $4\mathsf a_\mp^2\rho^{20}>|\mathsf b_\mp|^2$,
$\mathsf z_\mp:=\mathsf a_\mp\rho^{20}+\mathsf c_\mp+\mathsf d>0$, and
$\mathsf z_\mp^2>|\mathsf b_\mp|^2\rho^{20}$. The exact slacks
\eqref{eq:startupslack} are positive for $1\le k\le9$.\\
\bottomrule
\end{tabular}
\end{table}

In Table~\ref{tab:startup}, $\mathsf b^\mp_r=\Rea(\mathsf b_\mp e^{ir\omega})$ for
$r=0,1,2$,
\[
 \epsilon(\rho)=\frac3{40}\,\frac{(\rho^2-1)(\rho+1)^2}{\rho(\rho^2+1)},\quad
 P_r^\mp(X)=(\mathsf a_\mp-\epsilon)X^2+(\mathsf b^\mp_r+\epsilon)X+\mathsf c_\mp+\mathsf d .
\]
Each quantity in the table is reduced modulo $F(\rho,q)$, which is quadratic in
$q$, to $(\mathcal P(\rho)+\mathcal R(\rho)q)/\mathcal D(\rho)$. If
$F=f_0+f_1q+f_2q^2$, the zeros of the numerator on either algebraic branch are
zeros of $f_0\mathcal R^2-f_1\mathcal P\mathcal R+f_2\mathcal P^2$. Sturm sequences
isolate these zeros and decide the signs of denominators. A simple resultant
zero that is assigned to the other branch by a sign change on its isolating
interval cannot also be a zero on the physical branch, since its multiplicity
would then be at least two. Signs at rational points reduce to comparisons of
squares in a quadratic number field. There are $44$ such checks.

\begin{proof}[Proof of \eqref{eq:spirallower}]
For $q\le1/5$, $q<\rho<(1+\sqrt2)q<1/2$ gives
$|\Delta|,|\mathsf p_\pm|<3/5$.
Altering the first gradient loses less than $99/100$ in any initial
pair, while the scalar coordinate gains at least
$(1-q^4)^2>99/100$. The minimizer pair has
$R_{0\star}>61/100>|T_{\star0}|$.
For $\kappa\ge2$, $q<\rho<2q$ gives
$R_{0\star}>11/16>|T_{\star0}|$.
For the other pairs use Table~\ref{tab:startup}:
if $\rho\le1$, then $X\le1$ and its first row gives positivity;
if $1\le\rho\le6/5$, write $k=3j+r$ and
$\omega=2\pi/3+\vartheta$.
The addition formulas and
$\log\rho\ge2(\rho-1)/(\rho+1)$ give
$|\vartheta|\le\epsilon\log\rho$, so
$|e^{3ij\vartheta}-1|\le\epsilon\log X\le\epsilon(X-1)$.
The slack is therefore at least $P_r^\mp(X)>0$ by the second row.
For $\rho\ge6/5$, bound the oscillatory term by
$-|\mathsf b_\mp|X$; the third row gives positivity for $k\ge10$
and checks the first nine lags.
The bounds $\rs(2)>33/100$ and $\rho_*<1713/1000$ show that these
intervals cover $\kappa\ge2$, overlapping the small-$q$ case.
Interpolation yields $f\in\cF(\R^3)$, and
$\norm{x_k}\ge\rho^k$, $\norm{x_0}=\sqrt5$ give the assertion.
\end{proof}

\subsection{Certificates and bounded-memory functionals}
We turn to the upper bound in \eqref{eq:bounds}. Its first ingredient is a
certificate with a gap linear in the distance to the optimal exponent. For a
weight vector $w$ and a point $(G,u)$ write $\Phi_{G,u}$ for the symbol
\eqref{eq:dualpoly}, and for $t\in\R$ let $u_q(t)=\sh(t-\log q)$ and
$u_V(t)=\sh V(t)$.

\begin{lemma}[Linear gap]\label{lem:lineargap}
There are $a>0$, $\delta_0>0$, and certificates \eqref{eq:strictcert} with lags at
most five at the exponents $t=T+\delta$, $0<\delta\le\delta_0$, whose coefficients are
bounded uniformly in $\delta$ and whose symbols satisfy
$\Phi_{T+\delta}\le-a\delta$ at every frequency.
\end{lemma}
\begin{proof}
Use at $t=T+\delta$ the boundary weights of
Theorem~\ref{thm:classification}, supported on $\{2,3,5\}$, and replace
$u_V=\sh V(t)$ by $u_q=\sh(t-\log q)$.
On the adjacent arc choose $\theta<1$ with $|V'|\le\theta$.
Then $u_q^2-u_V^2\ge(1-\theta)\delta$, since $u_V\ge1/2$.
Lemma~\ref{lem:dualfactor} gives
\[
 \Phi_{G,u_q}(c)\le-(G-c)(1-\theta)\delta.
\]
For $T\ne0$, take $\delta\le|T|/2$ so that
$G-c\ge\ch(T/2)-1>0$.
For $T=0$, add equal minimizer weights $b\delta$.
On $c\ge19/20$, $g\le-1/8$, giving the gap $b\delta/4$;
on $c\le19/20$, the gap is at least
$\delta((1-\theta)/20-2b)$.
Any $0<b\le(1-\theta)/80$ works. All multipliers stay bounded.
\end{proof}

Fix $M_0=5$. We record the inner products and function values in a fixed
number of consecutive iterates as a Gram window. This places the
limit $\delta\downarrow0$ in a finite-dimensional space.

Fix an exponent $t$. For positions $e_l$ and values $\psi_l$ put
$f_l=e^{-lt}e_l$ and $h_l=e^{-2lt}\psi_l$. A banded functional of width $D\ge0$
and offset $o\in\Z$ is
\begin{equation}\label{eq:banded}
 \mathcal B_j=\sum_{a,b=0}^Dc_{ab}\ip{f_{j+o+a}}{f_{j+o+b}}+\sum_{a=0}^Dd_ah_{j+o+a},
\end{equation}
with real coefficients $c_{ab}=c_{ba}$ and $d_a$. Its symbol, charge and size are
$\hat{\mathcal B}(\omega)=\sum_{a,b}c_{ab}\cos((a-b)\omega)$,
$\chi(\mathcal B)=\sum_ad_a$ and $|\mathcal B|=\sum_{a,b}|c_{ab}|+\sum_a|d_a|$.
The charge is the response to adding a constant to all scaled function
values; zero charge allows these terms to telescope. The
functional lies inside a segment $e_0,\dots,e_n$ if $1\le j+o$ and
$j+o+D\le n-1$. After substituting the recurrence, the scaled certificate
$e^{-2jt}Z_j$ has this form with offset $-1$, width $M'+2$, zero
charge and symbol $\Phi_t$.

For a probability measure $\nu$ on $[-1,1]$ and $C\in\R$, the realization of
Lemma~\ref{lem:spectralinterp} in $L^2([-1,1],\nu;\C)$, taken for all $l\in\Z$,
is the stationary sequence of $(\nu,C)$ at the exponent $t$: after removing
the factor $e^{lt}$, its pairwise inner products depend only on index differences. It satisfies the
recurrence for every $l$, $f_l(c)=e^{il\arccos c}$, and
$h_l=(C+2\int\Rea m\,d\nu)/4$; hence every banded functional takes on it the
value
\begin{equation}\label{eq:stationaryvalue}
 \mathcal B_j=\int\hat{\mathcal B}(\arccos c)\,d\nu(c)+\frac{\chi(\mathcal B)}4\Bigl(C+2\int\Rea m\,d\nu\Bigr),
\end{equation}
independent of $j$. The pair $(\nu,C)$ is $H$-feasible at $(v,t)$ if $|C|\le\bar g$
and the lag conditions of \eqref{eq:feasible} hold for $1\le m\le H$. By
Lemma~\ref{lem:spectralinterp}, the stationary sequence of an $H$-feasible pair
satisfies every minimizer inequality and every interpolation inequality
between indices at distance at most $H$.

\begin{lemma}[Telescoping with finite real filters]\label{lem:telescoping}
Let $D=M_0+2=7$ and let $\mathcal B$ be a banded functional of width $D$ and
offset $o$ with $\chi(\mathcal B)=0$. Suppose that, for some $\varepsilon\ge0$,
real coefficients $\varpi_{ra}$ give the finite factorization
\begin{equation}\label{eq:filterfactor}
 -\hat{\mathcal B}(\omega)-\varepsilon
 =\sum_{r=1}^R\Bigl|\sum_{a=0}^D\varpi_{ra}e^{ia\omega}\Bigr|^2 .
\end{equation}
For each integer $0\le a_0\le D$ there is a banded functional $\mathcal U$ of
width $D-1$ and offset $o$, with $|\mathcal U|\le c_D(|\mathcal B|+\varepsilon)$
for a constant $c_D$ independent of $R$ and of the chosen factors, such that
for all sequences in every real Hilbert space and every $j$,
\begin{equation}\label{eq:telescoping}
 \mathcal B_j=-\varepsilon\norm{f_{j+o+a_0}}^2-\sum_{r=1}^R\Bigl\|\sum_{a=0}^D\varpi_{ra}f_{j+o+a}\Bigr\|^2+\mathcal U_{j+1}-\mathcal U_j .
\end{equation}
\end{lemma}
\begin{proof}
Fourier orthogonality gives
$\sum_{r,a}\varpi_{ra}^2=-\sum_ac_{aa}-\varepsilon\le|\mathcal B|$.
Add the filter squares and
$\varepsilon\norm{f_{j+o+a_0}}^2$ to $\mathcal B_j$.
Its quadratic coefficient matrix $\tilde c$ now has zero symbol.
For each diagonal offset $e$, put
$\mu_{0,a}=\tilde c_{aa}$ and $\mu_{e,a}=2\tilde c_{a,a+e}$ for $e>0$.
Zero symbol means $\sum_a\mu_{e,a}=0$.
Each diagonal therefore telescopes by
\[
 \sum_a\mu_{e,a}g_{j+o+a}
 =R_{j+1}-R_j,\qquad
 R_j=\sum_a\mu_{e,a}\sum_{r=0}^{a-1}g_{j+o+r},
 \quad g_l=\ip{f_l}{f_{l+e}}.
\]
The zero-charge value terms telescope by the same identity.
Only indices up to $j+o+D-1$ occur. Finally
$\sum_{a,b}|\sum_r\varpi_{ra}\varpi_{rb}|
\le(D+1)\sum_{r,a}\varpi_{ra}^2$
gives a bound independent of $R$; for $D=7$, $c_D=63$ suffices.
\end{proof}

For a nonnegative trigonometric polynomial of degree at most seven,
the scalar Fej\'er--Riesz theorem gives a polynomial factor of degree at
most seven~\cite[\S1 and Theorem~2.1]{DR}. The symbols here are real and
even. If the factor is $p=u+iv$, with real-coefficient polynomials $u,v$,
averaging $|p(e^{i\omega})|^2$ at $\omega$ and $-\omega$ gives
$|u(e^{i\omega})|^2+|v(e^{i\omega})|^2$.
Thus \eqref{eq:filterfactor} holds with two real filters of width seven
for every certificate of Lemma~\ref{lem:lineargap}, after subtracting
its gap $\varepsilon=a\delta$.
The algebra behind the minimizer correction is also explicit. With
$\Delta=u'^2-u^2$, $g_{u'}=1-c^2-u'^2$ and equal minimizer weights
$\gamma\ge0$, subtraction of \eqref{eq:alphabeta} gives
\begin{equation}\label{eq:sourcefilter}
\begin{aligned}
 -\Phi_{G,u'}-2\gamma g_{u'}-\varepsilon
 ={}&-\Phi_{G,u}+(G-c)\bigl\{\Delta(\alpha_J-1)
       +2(u'-u)(1-c^2)\beta_J\bigr\}\\
 &+2\gamma(1-c)^2+(\Delta-4\gamma)(1-c)\\
 &+(G-1)\Delta+2\gamma u'^2-\varepsilon,
\end{aligned}
\end{equation}
The last terms are nonnegative for
$4\gamma\le\Delta$ and
$\varepsilon\le(G-1)\Delta+2\gamma u'^2$, as in
Lemma~\ref{lem:lineargap}. This retains the same width at the critical
point.

\subsection{The rate of bounded-memory trajectories}
An $H$-segment retains every minimizer inequality and the
interpolation inequalities of lags at most $H$. Its envelope $\Gamma_H$
is finite and submultiplicative, with
$\rho_H=\lim_p\Gamma_H(p)^{1/p}\ge e^T$.
For $H\le k\le n-1$ record the Gram matrix of
$e_{k-H},\ldots,e_{k+1}$ and the values
$\psi_{k-H+1},\ldots,\psi_k$ as $Y_k\in
\mathcal E_H=\mathrm{Sym}_{H+2}(\R)\times\R^H$.
The admissible windows form a closed convex cone $\mathcal K_H$:
positive semidefiniteness and every interpolation inequality are closed
convex conditions on these entries. The minimizer inequalities imply
$|\psi_l|\le\norm{e_l}^2/2$, hence
$|\mathcal V(Y)|\le|\mathcal V|\operatorname{tr}Y$ for a linear
functional, with coefficient norm $|\mathcal V|$.
The envelope and the old positions already in $Y_k$ give
\begin{equation}\label{eq:windowgrowth}
 \operatorname{tr}Y_{k+p}\le(H+2)\max_{\max(0,p-H)\le r\le p+1}\Gamma_H(r)^2\operatorname{tr}Y_k
 \quad(p\ge0,\ k+p\le n-1).
\end{equation}

\begin{lemma}[Plateau for classes of segments]\label{lem:plateauclass}
Let $\mathcal G$ be a class of $H$-segments in Euclidean spaces $\R^d$, over all
finite $d$, closed under restriction to
consecutive indices, and let $\Gamma_{\mathcal G}$ be defined as $\Gamma$, with the
supremum over $\mathcal G$. Assume that $\Gamma_{\mathcal G}(p)>0$ for every $p$ and
that $r=\lim_p\Gamma_{\mathcal G}(p)^{1/p}>0$. Then there is no banded functional
$\mathcal B$ at the exponent $\log r$ with $\chi(\mathcal B)=0$, $\hat{\mathcal B}<0$ on
$[0,\pi]$, and $\mathcal B_j\ge0$ whenever $\mathcal B_j$ lies inside a member of
$\mathcal G$.
\end{lemma}
\begin{proof}
If $\hat{\mathcal B}\le-\varepsilon<0$, localization and Parseval give
\eqref{eq:energy}, exactly as in Lemma~\ref{lem:energy}.
Zero charge cancels the interior values; fixed bandwidth and
$|h_l|\le\norm{f_l}^2/2$ control all endpoint terms.
Closure under restriction makes $\Gamma_{\mathcal G}$ submultiplicative;
a zero state stays zero by the minimizer inequality.
Thus $\Gamma_{\mathcal G}(p)\ge r^p$, and the block-energy argument of
Lemma~\ref{lem:plateau} applies to this class without any change.
It forces growth strictly below $r$, a contradiction.
\end{proof}

\begin{lemma}[Continuation]\label{lem:continuation}
Every $Y\in\mathcal K_H$ is the window $Y_H$ of an infinite segment, and hence of
an infinite $H$-segment.
\end{lemma}
\begin{proof}
Realize the Gram matrix by $e_0,\ldots,e_{H+1}$ and set
$\sigma_l=(e_{l+1}+q^2e_{l-1})/(2q)$ for $1\le l\le H$.
These data and the minimizer datum satisfy every pairwise interpolation
inequality. Theorem~\ref{thm:interpolation} supplies a single $\psi\in\cS$.
Continue the recurrence with $\sigma_l=\nabla\psi(e_l)$ and
$\psi_l=\psi(e_l)$; all subsequent inequalities hold by necessity.
\end{proof}

\begin{lemma}[Bounded memory preserves the rate]\label{lem:relaxedrate}
There is $H_0\ge M_0+2$ such that $\rho_H=e^T$ for every $H\ge H_0$.
\end{lemma}
\begin{proof}
Let $t_H$ be the largest exponent admitting an $H$-feasible pair at
$(t-\log q,t)$. The minimizer inequality bounds
$|t-\log q|\le\arsinh1$. Compactness gives attainment, and a weak limit
of maximizers shows $t_H\downarrow T$ as $H\to\infty$.
Choose $H_0\ge M_0+2$ with $t_{H_0}<T+\delta_0$.
For $H\ge H_0$, the multipliers of Lemma~\ref{lem:lineargap} exclude
every $t\in(T,T+\delta_0]$ using only lags at most five; larger
exponents exceed $t_H$. If $\rho_H>e^T$, finite-dimensional separation
at $t=\log\rho_H$ supplies a strict certificate using lags at most
$H$. Its local functional has zero charge and is nonnegative on every
$H$-segment, contradicting Lemma~\ref{lem:plateauclass}.
\end{proof}

\subsection{A limiting functional at the exact rate}
\begin{lemma}[Functionals above the rate]\label{lem:prices}
Let $H\ge H_0$. For $0<\delta\le\delta_0$ there are linear functionals $\mathcal V_\delta$
on $\mathcal E_H$ with $\sup_\delta|\mathcal V_\delta|<\infty$ such that, with $t=T+\delta$:
\begin{enumerate}[label=\textup{(\roman*)}]
\item for every $H$-segment and $H\le k\le n-2$,
$\mathcal V_\delta(Y_k)-e^{-2t}\mathcal V_\delta(Y_{k+1})\ge\delta\norm{e_{k-M_0}}^2$;
\item for every infinite $H$-segment and $k\ge H$,
$\mathcal V_\delta(Y_k)\ge\delta\sum_{i\ge0}e^{-2it}\norm{e_{k+i-M_0}}^2$.
\end{enumerate}
In particular, $\mathcal V_\delta\ge0$ on $\mathcal K_H$.
\end{lemma}
\begin{proof}
The certificate gives a nonnegative banded functional
$\mathcal B_j=e^{-2jt}Z_j$ of offset $-1$, width $M_0+2$, zero charge,
and symbol at most $-a\delta$. Its coefficients are uniformly bounded.
Apply Lemma~\ref{lem:telescoping} with $\varepsilon=a\delta$, $a_0=1$:
$\mathcal U^\delta_{j+1}-\mathcal U^\delta_j
\ge a\delta\norm{f_j}^2$.
Define
$\mathcal V_\delta(Y_k)=-a^{-1}e^{2(k-M_0)t}
\mathcal U^\delta_{k-M_0}$.
Since $H\ge M_0+2$, all its indices lie in $Y_k$ and its coefficients
are uniformly bounded. Multiplication of the preceding inequality gives
(i). Iteration gives the partial sums in (ii).
By Lemma~\ref{lem:relaxedrate},
$\Gamma_H(p)\le K_\delta e^{p(T+\delta/2)}$; hence
\eqref{eq:windowgrowth} makes the discounted terminal term tend to zero.
Lemma~\ref{lem:continuation} then gives nonnegativity on the entire cone.
\end{proof}

\begin{lemma}[A functional at the exact rate]\label{lem:exactprice}
Let $H\ge H_0$. There is a linear functional $\mathcal V_*$ on $\mathcal E_H$ such that
\begin{enumerate}[label=\textup{(\alph*)}]
\item $\mathcal V_*\ge0$ on $\mathcal K_H$;
\item $\mathcal V_*(Y_{k+1})\le e^{2T}\mathcal V_*(Y_k)$ for every $H$-segment and
$H\le k\le n-2$;
\item if $(\nu,C)$ is $H$-feasible at $(T-\log q,T)$, the windows of its
stationary sequence at the exponent $T$ satisfy
$\mathcal V_*(Y_k)\ge\frac12e^{2(k-M_0)T}>0$ for $k\ge H$.
\end{enumerate}
\end{lemma}
\begin{proof}
By the uniform bound in Lemma~\ref{lem:prices}, there are $\delta_n\downarrow0$ such
that $\mathcal V_{\delta_n}$ converges to a linear functional $\mathcal V_*$. Properties
(a) and (b) follow from Lemma~\ref{lem:prices} in the limit. For (c), the
stationary sequence restricted to $l\ge0$ is an infinite $H$-segment with
$\norm{e_l}^2=e^{2lT}$, and Lemma~\ref{lem:prices}(ii) gives
$\mathcal V_\delta(Y_k)\ge\delta\sum_{i\ge0}e^{-2i(T+\delta)}e^{2(k+i-M_0)T}
=\frac{\delta}{1-e^{-2\delta}}e^{2(k-M_0)T}\ge\frac12e^{2(k-M_0)T}$; the bound passes to
the limit.
\end{proof}

Fix $H\ge H_0$ and let $\mathcal N=\{Y\in\mathcal K_H:\mathcal V_*(Y)=0\}$. By (a) and (b),
$Y_k\in\mathcal N$ implies $Y_{k+1}\in\mathcal N$. An $\mathcal N$-segment is an $H$-segment
all of whose windows lie in $\mathcal N$. These segments form a class closed
under restriction to consecutive indices; let $\Gamma_{\mathcal N}\le\Gamma_H$ be its
growth function and $\rho_{\mathcal N}=\lim_p\Gamma_{\mathcal N}(p)^{1/p}$, with
$\rho_{\mathcal N}=0$ if $\Gamma_{\mathcal N}(p)=0$ for some $p$.

\begin{lemma}[A smaller rate on the kernel]\label{lem:kernel}
$\rho_{\mathcal N}<e^T$.
\end{lemma}
\begin{proof}
Suppose $\rho_{\mathcal N}=e^T$.
For stationary windows at $T$, linearity gives
$\mathcal V_*(Y_H^{\nu,C})=\int\vartheta\,d\nu+\vartheta_0C$
with continuous $\vartheta$.
Lemma~\ref{lem:exactprice}(c) excludes the equation
$\mathcal V_*(Y_H^{\nu,C})=0$ from the compact $H$-feasible set.
Separate the finite system of lag and minimizer inequalities together
with this one equality, as in
Theorem~\ref{thm:spectralgeneral}(a). This gives a nonnegative
combination $Z_j$ of interpolation slacks and a real multiplier $\zeta$
for the equality, whose combined frequency symbol is strictly negative
and whose value coefficient cancels.
Thus
$\mathcal B_j=e^{-2jT}(Z_j+\zeta\mathcal V_*(Y_{j+H}))$
has zero charge, negative symbol, and is nonnegative inside every
$\mathcal N$-segment. Its symbol and charge can be read from
\eqref{eq:stationaryvalue} for $\nu=\delta_c$ and arbitrary $C$.
Lemma~\ref{lem:plateauclass} contradicts the assumed rate.
\end{proof}

\begin{lemma}[Contraction near the kernel]\label{lem:nearkernel}
There are an integer $p_*\ge H$ and a number $\eta>0$ such that, if an
$H$-segment and an index $k\ge H$ satisfy $k+p_*\le n-1$ and
$\mathcal V_*(Y_k)\le\eta\operatorname{tr}Y_k$, then
$\operatorname{tr}Y_{k+p_*}\le\frac12e^{2p_*T}\operatorname{tr}Y_k$.
\end{lemma}
\begin{proof}
Choose $\rho_{\mathcal N}<r<e^T$ and
$\Gamma_{\mathcal N}(p)\le K_rr^p$.
For an initial window in $\mathcal N$, its forward invariance and
\eqref{eq:windowgrowth} give
\[
 \operatorname{tr}Y_{k+p}\le(H+2)K_r^2\max(1,r)^{2(H+1)}
 r^{2(p-H)}\operatorname{tr}Y_k\quad(p\ge H).
\]
Choose $p_*$ so this coefficient is at most $e^{2p_*T}/4$.
If no $\eta>0$ works, normalize failing segments to
$\operatorname{tr}Y_H=1$ and restrict them to length $H+p_*+1$.
The comparison bound and $|\psi_l|\le\norm{e_l}^2/2$ bound all their
Gram entries and values. A convergent subsequence therefore yields an
admissible segment with $Y_H\in\mathcal N$ but
$\operatorname{tr}Y_{H+p_*}\ge e^{2p_*T}/2$, contradicting the bound
just proved.
\end{proof}

\subsection{Uniform finite-horizon bounds}\label{app:finite-uniform}
\begin{proof}[Proof of Theorem~\ref{thm:rate}]
Fix $H\ge H_0$ and let $p_*,\eta$ be from
Lemma~\ref{lem:nearkernel}. Put $\mathsf A=|\mathcal V_*|/\eta$;
choose $\mathsf B\ge1$ bounding the discounted window growth over at
most $p_*$ steps by \eqref{eq:windowgrowth}.
For $\tilde s_l=e^{-2(l-k)T}\operatorname{tr}Y_l$, a window with
$\mathcal V_*(Y_l)>\eta\operatorname{tr}Y_l$ satisfies
$\tilde s_l\le\mathsf A\operatorname{tr}Y_k$.
At every other window, Lemma~\ref{lem:nearkernel} gives
$\tilde s_{l+p_*}\le\tilde s_l/2$.
Advance by one step in the first case and by $p_*$ steps in the second;
bound the intervening windows by $\mathsf B$. This proves
\begin{equation}\label{eq:constantwindow}
 \operatorname{tr}Y_l\le\mathsf B\max(1,\mathsf A\mathsf B)
 e^{2(l-k)T}\operatorname{tr}Y_k\quad(H\le k\le l\le n-1).
\end{equation}
A trajectory from rest has initial window trace at most
$c_H\norm{e_0}^2$, by the comparison matrix.
Thus \eqref{eq:constantwindow} gives the upper bound uniformly in
$f,d,N$; the finitely many smaller horizons are covered by the same
comparison bound.
For the lower bound choose the at-most-two-atom optimal spectrum,
using the three-cycle if $T=0$, and apply Lemma~\ref{lem:startup}.
Embedding unused coordinates with $\mu\norm w^2/2$ places the example
in $\R^5$. The same function and start work at every horizon.
Taking $N$th roots proves both limits.
The three-dimensional bound \eqref{eq:spirallower} follows from the
explicit construction above.
\end{proof}

The constant $C_\kappa$ is determined by the number $\eta$ of
Lemma~\ref{lem:nearkernel}, which is produced by a compactness argument. The same argument gives
$\Gamma(p)\le C\rh^p$ for segments with arbitrary initial states.

\subsection{Asymptotic expansions}\label{app:finite-expansions}
\begin{proof}[Proof of Theorem~\ref{thm:regularity}, expansions]
At the threshold $V(t)=\log\varphi+O(t^2)$, so
$T=\log(q/q_c)+O(\log^2(q/q_c))$, $q_c=\varphi^{-1}$.
Since $d\log q/d\kappa=1/(\sqrt\kappa(\kappa-1))$ equals
$1/(4\kappa_c)$ there, exponentiation gives \eqref{eq:critical}.
For $q\downarrow0$, expand the regular system
\eqref{eq:p1system}, \eqref{eq:p1blowup} at $z=0$:
$u_1=1-z/2+z^2/8-z^3/4+O(z^4)$.
Substitute $z=2\rho/(1+\rho^2)$ into
$\rho/q=u+\sqrt{1+u^2}$ and compare coefficients.
The spiral uses $u=1-z/2$; subtraction gives the stated cubic gap.
Finally $q\uparrow1$ lies on the spiral, and its unique positive root
converges to the unique positive root of $F(\rho,1)$.
\end{proof}

\section{General tunings and rate certificates}\label{app:general}
Throughout, \eqref{eq:quad} holds and
$A=1+\beta-(x+y)/2$, $B=(y-x)/2>0$.
The roots of $a,b$ lie in the unit disk by the quadratic Schur criterion;
$a-b=(x-y)z$, $a(1)=x$ and $b(1)=y$.
\begin{lemma}[Dominant quadratic roots]\label{lem:dominant}
Suppose $\rhoq>\sqrt\beta$. If $\rho_a\ge\rho_b$, the roots of $a$ are real and
$\rho_a$ is a positive root of $a$. If $\rho_b\ge\rho_a$, the roots of $b$ are
real and distinct, $-\rho_b$ is a root of $b$, and the other root is
$-\beta/\rho_b\in(-\rho_b,0]$.
\end{lemma}

\begin{proof}
The root radius of $z^2-pz+\beta$ is $\sqrt\beta$ for
$|p|\le2\sqrt\beta$ and
$(|p|+\sqrt{p^2-4\beta})/2$ otherwise. The latter is strictly
increasing in $|p|$. Since $1+\beta-x>1+\beta-y$, the dominant
endpoint has the asserted sign; Vieta's formula gives its other root.
\end{proof}
\begin{lemma}[The angle at the threshold]\label{lem:angle}
For $\omega\in(0,\pi)$,
\begin{equation}\label{eq:Xi}
 \cot\Theta(\omega)=\frac{\Xi(\cos\omega)}{\sin\omega},\quad
 \Xi(c)=\frac{B^2-(1-\beta)^2-A^2+2A(1+\beta)c-4\beta c^2}{2B(1-\beta)} .
\end{equation}
The polynomial $\Xi$ has degree at most two and is concave, $\Xi(1)<0$ and $\Xi(-1)<0$; consequently
$\Theta(\pi)=\pi$ and $\Theta(\omega)\to\pi$ as $\omega\to0$.
\end{lemma}
\begin{proof}[Proof of Lemma~\ref{lem:angle}]

The identities $1-m=-a(z)/(Bz)$ and $1+m=b(z)/(Bz)$ prove
\eqref{eq:factor}. On the unit circle
$Bm=(1+\beta)\cos\omega-A+i(1-\beta)\sin\omega$.
Its imaginary part is positive, and $\cot\Theta=g/(2s)$ gives
\eqref{eq:Xi}. The leading coefficient is nonpositive, while
\[
 \Xi(1)=-\frac{xy}{(y-x)(1-\beta)},\qquad
 \Xi(-1)=-\frac{(2+2\beta-x)(2+2\beta-y)}{(y-x)(1-\beta)}.
\]
Both are negative. Thus $\cot\Theta\to-\infty$ at both ends and
the continuous endpoint angles equal $\pi$.

\end{proof}
The comparison with $\omega/2$ is polynomial: with $c=\cos\omega$,
\begin{equation}\label{eq:Hr}
 H_r(c)=\Rea\bigl[(1-z)a(z)\ov{b(z)}\bigr],\qquad z=re^{i\omega},
\end{equation}
has degree at most three and $H_1(c)=(1-c)\cL(c)$, where
\begin{equation}\label{eq:E}
\begin{aligned}
 \cL(c)&=4\beta c^2+2(y-1-2\beta-\beta^2+\beta x)c
             +2-2x+xy+2\beta^2-2\beta y\\
 &=(y-x)(1-\beta)\bigl(1+c-\Xi(c)\bigr).
\end{aligned}
\end{equation}
Thus $\Theta(\omega)>\omega/2$ exactly when $\cL(\cos\omega)>0$.

\begin{lemma}[The lag-one cubic]\label{lem:cubic}
Put $P=1+\beta-x$, $Q=1+\beta-y$ and
$C=Qr^2+P\beta$, $D=Pr^2+Q\beta$,
$E=r^4+PQr^2+\beta^2$. Then
\[
 H_r(c)=E+r^2D-r(C+D+E+\beta r^2)c
          +r^2(C+2\beta)T_2(c)-\beta r^3T_3(c).
\]
At $r=1$, $H_1(c)=(1-c)\cL(c)$ with $\cL$ as in \eqref{eq:E},
$\cL(-1)=(2+2\beta-x)(2+2\beta-y)>0$ and
$\cL(1)=2(1-\beta)(y-x)+xy>0$.
For $\omega\in(0,\pi]$, $\Theta(\omega)>\omega/2$ exactly when
$\cL(\cos\omega)>0$.
The cycle-polynomial identity \eqref{eq:gtd} holds with $\kappa'=x/y$.
\end{lemma}

\begin{proof}
Expand \eqref{eq:Hr} using $z\bar z=r^2$ and
$z+\bar z=2rc$. Substitution $r=1$ gives \eqref{eq:E} and the
endpoint values. Strict monotonicity of cotangent on $(0,\pi)$ gives
the angle comparison. For the cycle polynomial of~\cite{GTD},
\begin{equation}\label{eq:gtd}
 x^2-2\bigl(\beta-c+\kappa'(1-\beta c)\bigr)x
 +2\kappa'(1-c)(1+\beta^2-2\beta c)=\kappa'\cL(c),
 \qquad \kappa'=x/y.
\end{equation}
This follows by collecting the coefficients of $c$.
\end{proof}

\subsection{Weighted integral quadratic constraints}
Fix a strict certificate $(\alpha_n,\beta_n,\gamma_\pm,\varepsilon)$
at $t$, with finitely many lags, and put $\varrho=e^t$.
Extend positions by zero to negative times and write
$\tilde e_k=\varrho^{-k}e_k$,
$\tilde\sigma_k=\varrho^{-k}\nabla\psi(e_k)$.
The weighted sum of interpolation inequalities becomes a quadratic
form in their Fourier transforms.
\begin{lemma}[The certificate as an integral quadratic constraint]\label{lem:iqc}
There is a $2\times2$ Hermitian matrix $\Pi(\omega)$ of trigonometric
polynomials with the following properties.
\begin{enumerate}[label=\textup{(\alph*)}]
\item For every $\psi\in\cS$ and every sequence with $\tilde e\in\ell_2$,
\[
 \mathcal J(\tilde e,\tilde\sigma):=\frac1{2\pi}\int_{-\pi}^{\pi}
 \begin{bmatrix}\hat e(\omega)\\\hat\sigma(\omega)\end{bmatrix}^*\Pi(\omega)
 \begin{bmatrix}\hat e(\omega)\\\hat\sigma(\omega)\end{bmatrix}d\omega\ \ge0,
\]
where $\hat e$ and $\hat\sigma$ are the Fourier transforms of $\tilde e$ and
$\tilde\sigma$, and $\Pi(\omega)$ acts blockwise.
\item $[1;m]^*\Pi(\omega)[1;m]=\Phi_t(\omega)$, where $m=m(t,\omega)$.
\end{enumerate}
\end{lemma}

\begin{proof}
Sum $2\varrho^{-2k-n}\alpha_n(R+T)_{k,k+n}$ and
$2\varrho^{-2k-n}\beta_n(R-T)_{k,k+n}$, together with the
minimizer inequalities weighted by $4\varrho^{-2k}\gamma_\pm$.
All sums converge absolutely because $|\psi(e)|\le\|e\|^2/2$ and
$\|\nabla\psi(e)\|\le\|e\|$.
The function-value coefficient is
\[
 4\left[\sum_n(\alpha_n-\beta_n)\sinh(nt)
           +\gamma_+-\gamma_-\right]
       \sum_k\varrho^{-2k}\psi(e_k)=0.
\]
The remaining form is a finite sum of shifts of
$(\tilde e,\tilde\sigma)$. Parseval gives its Hermitian
trigonometric matrix $\Pi$. Evaluation on the wave
$(e^{ik\omega},me^{ik\omega})$ gives exactly $\Phi_t$, since
the same cancellation removes its function values.
\end{proof}
\begin{lemma}[Homotopy]\label{lem:homotopy}
Let $\mathcal G$ be the causal convolution operator on $\ell_2(\Z_{\ge0};\R^d)$ of a
stable, strictly proper rational transfer function $G$, and let $\Delta$ be a
causal map on sequences with $\norm{P_N\Delta\xi}\le\norm{P_N\xi}$ and
$\norm{P_N\Delta\xi-P_N\Delta\xi'}\le\norm{P_N(\xi-\xi')}$ for every $N$, where $P_N$
truncates to the indices $<N$, and with $\Delta0=0$. Let $\Pi$ be a bounded
Hermitian matrix function such that
$\mathcal J(\xi,\tau\Delta\xi)\ge0$ for all $\xi\in\ell_2$ and $\tau\in[0,1]$, with
$\mathcal J$ as in Lemma~\ref{lem:iqc}, and
\[
 [G(e^{i\omega});1]^*\Pi(\omega)[G(e^{i\omega});1]\le-\varepsilon<0\quad(\omega\in[-\pi,\pi]).
\]
Then there is $c_*$ such that for all $\tau\in[0,1]$ and $f\in\ell_2$ the solution
$\xi$ of $\xi=\mathcal G\tau\Delta\xi+f$ satisfies $\norm\xi\le c_*\norm f$.
\end{lemma}

\begin{proof}
This is the discrete-time IQC continuation argument~\cite{MR}.
Strict properness makes the loop well posed and causal.
For a square-summable solution, put $\zeta=\tau\Delta\xi$.
The strict frequency inequality, Parseval, and
$\xi=\mathcal G\zeta+f$ give
\[
 \varepsilon\|\zeta\|^2
 \le C_1\|f\|\|\zeta\|+C_2\|f\|^2,
 \qquad \|\xi\|\le C_*\|f\|,
\]
with constants uniform in $\tau$.
To establish square summability, continue from $\tau=0$ in steps
$|\tau'-\tau|\le(2C_*\|\mathcal G\|)^{-1}$.
For a solution at $\tau'$, solve the $\tau$-loop with forcing
$P_N\{f+\mathcal G(\tau'-\tau)\Delta\xi'\}$.
Causality identifies the first $N$ entries, giving
$\|P_N\xi'\|\le C_*\|f\|+\frac12\|P_N\xi'\|$.
The resulting bound is uniform in $N$, so the solution lies in $\ell_2$.
Finitely many steps reach $1$.
\end{proof}
\begin{lemma}[Weighted shift]\label{lem:shift}
Let $\varrho>0$, and suppose that every trajectory $(e_k)_{k\ge-1}$ of
\eqref{eq:normrec}, with arbitrary $e_{-1},e_0$ and a fixed $\psi\in\cS$, satisfies
$\sum_{j\ge0}\varrho^{-2j}\norm{e_j}^2\le C_0(\norm{e_0}^2+\norm{e_{-1}}^2)$ for some $C_0>0$. Put
$C'=C_0\max(1,\varrho^{-2})$ and $\theta=C'/(1+C')$. Then
$\norm{e_k}^2\le C_0\theta^{\lfloor k/2\rfloor}\varrho^{2k}(\norm{e_0}^2+\norm{e_{-1}}^2)$
for all $k\ge0$.
\end{lemma}

\begin{proof}
For $S_k=\sum_{j\ge k}\varrho^{-2j}\|e_j\|^2$, restarting at
$(e_{k-1},e_k)$ gives
$S_k\le C'(S_{k-1}-S_{k+1})$.
Since $S_{k+1}\le S_k$,
$S_{k+1}\le\theta S_{k-1}$; induction yields
$S_k\le\theta^{\lfloor k/2\rfloor}S_0$ and the claimed bound.
\end{proof}

\subsection{One-sided multipliers and exact quadratic rates}\label{app:general-multipliers}
For nonnegative $\lambda_1,\ldots,\lambda_N$, not all zero, define
\[
 M(z)=\sum_{n=1}^N\lambda_n(1-z^n).
\]
This polynomial encodes the nonnegative weights of the lower-lag
inequalities; in control terminology it is a finite causal
Zames--Falb multiplier. The following criterion transfers its sign
between circles by the harmonic minimum principle.
\begin{theorem}[Multiplier criterion]\label{thm:mult}
Assume \eqref{eq:quad}.
\begin{enumerate}[label=\textup{(\alph*)}]
\item If $\Rea[M(z)a(z)/b(z)]>0$ on $|z|=1$, $z\ne1$, then $\rho<1$.
If $\Rea[M(z)a(z)/b(z)]\ge0$ on $|z|=1$, then $\rho\le1$.
\item If, in addition to the strict hypothesis of \textup{(a)},
$\Rea[M(z)a(z)/b(z)]\ge0$ on $|z|=r$ for some $r\in(\rho_b,1)$, then $\rho\le r$.
\end{enumerate}
The same holds with $a$ and $b$ interchanged and $\rho_a$ in place of
$\rho_b$, for multipliers built from upper lags.
\end{theorem}
The proof follows Lemma~\ref{lem:symbols}.

Put $a^\sharp(u)=u^2a(1/u)$ and $b^\sharp(u)=u^2b(1/u)$.
Both have constant coefficient $1$ and no zero in the closed unit disk.
\begin{lemma}[Symbols of one-sided certificates]\label{lem:symbols}
Let $\varrho=e^t$ and $M(z)=\sum_{n=1}^N\lambda_n(1-z^n)$ with $\lambda_n\ge0$.
\begin{enumerate}[label=\textup{(\alph*)}]
\item If $t\ge0$, the weights $\alpha_n=\lambda_n\varrho^{-n}$ on $\ell_n\ge0$ and
$\gamma_-=\sum_n\alpha_n\sh nt$ on $\bar g-C\ge0$ cancel $C$, and, with $u=e^{i\omega}/\varrho$,
the symbol is
$\Phi_t(\omega)=-\frac{\varrho^2}{B^2}|a^\sharp(u)|^2\Rea\frac{M(u)b^\sharp(u)}{a^\sharp(u)}$.
\item If $t\le0$, the weights $\alpha_n=\lambda_n\varrho^n$ on $\ell_n\ge0$ and
$\gamma_+=-\sum_n\alpha_n\sh nt$ on $\bar g+C\ge0$ cancel $C$, and, with $z=\varrho e^{i\omega}$,
the symbol is
$\Phi_t(\omega)=-\frac1{B^2\varrho^2}|b(z)|^2\Rea\frac{M(z)a(z)}{b(z)}$.
\end{enumerate}
At $t=0$ the two expressions agree, and on $|z|=1$ the functions
$\Rea[Mb^\sharp/a^\sharp]$ and $\Rea[Ma/b]$ have the same sign. For upper lags the
same holds with $a$ and $b$ interchanged.
\end{lemma}

\begin{proof}
The stated minimizer weights cancel $C$.
Use $\cosh(nt)\pm\sinh(nt)=\varrho^{\pm n}$ in the
lag sum and then \eqref{eq:factor}.
For $t\ge0$, substitute $z=1/\bar u$; for $t\le0$, use
$z=\varrho e^{i\omega}$ directly.
On the unit circle
$b^\sharp/a^\sharp=\overline{b/a}=|b/a|^2a/b$.
Conjugating \eqref{eq:factor} interchanges the lag orientations.
\end{proof}
\begin{proof}[Proof of Theorem~\ref{thm:mult}]

The real part of $F=Mb^\sharp/a^\sharp$ is harmonic on the
disk, nonnegative on its boundary and positive at $0$.
The strong minimum principle makes it positive in the interior;
Lemma~\ref{lem:symbols}(a) gives strict certificates at every $t>0$,
hence $\rho\le1$.
For the strict hypothesis, the zero at $z=1$ is removed by a
small positive minimizer weight: $g(0,0)<0$, and the remaining
compact arc already has a strict margin. Theorem~\ref{thm:oneexp}
gives $\rho<1$.
For part~(b), $Ma/b$ is analytic on $r<|z|<1$ and has
nonnegative real part on both boundary circles.
The strong minimum principle gives strict positivity inside;
Lemma~\ref{lem:symbols}(b) supplies certificates for
$\log r<t<0$. Together with the positive exponents they imply
$\rho\le r$. Interchanging $a,b$ proves the upper-lag case.

\end{proof}
\begin{proposition}[Single-harmonic spectra]\label{prop:single}
Let $\omega_0=2\pi j/k$ with $\gcd(j,k)=1$ and $1\le j<k/2$. At $t=0$ the pair
$(\delta_{\omega_0},0)$ is feasible if and only if $\Theta(\omega_0)\le\pi/k$.
\end{proposition}

\begin{proof}
For one atom, the signed lag slacks are positive multiples of
$\cos\Theta-\cos(2\pi r/k-\Theta)$, $r=0,\ldots,k-1$.
All are nonnegative exactly when $\Theta\le\pi/k$;
then $\cos\Theta>0$ also gives the minimizer constraint.
\end{proof}
\begin{proof}[Proof of Theorem~\ref{thm:A}]

For $M=1-z$, the real part in Theorem~\ref{thm:mult} is
$H_1(c)/|b|^2$, giving (i).
Proposition~\ref{prop:single} and start-up realization give the
single $\R^3$ trajectory in (ii). If the phase inequality is strict,
the nonzero lag residues have a common positive margin.
Writing $g_0=g(0,\omega_0)>0$, at a small positive exponent
multiples of $k$ have slack $g(t)(\cosh(nt)-1)\ge0$. For other
residues the lower slack satisfies
\[
 \ell_n(t)\ge\ell_n(0)+g(t)(\cosh(nt)-1)
        -2|g(t)-g_0|-2|s(t)-s(0)|,
\]
and the upper slack has the same estimate. The last two terms tend
to zero independently of $n$, whereas the original nonzero residues
have a common positive minimum. All lags therefore remain feasible
for some $t>0$, proving $\rho\ge e^t>1$.
On the relative closure of $\Omega_1$, $\cL\ge0$, so the weak
multiplier criterion gives $\rho\le1$.
A phase-locked harmonic gives the reverse inequality and an actual
nonconverging trajectory, while nearby points of $\Omega_1$ converge.

\end{proof}
For the lag-one multiplier write
$\mathsf G(z)=(1-z)a(z)/b(z)$, which is analytic for $|z|>\rho_b$.
Where $b(re^{i\omega})\ne0$,
$H_r(\cos\omega)=|b(re^{i\omega})|^2\Rea\mathsf G(re^{i\omega})$.
\begin{lemma}[Near the unit circle]\label{lem:near1}
Assume $\Theta>\omega/2$ on $(0,\pi]$. There is $\eta>0$ with $\Rea\mathsf G>0$ on
$\{1-\eta\le|z|<1\}$. If $H_r\ge0$ on $[-1,1]$ for some $r\in(\rho_b,1)$, then
$H_{r'}\ge0$ on $[-1,1]$ for all $r'\in[r,1)$.
\end{lemma}

\begin{proof}
Write $\mathsf G=(1-z)\gamma(z)$, $\gamma=a/b$.
For $z=(1-\delta)e^{i\theta}$,
\[
 \Rea\mathsf G=\frac xy\delta+
       \frac{\cL(1)}{2y^2}\theta^2
       +O(\delta^2+\delta\theta^2+|\theta|^3).
\]
Both displayed coefficients are positive. This proves positivity
near $1$ inside the disk; compactness supplies a thin positive annulus
on the remaining arc. The second assertion is the minimum principle
on $r\le|z|\le1$.
\end{proof}
\begin{proof}[Proof of Theorem~\ref{thm:B}]

Lemma~\ref{lem:near1} makes the admissible radii a nonempty upper
interval with supremum $1$. The multiplier criterion gives
$\rho\le r$ at each such radius, proving (a).
Continuity gives the necessity in (b).
For sufficiency, when $\rho_a>\rho_b$, apply the annulus argument
at $r=\rho_a$.
When $\rho_b\ge\rho_a$, the only possible boundary pole is
$z_b=-\rho_b$, with residue
$(1-z_b)a(z_b)/b'(z_b)<0$.
Let its residue be $v<0$ and write $z=z_b+w$.
The condition $|z|>\rho_b$ implies
$\Rea w<|w|^2/(2\rho_b)$, so
$\Rea(v/w)>v/(2\rho_b)$. The real part is thus bounded below
near the pole. If $d$ is the annulus diameter, the harmonic function
$\Rea\mathsf G-\epsilon\log(|z-z_b|/d)$ tends to $+\infty$
there and has nonnegative lower limits at all other boundary points.
The minimum principle and $\epsilon\downarrow0$ give nonnegativity
in the annulus; the strong minimum principle gives strict positivity.
Thus $r_1=\rhoq$.
For (c), $\beta>0$ and $|A|<2\sqrt\beta$.
Choose $2\sqrt\beta\cos\omega_0=A$; at
$t_0=\log\sqrt\beta$, this mode has $m=0$, $g=1$.
Its signed slacks are at least $\cosh(t_0)-1>0$.
The estimate
$\ell_n,\upsilon_n\ge g(\cosh t-1)-2|s|$
retains positivity at a slightly larger negative exponent,
uniformly in $n$, so realization gives $\rho>\sqrt\beta$.

\end{proof}
\begin{corollary}[Explicit regions]\label{cor:C}
Assume \eqref{eq:quad}.
\begin{enumerate}[label=\textup{(\alph*)}]
\item $\cL(-1)=(2+2\beta-x)(2+2\beta-y)>0$ and $\cL(1)=2(1-\beta)(y-x)+xy>0$.
Hence $\Omega_1$ is the set where $\beta=0$, or the vertex
$c_*=(1+2\beta+\beta^2-\beta x-y)/(4\beta)$ lies outside $(-1,1)$, or the
discriminant $(y-1-2\beta-\beta^2+\beta x)^2-4\beta(2-2x+xy+2\beta^2-2\beta y)$
is negative. On the boundary of $\Omega_1$ the tangency frequency is
$\omega_*=\arccos c_*$, and the phase-locked points of
Theorem~\ref{thm:A}\textup{(iii)} are those with $c_*=\cos(2\pi/K)$.
\item With $c_{jk}=\cos(2\pi j/k)$ and $h_{jk}=\cot(\pi/k)\sin(2\pi j/k)$,
\begin{align*}
 \Theta(2\pi j/k)\le\pi/k&\iff \Xi(c_{jk})\ge h_{jk}\\
 &\iff\cL(c_{jk})\le(y-x)(1-\beta)(1+c_{jk}-h_{jk}).
\end{align*}
For $j=1$ the right side vanishes, and the $K$-tongue is
$\{\cL(\cos(2\pi/K))\le0\}$.
\item If $\rho_a\ge\rho_b$ and $\rhoq>\sqrt\beta$, then with $r=\rho_a$,
$H_r(c)=(1-c)K_\mu(c)$ for a convex polynomial $K_\mu$ of degree at most two with $K_\mu(-1)\ge0$;
within $\Omega_1$ the set $\cQ$ is given there by $K_\mu(1)\ge0$ together with
the vertex-or-discriminant condition for $K_\mu$. If $\rho_b>\rho_a$ and
$\rhoq>\sqrt\beta$, then with $r=\rho_b$, $H_r(c)=(1+c)K_L(c)$ for a concave
polynomial $K_L$ of degree at most two with $K_L(1)\ge0$, and within $\Omega_1$ the set $\cQ$ is
given there by $K_L(-1)\ge0$.
\end{enumerate}
\end{corollary}
The proof is given together with Lemma~\ref{lem:Kmu} below.
\begin{lemma}[Quotients at the dominant real roots]\label{lem:Kmu}
Let $r>0$. If $a(r)=0$, equivalently
$x=(1-r)(1-\beta/r)$, then $H_r(c)=(1-c)K_\mu(c)$,
where $K_\mu$ has degree at most two, quadratic coefficient
$4\beta r^3$, and
$K_\mu(-1)=(\beta+r^2)(1+r)b(-r)$.
If $b(-r)=0$, equivalently $y=(1+r)(1+\beta/r)$, then
$H_r(c)=(1+c)K_L(c)$, where $K_L$ has degree at most two,
quadratic coefficient $-4\beta r^3$, and
$K_L(1)=(\beta+r^2)(1-r)a(r)$.
Both quotients are determined by polynomial division of the cubic
in Lemma~\ref{lem:cubic}.
\end{lemma}

\begin{proof}[Proof of Lemma~\ref{lem:Kmu} and Corollary~\ref{cor:C}]
Substitute $x=(1-r)(1-\beta/r)$ or
$y=(1+r)(1+\beta/r)$ in \eqref{eq:Hr} and divide by
$1-c$ or $1+c$.
The displayed quotients and endpoint identities follow by coefficient
comparison. Lemma~\ref{lem:dominant} gives the required endpoint signs.
A convex polynomial of degree at most two is positive on $[-1,1]$
precisely when its endpoints are positive and either its vertex lies
outside $(-1,1)$ or its discriminant is negative.
For nonnegativity replace the strict signs by weak signs; a concave
polynomial is nonnegative exactly when both endpoints are.
These tests prove (a),(c), including the linear case $\beta=0$.
For (b), use the monotonicity of cotangent and \eqref{eq:E}.
\end{proof}
\begin{proof}[Proof of Corollary~\ref{cor:D}]

For Polyak's tuning, put $u_0=(q^{-1}-q)/2$.
Then $2u_0(1+c-\Xi(c))=(c+u_0)^2+2u_0-1$.
Its minimum is positive exactly when $u_0>1/2$; at equality
the double contact is $c=-1/2$, giving the three-tongue threshold.
For $\beta=0$, the linear $\cL$ has positive endpoints.
Normalization gives
$\operatorname{Lip}(I-\alpha\nabla f)\le
|1-(x+y)/2|+(y-x)/2=\max(|1-x|,|1-y|)$.
The endpoint scalar quadratics attain that rate, giving the stated
formula and, by the quotient tests, membership in $\cQ$.

\end{proof}

\section{Threshold geometry and periodic approximation}\label{app:levels}
For signed $n\ne0$, the basic polynomial identity is
\begin{equation}\label{eq:Lambdapoly}
 \Lambda_n(\omega)=\frac{\sin\Theta}{\sin\omega}
 \left[\Xi(c)(1-T_{|n|}(c))-
       \sgn(n)(1-c^2)U_{|n|-1}(c)\right],\quad c=\cos\omega.
\end{equation}
At the endpoints,
$\Lambda_n(0)=0$, $\Lambda_n(\pi)=(-1)^n-1\le0$ and
$\cos\Theta=-1$. Removing endpoint mass therefore preserves
feasibility and leaves a nonzero measure.
\begin{lemma}[Ends]\label{lem:ends}
Let $q$ be a tuning and $s\in\R$. Near $\omega=0$ and $\omega=\pi$,
\begin{equation}\label{eq:A-ends}
 \pi-\Theta_{q,s}(\omega)=\arctan\frac{\sin\omega}{-\Xi_q(\cos\omega)-s\sin\omega},
\end{equation}
with $-\Xi_q(\pm1)>0$. Hence $\Theta_{q,s}(\omega)$ is continuous in $(q,s,\omega)$ on (tunings)$\times\R\times[0,\pi]$
and extends to a real analytic function of $(q,s,\omega)$ on a neighborhood of every
compact subset. With $\tau_q=-1/\Xi_q(1)$, every lag function of $\Theta_{q,s}$ satisfies
$\Lambda_n(0)=\partial_\omega\Lambda_n(0)=0$ and $\partial^2_\omega\Lambda_n(0)=-n(n+2\tau_q)$.
\end{lemma}

\begin{proof}
The denominator $-\Xi_q(\cos\omega)-s\sin\omega$ is positive
near the endpoints. Formula~\eqref{eq:A-ends} gives the local
analytic extensions, agreeing with the interior arccotangent.
At $0$, $\pi-\Theta=\tau_q\omega+O(\omega^2)$; substitute in
$\Lambda_n=-\cos(\pi-\Theta)+\cos(n\omega+\pi-\Theta)$.
\end{proof}
\begin{lemma}[Exact alternative]\label{lem:alt}
Let $\Theta$ be a curve and $\mathcal N\subset\Z\setminus\{0\}$. Then $\FP[\mathcal N](\Theta)$ is infeasible if and only if
a strict certificate symbol with lags in $\mathcal N$ exists.
\end{lemma}

\begin{proof}
For finitely many lags, take the compact convex hull of the moment curve
\[
 \omega\longmapsto\bigl((\Lambda_n(\omega))_n,\cos\Theta(\omega)\bigr).
\]
Feasibility asks whether this hull meets the nonnegative orthant.
Strict separation~\cite[\S2.5.1]{BoydVandenberghe} gives the
nonnegative certificate coefficients when they are disjoint:
the separating functional has nonnegative coordinates because
the orthant is a cone.
For arbitrary lags, normalize feasible measures to probability.
Compactness and the finite-intersection property imply that
finite-subsystem feasibility gives full feasibility. Integrating
a strict symbol proves the converse.
\end{proof}
\subsection{General spectra and monotone thresholds}\label{app:levels-realization}
\paragraph{Realization of general spectra at exponent zero.}
For the lower bound in Theorem~\ref{thm:exactthr}, let $(\nu,0)$ be
feasible at $t=0$, with $\nu$ normalized to have mass one. Work in
the real Hilbert space $L^2(\nu;\C)\oplus\R$, and put
$z(\omega)=e^{i\omega}$, $m=m(0,\omega)$ and
$r=(z-A+\beta)/B$. As in Lemma~\ref{lem:startup}, choose
$\tau\in(-1,1)$ so that the recurrence
$p_{-1}=p_0=1$, $p_{k+1}=(A+B\tau)p_k-\beta p_{k-1}$ satisfies
$p_k\to0$ and $\delta_0=\inf_{k\ge1}(1-p_k)^2>0$.
For $Y>0$ prescribe
\[
 \begin{aligned}
 e_0&=(1,Y),&\sigma_0&=(r,\tau Y),&
 \psi_0&=\tfrac12\int\Rea r\,d\nu+\tfrac\tau2Y^2,\\
 e_k&=(z^k,p_kY),&\sigma_k&=(mz^k,\tau p_kY),&
 \psi_k&=\tfrac12\int\Rea m\,d\nu+\tfrac\tau2p_k^2Y^2
 \quad(k\ge1).
 \end{aligned}
\]
Include the minimizer datum $(0,0,0)$ and set $e_{-1}=e_0$.
The identities $B r=z-A+\beta$ and
$Bmz=z^2-Az+\beta$ give \eqref{eq:normrec} at every step.

All interpolation inequalities excluding index $0$ follow by
integrating the feasible spectral inequalities. The scalar summand
is the quadratic $\tau u^2/2$: it contributes $T_{ij}=0$ and
$(1-\tau^2)(u_i-u_j)^2/4$ to $R_{ij}$.
The exceptional first-summand deficits
$|T_{0k}|-R_{0k}$, $k\ge1$, and $|T_{0\star}|-R_{0\star}$
are bounded above by a finite $M\ge0$, because $z^k$, $m$ and $r$
are uniformly bounded in $L^2(\nu)$ and their prescribed values are
bounded. Thus
\[
 Y^2\ge\frac{4M}{(1-\tau^2)\min\{1,\delta_0\}}
\]
makes every interpolation inequality valid, with one $Y$ independent
of the horizon.

For each $N$, restrict the data to the finite-dimensional span of
$e_0,\ldots,e_N,\sigma_0,\ldots,\sigma_N$.
Theorem~\ref{thm:interpolation} and Lemma~\ref{lem:normalization}
give an objective in $\cF$ realizing these first $N$ steps. Since
$\|e_N\|^2\ge\int|z|^{2N}\,d\nu=1$ and
$\|e_0\|^2=1+Y^2$, we obtain $W_N\ge(1+Y^2)^{-1/2}$ for all $N$.
The finite-dimensional objective and its dimension may depend on
$N$; the lower-bound constant is common to all horizons.
\begin{proposition}[Monotonicity and phase floors]\label{prop:mono}
\begin{enumerate}[label=\textup{(\alph*)}]
\item Let $\Theta,\Theta'$ be curves with $\Theta'\le\Theta$ on $(0,\pi)$. If $\FP(\Theta)$ is feasible, then
$\FP(\Theta')$ is feasible.
\item For a tuning, a strict certificate at $t=0$ exists if and only if there are finitely
many $\lambda_n\ge0$, $n\ne0$, and $\gamma>0$ such that
$M(z)=\gamma+\sum_n\lambda_n(1-z^n)$ satisfies
$\Theta(\omega)>\arg M(e^{i\omega})+\pi/2$ for all $\omega\in(0,\pi]$.
\item For every such $M$ and every $\omega=2\pi j/k$, $\arg M(e^{i\omega})+\pi/2\ge\pi/k$.
For $M=1-z$, $\arg M(e^{i\omega})+\pi/2=\omega/2$.
\end{enumerate}
\end{proposition}

\begin{proof}
Delete the phase-$\pi$ set from a feasible measure; its contribution
to every constraint is nonpositive, and its minimizer contribution
is strictly negative. On the remainder, reweight by
$\sin\Theta/\sin\Theta'$.
This ratio is bounded: it is at most $1$ where
$\Theta'>\pi/2$, and at most $1/\sin(\min\Theta')$ elsewhere.
The normalized lag $\Lambda_n/\sin\Theta=\ell_n[\cot\Theta]$
and normalized minimizer are nondecreasing in cotangent, proving (a).
For (b), $\Rea M\ge\gamma>0$ and
$S=\Rea(e^{-i\Theta}M)$.
Its negativity is exactly the phase-floor inequality;
compactness supplies a uniform margin. A strict symbol has
$\gamma>0$ by its value at $0$.
At $2\pi j/k$, all nonzero summands $1-z^n$ lie in the sector
$[-\pi/2+\pi/k,\pi/2-\pi/k]$, as does their positive sum
with $\gamma$. This proves (c), the discrete-time phase limitation
of~\cite{ZCH}; $M=1-z$ gives equality with $\omega/2$.
\end{proof}

\subsection{The dip and coordinates of a gap}\label{sec:convexdip}
\begin{theorem}[Convexity]\label{thm:convex}
Let $q$ be a tuning, $c=\cos\omega$, and write $\Xi(c)=\Xi_0+\Xi_1c+\Xi_2c^2$. Then
\[
 h_q''(\omega)\,\sin^3\omega=\tfrac12(1+c)^2\,\Xi(1)+\tfrac12(1-c)^2\,\Xi(-1)+\Xi_2\,(1-c^2)^2<0\quad(0<\omega<\pi),
\]
so $h_q$ is strictly concave on $(0,\pi)$. Consequently, for every shift $s$,
$\Theta_{q,s}''(\omega)>0$ at every $\omega\in(0,\pi)$ with $\Theta_{q,s}(\omega)\le\pi/2$, and
$\{\Theta_{q,s}\le\pi/2\}$ is a compact interval in $(0,\pi)$ or empty.
\end{theorem}

\begin{proof}
Twice differentiating $h=\Xi(c)/\sin\omega$ and expressing
$\Xi_0,\Xi_1$ through $\Xi(1),\Xi(-1)$ gives the displayed
identity. Its sign follows from $\Xi(\pm1)<0$ and $\Xi_2\le0$.
For $u=h+s$, differentiation gives
$\Theta''=-u''/(1+u^2)+2u(u')^2/(1+u^2)^2>0$ where $u\ge0$.
The set $\{u\ge0\}$ is a compact interval because $u$ is
strictly concave and tends to $-\infty$ at both ends.
\end{proof}
\begin{corollary}[Connected phase sublevel sets]\label{cor:dip}
Let $\Theta=\Theta_{q,s}$ be a shifted curve, $\ell(\omega)=n\omega/2-\pi m$ a lattice line
and $J$ an interval on which $\ell\le\pi/2$. Then $\{\omega\in J:\Theta(\omega)<\ell(\omega)\}$
is an interval. In particular the dip $\{\omega:\Theta(\omega)<\omega/2\}$ and the
closed dip
\[
 \bar\Delta=\{\omega\in(0,\pi]:\ \Theta(\omega)\le\omega/2\}
\]
are intervals; $\bar\Delta$ is compact and contained in $(0,\pi)$, and if the dip is
nonempty, $\bar\Delta$ is its closure.
\end{corollary}

\begin{proof}
On the interval $\{\Theta\le\pi/2\}$, the function
$\Theta-\ell$ is strictly convex. Its strict and weak
sublevel sets are intervals. Taking $\ell=\omega/2$
gives compact containment and the closure statement.
\end{proof}
\begin{lemma}[Dip lemma]\label{lem:dip}
Let $\Theta$ be a curve and $\bar\Delta=\{\omega\in(0,\pi]:\Theta\le\omega/2\}$. Every solution
of $\FP(\Theta)$ gives positive mass to $\bar\Delta$.
\end{lemma}

\begin{proof}
Outside $\bar\Delta\cup\{0\}$ the lag-one function is strictly
negative. A measure giving no mass to $\bar\Delta$ violates
that lag unless it is carried by $0$, where it violates the
minimizer constraint.
\end{proof}
\begin{proof}[Proof of Theorem~\ref{thm:tangency}]

Strict convexity makes $\omega_*$ the unique nonzero zero of
$\Lambda_1\le0$. At that frequency
$\Lambda_{-K}=2\sin(K\omega_*/2)\sin((K+1)\omega_*/2)<0$.
Choose a small positive weight $\epsilon$ on $-K$.
Near $\omega_*$ it removes the zero; away from $0,\omega_*$
the lag-one strict margin is preserved. Near $0$, the expansion
in Lemma~\ref{lem:ends} retains a negative quadratic coefficient
for small $\epsilon$. Adding a still smaller
$\gamma\cos\Theta$, $\gamma>0$, makes the symbol strictly
negative also at $0$. The exact criterion gives convergence.
The remaining boundary contacts are precisely the locked ones
of Theorem~\ref{thm:A}(iii).

\end{proof}
\begin{proof}[Proof of Theorem~\ref{thm:levelzero}]

The compact interval $\bar\Delta$ either is empty, contains one
of the frequencies $2\pi/K'$, or lies in the unique interval
between consecutive such frequencies. The first two cases are
Theorem~\ref{thm:A}. In the last case a singleton is a tangency
handled by Theorem~\ref{thm:tangency}; otherwise its interior
is the dip and the endpoint inequalities define $I_K$.
For a singleton, decreasing $\kappa$ strictly raises the angle
curve by Lemma~\ref{lem:kappa}, placing the point on
$\partial\Omega_1$.

\end{proof}
\begin{proposition}[Coordinates of a gap]\label{prop:chart}
\begin{enumerate}[label=\textup{(\alph*)}]
\item The map $q\mapsto\Xi_q$ is a homeomorphism from the tunings with $\beta>0$ onto
the set of real quadratics $\Xi(c)=\Xi_0+\Xi_1c+\Xi_2c^2$ with $\Xi_2<0$, $\Xi(1)<0$ and
$\Xi(-1)<0$. Tunings with $\beta=0$ have $\Xi_2=0$ and lie in $\Omega_1$.
\item Let $K\ge2$. A tuning lies in $I_K$ if and only if $\beta>0$ and
\[
 1+c-\Xi_q(c)=a\,(c-r_1)(c-r_2)\quad\text{with}\quad c_K<r_1<r_2<c_{K+1}.
\]
Then $a=-\Xi_2>a_0(r_1,r_2):=2/\bigl((1-r_1)(1-r_2)\bigr)$, the closed dip is
$[\arccos r_2,\arccos r_1]$, and $q\mapsto(a,r_1,r_2)$ is a homeomorphism of $I_K$
onto
\[
 \mathcal G_K=\bigl\{(a,r_1,r_2):\ c_K<r_1<r_2<c_{K+1},\ a>a_0(r_1,r_2)\bigr\}.
\]
In particular $I_K$ is connected.
\item Consider the closure of $\mathcal G_K$ within $\{a>a_0(r_1,r_2),\ r_1>-1\}$, where the
quadratic $1+c-a(c-r_1)(c-r_2)$ still corresponds to a tuning by (a). Its face
$\{r_1=r_2=r,\ c_K<r<c_{K+1}\}$ consists exactly of the points of $\partial\Omega_1$ that are not
phase-locked and whose tangency frequency $\arccos r$ lies in the window. Its face $r_2=c_{K+1}$
lies on the edge $\Theta(2\pi/(K+1))=\pi/(K+1)$ of the $(K+1)$-tongue, and $r_1=r_2=c_{K+1}$ is a
phase-locked point. For $K\ge3$ the face $r_1=c_K$ lies on the edge of the $K$-tongue,
$r_1=r_2=c_K$ is a phase-locked point, and $(r_1,r_2)=(c_K,c_{K+1})$ is a junction of the two
tongues. For $K=2$, as $r_1\downarrow-1$ the tunings approach the line $\alpha L=2(1+\beta)$.
\item For every fraction $j/k$ in lowest terms with $1/(K+1)<j/k<1/K$, put
$c_{jk}=\cos(2\pi j/k)$ and $h_{jk}=\cot(\pi/k)\sin(2\pi j/k)$. Then
$h_{jk}>1+c_{jk}$, and the tuning with coordinates $(a,r_1,r_2)\in\mathcal G_K$ satisfies
$\Theta(2\pi j/k)\le\pi/k$ if and only if
\[
 r_1<c_{jk}<r_2\qquad\text{and}\qquad a\ge A_{jk}(r_1,r_2):=\frac{h_{jk}-1-c_{jk}}{(c_{jk}-r_1)(r_2-c_{jk})} .
\]
\item The face $a=a_0(r_1,r_2)$ is $\Xi(1)=0$, and $\Xi_q(1)=-\alpha L/\bigl((\kappa-1)(1-\beta)\bigr)$. Along every
sequence of tunings whose quadratics converge to a quadratic with $\Xi(1)=0$ and
$\Xi(-1)<0$, the condition number tends to infinity.
\end{enumerate}
\end{proposition}
The proof follows Lemma~\ref{lem:Efactor}.

For a quadratic $\Xi$, put
\begin{equation}\label{eq:Ez}
 E_\Xi(z)=\Xi\left(\frac{z+z^{-1}}2\right)-\frac{z-z^{-1}}2,
 \qquad p_\Xi(z)=z^2E_\Xi(z).
\end{equation}
\begin{lemma}\label{lem:Efactor}
For every tuning and $z\ne0$,
$2B(1-\beta)E_{\Xi_q}(z)=-a(z)\,b(z^{-1})$.
\end{lemma}

\begin{proof}
Expand in the five Laurent monomials $z^{-2},\ldots,z^2$;
their coefficients agree by \eqref{eq:Xi}.
\end{proof}
\begin{proof}[Proof of Proposition~\ref{prop:chart}]

For (a), $p_\Xi$ has no unit-circle zero: on that circle
$E_\Xi=\Xi(\cos\omega)-i\sin\omega$, whose real-axis
intersections are the two negative values $\Xi(\pm1)$.
Its winding number is zero, so $p_\Xi$ has two inner and two
outer roots. Since $p_\Xi(0)=\Xi_2/4<0$ and
$p_\Xi(\pm1)<0$, the inner roots are conjugate or have the
same real sign. Their product is $\beta\in(0,1)$.
Take them as the roots of $a=z^2-Pz+\beta$; reciprocal
outer roots give $b=z^2-Qz+\beta$, since the quartic's
leading and constant coefficients agree.
Coefficient comparison in Lemma~\ref{lem:Efactor} gives
$(y-x)(1-\beta)>0$, where $x=1+\beta-P$, $y=1+\beta-Q$.
Schur stability gives $0<x<y<2(1+\beta)$.
This is the unique tuning; separation of the inner and outer
root multisets also gives continuous dependence, including
multiple roots. At $\beta=0$, $\cL$ is linear and positive.

For (b), the dip is exactly the negative set of
$1+c-\Xi(c)=a(c-r_1)(c-r_2)$.
The window and strict tongue exclusions give the ordered roots;
$\Xi(1)<0$ is $a>a_0$, while $\Xi(-1)<0$ is automatic.
These identities give both continuous coordinate maps.
The region above the continuous graph $a_0$ over its open
triangle is path connected.
Part (c) follows by allowing the ordered roots to coincide or
reach window endpoints. For $K=2$, the face $r_1=-1$
is $\Xi(-1)=0$, hence the excluded stability boundary.

For (d),
$h_{jk}-1-c_{jk}=
2\cos(\pi j/k)\sin(\pi(j-1)/k)/\sin(\pi/k)>0$.
Substitute $\Xi(c)=1+c+a(c-r_1)(r_2-c)$.
For (e), the endpoint identity is in the proof of
Lemma~\ref{lem:angle}.
If $\Xi_n\to\bar\Xi$ with $\bar\Xi(1)=0$ and
$\bar\Xi(-1)<0$, the negative-endpoint identity rules out
$y_n\to0$. Hence
$\kappa_n-1=y_n/(|\Xi_n(1)|(1-\beta_n))\to\infty$.

\end{proof}

\subsection{Monotonicity, periodicity and the threshold}
\begin{proposition}[Monotone form]\label{prop:monoform}
For a tuning $q$ and $s\in\R$, $\FP(\Theta_{q,s})$ is feasible if and only if there is
a nonzero finite measure $\nu\ge0$ on $(0,\pi)$ with $\int|h_q+s|\,d\nu<\infty$ and
\[
 \int\ell_n[h_q+s]\,d\nu\ge0\quad(n\ne0),\qquad\int(h_q+s)\,d\nu\ge0 .
\]
If $\nu$ satisfies these conditions for a function $h$, and $h'$ is $\nu$-measurable with $h'\ge h$ on $\supp\nu$ and bounded above there, then
$\int\ell_n[h']\,d\nu=\int\ell_n[h]\,d\nu+\int(h'-h)(1-\cos n\omega)\,d\nu$, and $\nu$ satisfies them
for $h'$.
\end{proposition}

\begin{proof}
Delete the endpoint atoms and put $\nu=\sin\Theta\,w$.
Then $\int|h|\,d\nu=\int|\cos\Theta|\,dw<\infty$,
and the constraints transform as stated.
Conversely $w=\nu/\sin\Theta$ is finite because
$1/\sin\Theta=\sqrt{1+h^2}\le1+|h|$.
Measurability and $h\le h'\le M$ give
$|h'|\le|h|+M^+$, hence integrability.
The displayed increment formula follows by linearity.
\end{proof}
\begin{lemma}[The condition number lowers the curve]\label{lem:kappa}
With $y=\alpha L$,
\[
 \frac{\partial\Xi_q(c)}{\partial\kappa}=\frac{N(c)}{y(1-\beta)(\kappa-1)^2},
\]
where $N(c)=\bigl(y-(1+\beta)(1-c)\bigr)^2+(1-\beta)^2(1-c^2)>0$ for $-1\le c\le1$.
Hence $h_{(\kappa',\alpha L,\beta)}\ge h_{(\kappa,\alpha L,\beta)}+\varepsilon$ on $(0,\pi)$ for $\kappa'>\kappa$ and some
$\varepsilon>0$, and the curve $\Theta_q$ decreases strictly as $\kappa$ grows.
\end{lemma}

\begin{proof}
Differentiate \eqref{eq:Xi} using $x=y/\kappa$.
Completing the square gives $N(c)$.
It is positive for $|c|<1$ and has endpoint values
$N(1)=y^2>0$, $N(-1)=(y-2(1+\beta))^2>0$.
Compactness supplies a uniform positive increment of $\Xi$;
division by $0<\sin\omega\le1$ gives the cotangent increment.
\end{proof}
\begin{proposition}[Threshold]\label{prop:threshold}
For every tuning $q$ there is a number $\sigma(q)$, the threshold, such that
$\FP(\Theta_{q,s})$ is feasible if and only if $s\ge\sigma(q)$. Moreover:
\begin{enumerate}[label=\textup{(\alph*)}]
\item $\mathcal D=\{q:\ \sigma(q)\le0\}$;
\item if $h_{q'}\ge h_q+e$ on $(0,\pi)$ for a constant $e$, then $\sigma(q')\le\sigma(q)-e$;
\item the set $\{(q,s):\ s\ge\sigma(q)\}$ is closed, and $\FP(\Theta_{q,\sigma(q)})$ is feasible.
\end{enumerate}
\end{proposition}
\begin{proof}[Proof of Proposition~\ref{prop:threshold}]

The monotone form makes feasibility increasing in $s$.
For large $s$, the atom at $2\pi/3$ is feasible;
for $s<-\sup h_q$, the minimizer constraint fails.
Normalize feasible measures to probability. Joint continuity
from Lemma~\ref{lem:ends} and weak compactness show that the
feasible epigraph is closed, so its finite lower endpoint is
attained. The exact criterion proves (a), and monotonicity
proves (b). The same reasoning applies to arbitrary lag subsets
and finite periodic grids; for $\PP_L$ the atom at $2\pi/L$
gives the upper witness.

\end{proof}
\begin{proposition}[Periodic spectra]\label{prop:periodic}
Let $k\ge3$, $\omega_j=2\pi j/k$, $\theta_j=\Theta(\omega_j)$, and let $w$ be a nonzero
measure on $\{\omega_j:1\le j<k/2\}$ with weights $w_j$. Then $w$ solves $\FP(\Theta_q)$
if and only if $\sum_jw_j\Lambda_n(\omega_j)\ge0$ for $n=1,\dots,k-1$ and
$\sum_jw_j\cos\theta_j\ge0$. In that case there are $f\in\cF(\R^{2r+1})$, with $r$
the number of nonzero weights, and $x_0$ whose trajectory converges to a
$k$-periodic orbit and stays at distance at least $c\norm{x_0-x^\star}>0$ from
$x^\star$. In particular $\mathcal T_k\subset\mathcal D$.
\end{proposition}

\begin{proof}
The lag values depend only on $n\bmod k$, with value zero
in residue zero. Reweight by $|g-2is|^{-1}$ and use start-up
realization: the iterates are
$((\sqrt{\nu_j}e^{iN\omega_j})_j,p_NY)$, $p_N\to0$.
The planar part has fixed positive norm.
An atom at $\pi$ for even $k$ can be deleted by its nonpositive
lag and minimizer contributions.
\end{proof}
\begin{lemma}[Pinning identity]\label{lem:pinid}
For a curve $\Theta$, $k\ge1$, $r,m\in\Z$ and $\omega\in[0,\pi]$, with $\Lambda_0:\equiv0$,
\[
 \Lambda_{r+mk}(\omega)=\cos\Theta(\omega)-\cos\bigl(mk\omega-\psi_r(\omega)\bigr),\qquad \psi_r=\Theta-r\omega .
\]
In particular, for $k\ge3$, $\Lambda_n=\Lambda_r$ on $\Omega_k$ for $n\equiv r\pmod k$, and, if $\Theta$ is differentiable, for $n=r+mk$,
$\partial_\omega\Lambda_n=\partial_\omega\Lambda_r+mk\sin(r\omega-\Theta)$ on $\Omega_k$.
\end{lemma}

\begin{proof}
Substitute $n=r+mk$ in \eqref{eq:Lambda} and differentiate.
On the grid, $mk\omega\in2\pi\Z$.
\end{proof}

\subsection{Smoothing and sampling}\label{app:levels-sampling}
Extend $\Theta=\Theta_{q,s}$ oddly to
$\T=\R/(2\pi\Z)$ and put $\zeta=e^{i\Theta}$.
The phase is represented locally by real-valued arguments of
$\zeta$.
Lemma~\ref{lem:ends} shows that $\zeta$ is $C^1$ with
Lipschitz derivative, and analytic between $0,\pi$.
Put $M_1=\sup|\Theta'|$,
$M_2=\sup_{\T\setminus\{0,\pi\}}|\Theta''|$.
Symmetrizing a measure and folding it back preserve all constraints.
Write
\[
 F_\mu(n)=\int_\T e^{in\omega}\bar\zeta(\omega)\,d\mu,
 \quad c_0(\mu)=\int\cos\Theta\,d\mu,
 \quad C_n(\mu)=\Rea(F_\mu(0)-F_\mu(n)).
\]
A feasible probability measure is nondegenerate of order $n_0$
when $c_0>0$ and $C_n>0$ for $0<|n|<n_0$.
Choose an even nonnegative smooth bump $\psi$ of integral $1$
supported in $[-1/4,1/4]$.
Set $\varphi=\psi*\psi$,
$K_\varepsilon=\varepsilon^{-1}(\varphi*\varphi)(\cdot/\varepsilon)$,
periodized for $0<\varepsilon<\pi$, and
$m_\varphi=\int t^2\varphi(t)\,dt$,
$m_K=\int t^2(\varphi*\varphi)(t)\,dt$.
Then $\widehat K_\varepsilon(\xi)=\hat\varphi(\varepsilon\xi)^2\in[0,1]$.
\begin{lemma}[Smoothing]\label{lem:A-smooth}
For every symmetric probability measure $\mu$ on $\T$ and every $n\in\Z$, the density $f=K_\varepsilon*\mu$
satisfies
\[
 C_n(f)\ge\widehat K_\varepsilon(n)\,C_n(\mu)+\bigl(1-\widehat K_\varepsilon(n)\bigr)c_0(\mu)-D_0-3D_n-M_2m_K\varepsilon^2,
\]
where $D_n=\sup_{|\xi-n|\le M_1}|\widehat K_\varepsilon(\xi)-\widehat K_\varepsilon(n)|\le M_1\min\{2m_\varphi\varepsilon^2(|n|+M_1),\varepsilon\}$.
Moreover $\int\cos\Theta\,f\ge c_0(\mu)-D_0-M_2m_K\varepsilon^2/2$.
\end{lemma}

\begin{proof}
Use Fubini and the uniform local Taylor estimate
$\Theta(v+t)=\Theta(v)+\Theta'(v)t+r$,
$|r|\le M_2t^2/2$.
The modulated convolution differs from
$e^{inv}\bar\zeta(v)\widehat K_\varepsilon(n-\Theta'(v))$
by at most $M_2m_K\varepsilon^2/2$.
Replace this kernel value by $\widehat K_\varepsilon(n)$
and use $|\Lambda_n|\le2$: the lag error is
$D_0+3D_n+M_2m_K\varepsilon^2$.
For the derivative bound, use
$|\hat\varphi'(\eta)|\le\min(m_\varphi|\eta|,1/2)$
in $\widehat K_\varepsilon'=2\varepsilon\hat\varphi\hat\varphi'$.
The same estimate at $n=0$ gives the minimizer inequality.
\end{proof}
\begin{lemma}\label{lem:A-gain}
There is $c_\varphi>0$ with $1-\widehat K_\varepsilon(\xi)\ge c_\varphi\min\{\varepsilon^2\xi^2,1\}$ for all real $\xi$.
\end{lemma}

\begin{proof}
For $|\eta|\le2\pi$,
$1-\hat\varphi(\eta)\ge2m_\varphi\eta^2/\pi^2$
by $1-\cos u\ge2u^2/\pi^2$ on $[-\pi,\pi]$.
Away from zero, strict loss, continuity and the
Riemann--Lebesgue lemma give a positive lower bound.
Since $1-\hat\varphi^2\ge1-\hat\varphi$, combine the two ranges.
\end{proof}
Sampling on $L$ equally spaced points identifies Fourier frequencies
that differ by multiples of $L$, producing the aliasing estimate below.
\begin{lemma}[Aliasing]\label{lem:A-alias}
Let $f\ge0$ be a smooth even function on $\T$ and $\mu_L=\frac{2\pi}L\sum_{j=0}^{L-1}f(2\pi j/L)\,\delta_{2\pi j/L}$.
Then $\mu_L$ is symmetric, $F_{\mu_L}$ is $L$-periodic in $n$, and there is $C(f,\Theta)$ with
$|F_{\mu_L}(n)-F_f(n)|\le C(f,\Theta)L^{-2}$ for $|n|\le L/2$.
\end{lemma}

\begin{proof}
For $h=f\bar\zeta$, the derivative is Lipschitz and
the periodic boundary terms cancel. Two integrations by parts
give $|\hat h(m)|\le C(1+|m|)^{-2}$.
Its Fourier series therefore converges absolutely, and finite-grid
orthogonality gives the aliasing formula
\[
 F_{\mu_L}(n)=2\pi\sum_{m\in\Z}\hat h(mL-n),
 \qquad F_f(n)=2\pi\hat h(-n).
\]
For $|n|\le L/2$, the omitted terms are bounded by
$CL^{-2}\sum_{m\ne0}(|m|-1/2)^{-2}$.
This is the standard trapezoidal aliasing argument
of~\cite[\S3]{TW14}, here using the proved polynomial
Fourier decay. Symmetry and periodicity follow directly from the weights.
\end{proof}
\begin{lemma}[Sampling]\label{lem:A-sampling}
Let $\Theta=\Theta_{q,s}$ and let $\mu$ be a feasible symmetric probability measure for $\FP(\Theta)$ that is
nondegenerate of order $n_0$, where $n_0$ is the least integer exceeding the larger root of
$P(x)=c_\varphi c_0x^2-6M_1m_\varphi x-(8M_1^2m_\varphi+M_2m_K)$, $c_0=c_0(\mu)$. Then there is $L_0$ such that
$\PP_L(\Theta)$ is feasible for every $L\ge L_0$.
\end{lemma}

\begin{proof}
Choose one small $\varepsilon\le1/n_0$ and apply the
smoothing estimate in three ranges.
For $|n|\ge1/\varepsilon$, the kernel loss gives a fixed
positive fraction of $c_0$.
For $n_0\le|n|<1/\varepsilon$, it gives the lower bound
$\varepsilon^2P(|n|)\ge\varepsilon^2P(n_0)>0$.
For the finitely many $0<|n|<n_0$, nondegeneracy survives
the $O(\varepsilon^2)$ error.
The smoothed density thus has one positive lag margin $\delta$
for all signed $n\ne0$, and a positive minimizer margin.
Lemma~\ref{lem:A-alias} preserves these margins on every
sufficiently fine grid: choose balanced representatives
$|n|\le L/2$ for nonzero residues; residue zero is identically zero.
Fold the sampled measure and delete its endpoints.
\end{proof}
\begin{proof}[Proof of Theorem~\ref{thm:sampling}]

Use the monotone-form measure $\nu$ at a feasible profile.
If $\int(1-\cos n\omega)\,d\nu=0$ for some $n\ne0$,
its support lies on a periodic grid, which remains feasible
under every weak increase of $\kappa$.
Otherwise every such integral is positive.
A uniform strict profile increase gives strictly positive
lag and minimizer increments; the reciprocal-sine weighting
is finite by Proposition~\ref{prop:monoform}.
The resulting measure is nondegenerate of every finite order,
so Lemma~\ref{lem:A-sampling} gives every sufficiently large period.
Apply the same argument at the attained shift $\sigma(q)$
and any additional $\varepsilon>0$ to obtain
$\inf_L\sigma_L=\sigma$.
Finally $\mathcal T_L\subset\mathcal D$, $\mathcal D$ is closed,
and increasing $\kappa$ by arbitrarily small amounts gives
periodic feasible tunings. This proves the closure identity.

\end{proof}
\begin{proof}[Proof of Theorem~\ref{thm:global}]

On a fixed finite grid, compactness gives lower semicontinuity
of $\sigma_L$. A positive shift makes all constraints strict
except those identically zero on the same support;
these remain zero for every tuning. This gives upper semicontinuity.
Thus $\sigma=\inf_L\sigma_L$ is upper semicontinuous;
closedness of its epigraph gives lower semicontinuity.
Lemma~\ref{lem:kappa} gives strict decrease in $\kappa$.
As $\kappa\downarrow1$, $B\downarrow0$ and the numerator
of $\Xi$ tends uniformly to $-N(c)<0$, hence
$h_q\to-\infty$ uniformly and $\sigma\to+\infty$.
The zero-crossing graph of this continuous strictly decreasing
function is continuous into $(1,\infty]$: at a finite test
$\kappa$, its positive or negative sign persists locally.

\end{proof}
\begin{corollary}[Periodic thresholds of the condition number]\label{cor:kappaL}
For $L\ge3$ put $\kappa_L(\alpha L,\beta)=\inf\{\kappa>1:\ \sigma_L(\kappa,\alpha L,\beta)\le0\}\in(1,\infty]$. Then
$\{(\kappa,\alpha L,\beta)\in\mathcal T_L\}=\{\kappa\ge\kappa_L(\alpha L,\beta)\}$ and
\[
 \kappa_\partial=\inf_{L\ge3}\kappa_L .
\]
\end{corollary}
\begin{proof}[Proof of Corollary~\ref{cor:kappaL}]

Each $\sigma_L$ is continuous and strictly decreasing in
$\kappa$, so its sublevel set is the stated closed upper ray.
The inequality $\sigma\le\sigma_L$ gives
$\kappa_\partial\le\kappa_L$.
For $\kappa>\kappa_\partial$, strict decrease gives
$\sigma(\kappa)<0$; the infimum formula produces an $L$
with $\sigma_L(\kappa)<0$. Let $\kappa\downarrow\kappa_\partial$.
The same inequalities include the infinite case.

\end{proof}

\subsection{Limiting symbols}
\begin{theorem}[Limit symbol]\label{thm:limitsymbol}
For every tuning $q$ there are $\lambda_n\ge0$, $n\ne0$, and $\gamma\ge0$ with
$\sum_n\lambda_n+\gamma=1$ such that
$S_q=\sum_n\lambda_n\Lambda^{q,\sigma(q)}_n+\gamma\cos\Theta_{q,\sigma(q)}$ is continuous,
$S_q\le0$ on $[0,\pi]$, $\max S_q=0$, and every threshold extremal $w$ is carried by
$\{S_q=0\}$ and satisfies $\int\Lambda^{q,\sigma(q)}_n\,dw=0$ whenever $\lambda_n>0$, and
$\int\cos\Theta_{q,\sigma(q)}\,dw=0$ if $\gamma>0$. Here $\Lambda^{q,s}_n$ are the lag functions
of $\Theta_{q,s}$.
\end{theorem}
The proof follows Lemma~\ref{lem:A-limit}.
\begin{lemma}[Limits of symbols]\label{lem:A-limit}
Let $\mathcal N\subset\Z\setminus\{0\}$, let $\Theta_j\to\Theta$ uniformly on $[0,\pi]$, all curves, and let
$S_j=\sum_{n\in\mathcal N}\lambda^{(j)}_n\Lambda^{(j)}_n+\gamma_j\cos\Theta_j\le0$ on $[0,\pi]$, with finitely many
$\lambda^{(j)}_n\ge0$, $\gamma_j\ge0$ and $\sum_n\lambda^{(j)}_n+\gamma_j=1$. Then, along a subsequence,
$\lambda^{(j)}_n\to\lambda_n$ for every $n$, $\sum_n\lambda_n\le1$, and
$S=\sum_{n\in\mathcal N}\lambda_n\Lambda_n+\gamma_*\cos\Theta$, $\gamma_*=1-\sum_n\lambda_n$, is continuous with $S\le0$ on $[0,\pi]$.
If moreover $\FP[\mathcal N](\Theta)$ is feasible, then $\max S=0$, and every solution $w$ of $\FP[\mathcal N](\Theta)$
is carried by $\{S=0\}$ and satisfies $\int\Lambda_n\,dw=0$ whenever $\lambda_n>0$, and $\int\cos\Theta\,dw=0$
if $\gamma_*>0$.
\end{lemma}

\begin{proof}
Take a coordinatewise convergent subsequence of the coefficients.
Fatou gives $\sum\lambda_n\le1$, and $|\Lambda_n|\le2$
gives uniform convergence of the limiting series.
For any nonnegative $\phi\in L^1$, the functions
$e^{-i\Theta_j}\phi$ form a compact subset of $L^1$.
The Riemann--Lebesgue estimate is uniform on this compact set
(approximate it by a finite $L^1$ net).
Split the lag sum into finite and high-frequency parts to obtain
$\int S\phi\le0$, hence $S\le0$ by continuity.
The missing coefficient mass contributes exactly
$(1-\sum\lambda_n)\cos\Theta$.
If $\max S<0$, a finite truncation is a strict symbol,
contradicting feasibility. Finally integrate $S$ against any
feasible measure: every weighted slack is nonnegative, so all
positive-coefficient slacks vanish and the measure is carried
by $\{S=0\}$.
\end{proof}
\begin{proof}[Proof of Theorem~\ref{thm:limitsymbol}]

Take normalized strict symbols at shifts
$s_j\uparrow\sigma(q)$, available by Lemma~\ref{lem:alt}.
The curves converge uniformly by Lemma~\ref{lem:ends}.
Apply Lemma~\ref{lem:A-limit} and attainment of the threshold.

\end{proof}
\section{Farey programs and polynomial certificates}\label{app:farey}
Throughout, the exponent is zero and $K\ge2$.
\begin{lemma}[Positions]\label{lem:positions}
Let $\Theta$ be a curve, $k\ge3$, $\omega_i=2\pi i/k$, and let
$\mathsf a_i=k\Theta(\omega_i)/\pi\in(0,k]$ be the phase index of $\omega_i$. For
$\mathsf a\in(0,k]$ and $\varrho\in\{0,\dots,k-1\}$ put
$F_{\mathsf a}(\varrho)=2\sin(\pi\varrho/k)\sin(\pi(\varrho-\mathsf a)/k)$.
\begin{enumerate}[label=\textup{(\alph*)}]
\item $\Lambda_n(\omega_i)=F_{\mathsf a_i}(in\bmod k)$ for all $n\in\Z$.
\item For $1\le\varrho\le k-1$, $F_{\mathsf a}(\varrho)<0$ if and only if $\varrho<\mathsf a$,
and $F_{\mathsf a}(\varrho)=0$ if and only if $\varrho=\mathsf a$; $F_{\mathsf a}(0)=0$.
\item If $\gcd(j,k)=1$, then $n\mapsto jn\bmod k$ is a
bijection of $\Z/k\Z$; position $1$ is the class $b+k\Z$ of the parent lags of
Theorem~\ref{thm:farey}, position $0$ is the class $k\Z$, and position $j$ is the class of the
lag $1$. If $p\equiv jn$ is the position of the lag $n$, then $in\equiv\tau_ip\pmod k$ with
$\tau_i\equiv ib$.
\end{enumerate}
\end{lemma}

\begin{proof}
Reduce $in$ modulo $k$ in \eqref{eq:Lambda}.
The first sine in $F_{\mathsf a}(\varrho)$ is positive for
$0<\varrho<k$, so its second factor determines the sign.
Multiplication by the unit $j\bmod k$ permutes residues;
its inverse gives the position formula.
\end{proof}
\begin{theorem}[Lattice lines]\label{thm:lattice}
Let $\Theta$ be a curve.
\begin{enumerate}[label=\textup{(\alph*)}]
\item $\Lambda_n<0$ at a point of $\mathcal C$ if and only if $0<n\xi-\ell\eta<1$ for
some integer $\ell$. The sign of $\Lambda_n$ changes only across the lattice
lines $n\xi-\ell\eta=0$ and $n\xi-\ell\eta=1$.
\item The points $(2\pi j/k,\pi/k)$ become the lattice points $(j,k)$. The
condition $\Theta(2\pi j/k)\le\pi/k$ says that $\mathcal C$ meets the ray through
$(j,k)$ at or beyond $(j,k)$, and $\Omega_1$ is the set of tunings whose curve
lies in the half plane $\xi<1$.
\end{enumerate}
\end{theorem}
\begin{proof}[Proof of Theorem~\ref{thm:lattice}]

Substitution in \eqref{eq:Lambda} gives
$\Lambda_n=2\sin(\pi n\xi/\eta)
\sin(\pi(n\xi-1)/\eta)$.
Since $0<1/\eta\le1$, the two sines have opposite signs
exactly when $0<n\xi-\ell\eta<1$ for an integer $\ell$.
At a rational frequency,
$(\xi,\eta)=\pi(j,k)/(k\Theta)$, proving (b).

\end{proof}
\begin{theorem}[Farey parents and first violations]\label{thm:farey}
\begin{enumerate}[label=\textup{(\alph*)}]
\item The lattice lines $n\xi-\ell\eta=1$ through $(j,k)$ have $n\equiv j^{-1}\pmod k$.
Among them, those with $|n|<k$ are the lower lag $n=b$, whose line also
passes through $(c,d)$, and the upper lag $n=-d$, whose line also passes
through $(a,b)$.
\item For a single atom at $\omega=2\pi j/k$ with phase $\Theta\in(0,\pi)$, the
lag $n$ is violated if and only if $jn\bmod k\in\{1,\dots,\lceil k\Theta/\pi\rceil-1\}$.
No lag is violated if and only if $\Theta\le\pi/k$; when $\Theta$ exceeds
$\pi/k$, the first violated lags are $n\equiv b\pmod k$, the class of the parent
lags $b$ and $-d$. In particular, for every curve, $\delta_{2\pi j/k}$ solves
$\FP(\Theta)$ if and only if $\Theta(2\pi j/k)\le\pi/k$.
\end{enumerate}
\end{theorem}
\begin{proof}[Proof of Theorem~\ref{thm:farey}]

The incidence equation is $nj-\ell k=1$.
The Farey determinant identities give the solutions
$(n,\ell)=(b,a),(-d,-c)$; all others differ by $(k,j)$.
The mediant relations and neighbor identities are the standard
unimodular characterization of Farey parents~\cite[Chapter~III]{HW}.
For (b), write $r=jn\bmod k$.
The inequality
$\cos(2\pi r/k-\Theta)>\cos\Theta$
is equivalent to $0<r<k\Theta/\pi$.
The first residue entering this interval is $r=1$,
whose lag class is $b\bmod k$.
When no lag is violated, $\Theta\le\pi/k$ also ensures the
minimizer constraint. At $\Theta=\pi$ that constraint fails.

\end{proof}
\begin{lemma}[The tongue of a fraction]\label{lem:tongue}
For every tuning $q=(\kappa,\alpha L,\beta)$ and $c\in[-1,1]$,
\begin{equation}\label{eq:Xikappa}
 \Xi_q(c)=\Xi_\infty(c)-\frac{N(c)}{\alpha L\,(1-\beta)(\kappa-1)},
\end{equation}
where $N$ is the polynomial of Lemma~\ref{lem:kappa} and
\[
 \Xi_\infty(c)=\frac{(1-c)\bigl(\alpha L\,(1+\beta)-2(1-2\beta c+\beta^2)\bigr)}{\alpha L\,(1-\beta)} .
\]
Consequently $\Theta_q(\omega_0)\le c'$ if and only if $\kappa\ge\kappa_{j/k}(\alpha L,\beta)$, where
\begin{equation}\label{eq:kappajk}
 \kappa_{j/k}(\alpha L,\beta)=1+\frac{\bigl(\alpha L-(1+\beta)(1-c_0)\bigr)^2+(1-\beta)^2s_0^2}
 {(1-c_0)\bigl(\alpha L\,(1+\beta)-2(1-2\beta c_0+\beta^2)\bigr)-h_0\,\alpha L\,(1-\beta)}
\end{equation}
when the denominator is positive, and $\kappa_{j/k}=\infty$ otherwise. In particular
$\kappa_\partial\le\kappa_{j/k}$ for every fraction $j/k$.
\end{lemma}

\begin{proof}
At fixed $y,\beta$, the numerator of \eqref{eq:Xi} is affine
in $x=y/\kappa$. Separate its constant part and its
$(\kappa-1)^{-1}$ part to obtain \eqref{eq:Xikappa}.
Since $N>0$, solving $\Xi(c_0)\ge h_0$ gives
\eqref{eq:kappajk}, including the empty case when its denominator
is nonpositive. The exact threshold criterion gives
$\kappa_\partial\le\kappa_{j/k}$.
\end{proof}
\begin{proposition}[Farey localization]\label{prop:farey}
Let $q\in I_K$. Then $\bar\Delta\subset(2\pi/(K+1),2\pi/K)$. If $w$ solves $\PP_L(\Theta_q)$, then
$w$ has an atom $2\pi j/k$ with $\gcd(j,k)=1$, $k\mid L$, $1/(K+1)<j/k<1/K$, $k\ge2K+1$ and
$\Theta(2\pi j/k)\le\pi j/k$; moreover $\Theta(2\pi j/k)>\pi/k$ unless $\delta_{2\pi j/k}$ alone
solves $\FP(\Theta_q)$.
\end{proposition}
The proof follows Definition~\ref{def:anchor}.
\begin{definition}[Anchors]\label{def:anchor}
The anchors of a solution $w$ of the full program are its atoms in the
closed dip $\bar\Delta$.
\end{definition}

\begin{proof}[Proof of Proposition~\ref{prop:farey}]
The closed dip lies strictly in the window by the gap definition
and Corollary~\ref{cor:dip}.
The dip lemma supplies a grid atom in it; reduce its fraction.
Any fraction strictly between the Farey neighbors
$1/(K+1),1/K$ has denominator at least $2K+1$.
The last assertion is the single-atom criterion.
\end{proof}
In the positions of an anchor, the complete periodic program is
\begin{equation}\label{eq:posprog}
 \sum_iw_iF_{\mathsf a_i}(\tau_ip)\ge0\quad(p\in\Z/k\Z),
 \qquad\sum_iw_i\cos\Theta(\omega_i)\ge0,
 \qquad\tau_i\equiv ib\pmod k.
\end{equation}
The anchor has $\tau_j=1$ and violates precisely the lag constraints
at positions $1,\ldots,\lceil\mathsf a_j\rceil-1$, where a position
is the residue $jn\bmod k$.
At the lag-one position $j$, its value is nonnegative exactly
when the anchor lies in the closed dip.

\subsection{Fans and two-atom compensation}
The fans of $j/k$ are the integer points
$(a,b)+m(j,k)$ and $(c,d)+m(j,k)$, $m\ge1$,
on $k\xi-j\eta=-1$ and $k\xi-j\eta=1$.
Their consecutive determinants are $\mp1$, so the members are
reduced Farey neighbors and approach $j/k$ from opposite sides.
For the endpoints of $I_K$ their lines are
$U:\eta=K\xi+1$ and $V:(K+1)\xi-\eta=1$.
In angle coordinates these are
$\Theta=\pi-K\omega/2$ and
$\Theta=(K+1)\omega/2-\pi$.
At each fan fraction their height is $\pi/k'$;
crossing at or beyond that point is precisely the single-atom criterion.
On the window both heights are at most $\pi/2$.
Corollary~\ref{cor:dip} thus makes their below-line sets
$\Delta_U,\Delta_V$ intervals.
A level-two gap has a nonempty such interval whose closure
lies strictly between consecutive points of the corresponding fan.
\begin{lemma}[Fan geometry]\label{lem:fan}

Let $K\ge2$ and $k_p=pK+1$. The fractions $(p-1)/k_{p-1}<p/k_p<1/K$ are Farey
neighbors, and the Farey parents of $p/k_p$ are $(p-1)/k_{p-1}$ and $1/K$.
The points $(p,k_p)$ lie on the line $U:\eta=K\xi+1$, which is the lattice line
$-K\xi+\eta=1$ of the lag $-K$. The lattice line of the lag $mK+1$ through $(1,K)$
is $(mK+1)\xi-m\eta=1$, and it passes through $(m+1,k_{m+1})$. Modulo $k_p$,
$-pK\equiv1$ and $(p-1)K+1\equiv-K$.
\end{lemma}

\begin{proof}
The parent determinant is
$pk_{p-1}-(p-1)k_p=1$; the mediant relations identify the
parents. Substitution proves the line equations and congruences.
\end{proof}
For integers $\tau$, write $r_\tau(p)\in\{0,\ldots,k-1\}$
for the residue of $p\tau$ modulo $k$.
\begin{lemma}[Record lows and Farey parents]\label{lem:record}
Let $2\le j<k$, $\gcd(j,k)=1$, $\tau\in\Z$, and let $i=r_\tau(j)$ and $m=(j\tau-i)/k$, so that
$j\tau=mk+i$. Assume $i\ge1$. Then $r_\tau(p)>i$ for all $1\le p\le j-1$ if and only if
$\gcd(m,j)=1$ and $i\,r'<k$, where $r/r'$ is the right Farey parent of $m/j$, that is,
the fraction with $rj-mr'=1$ and $1\le r'\le j-1$. In that case
\[
 \min_{1\le p\le j-1}r_\tau(p)=\frac{k+l'i}{j},\qquad l'=j-r',
\]
and the minimum is attained at $p=l'$ only.
\end{lemma}

\begin{proof}
Use the determinant-$k$ lattice
$\Gamma=\{(p,y):y\equiv p\tau\pmod k\}$.
If $g=\gcd(m,j)>1$, the point $(j/g,i/g)$ violates the
record condition. Otherwise the Farey identities give a lattice basis
\[
 u=\left(l',\frac{k+l'i}{j}\right),\qquad
 v=\left(r',\frac{r'i-k}{j}\right),
 \quad u+v=(j,i),\quad\det(u,v)=-k.
\]
If $ir'\ge k$, $v$ violates the condition.
If $ir'<k$, then $u_2>i>v_2$ with $v_2<0$.
For $au+bv$ with first coordinate in $[1,j-1]$,
either $a\ge1,b\le0$ and its second coordinate is at least
$u_2$, uniquely at $u$, or $a\le0,b\ge1$ and it is negative.
This proves both the equivalence and the unique minimum.
\end{proof}
The next result identifies the admissible companion frequencies and
mass ratios for two-atom measures on the periodic grid.
\begin{theorem}[Two-atom combs]\label{thm:twoatom}
Let $2\le j<k/2$, $\gcd(j,k)=1$, $b\equiv j^{-1}\pmod k$, $\omega_a=2\pi j/k$ with phase index
$\mathsf a=k\Theta(\omega_a)/\pi$, and $\omega_c=2\pi i/k$, $1\le i<k/2$, $i\ne j$, with phase index
$\mathsf a_c=k\Theta(\omega_c)/\pi$. Assume that for some integer $t$ with $2\le t\le j$
\[
 t-1<\mathsf a\le t\qquad\text{and}\qquad i<\mathsf a_c,
\]
that is, the anchor lies in the closed dip with bad positions $1,\dots,t-1$, and the
companion lies outside the closed dip. Put $\tau\equiv ib\pmod k$ and $m=(j\tau-i)/k$ with
$0\le\tau<k$, so that $r_\tau(j)=i$.
\begin{enumerate}[label=\textup{(\alph*)}]
\item There are $w_0,w_c>0$ such that $w=w_0\delta_{\omega_a}+w_c\delta_{\omega_c}$ satisfies all lag
constraints of $\FP(\Theta)$ if and only if
 \begin{equation}\label{eq:sel}
 r_\tau(p)>\mathsf a_c\quad(1\le p\le t-1)
 \end{equation}
 and $\underline{\mathsf r}\le\overline{\mathsf r}$, where
 \[
 \begin{gathered}
 \underline{\mathsf r}=\max_{1\le p\le t-1}\frac{-F_{\mathsf a}(p)}{F_{\mathsf a_c}(r_\tau(p))},\\
 \overline{\mathsf r}=\min\Bigl\{\frac{F_{\mathsf a}(p)}{-F_{\mathsf a_c}(r_\tau(p))}:\ 1\le p\le k-1,\ 1\le r_\tau(p)<\mathsf a_c\Bigr\};
 \end{gathered}
 \]
 the admissible ratios are then $w_c/w_0\in[\underline{\mathsf r},\overline{\mathsf r}]$. Together with
 $w_0\cos\Theta(\omega_a)+w_c\cos\Theta(\omega_c)\ge0$ this characterizes the two-atom solutions of
 $\PP_k$ on $\{\omega_a,\omega_c\}$.
\item \textup{(Farey selection rule.)} For $t=j$, condition \eqref{eq:sel} holds if and only
if $\gcd(m,j)=1$, $ir'<k$ and $\mathsf a_c<(k+l'i)/j$, where $l/l'<m/j<r/r'$ are the Farey
parents of $m/j$. For $t<j$, \eqref{eq:sel} is the condition
$\min_{1\le p\le t-1}r_\tau(p)>\mathsf a_c$ and is implied by the case $t=j$.
\item \textup{(Two fan sequences.)} For the fan point $p/(pK+1)$ with the companion $i=p-1$,
$\tau\equiv K+1$, and \eqref{eq:sel} holds for every $t\in[2,p]$ if and only if $\mathsf a_c<K+1$;
this forces $p\le K+1$, and then every lag is decided by (a). For
$(\ell+1)/(K+\ell(K+1))$ with the companion $i=\ell$, $\tau\equiv-K$, and \eqref{eq:sel} holds if
and only if $\mathsf a_c<k-(t-1)K$; for $t=\ell+1$ this is $\mathsf a_c<K+\ell$.
\end{enumerate}
\end{theorem}

\begin{proof}
At anchor position $p$, the two lag values are
$F_{\mathsf a}(p)$ and $F_{\mathsf a_c}(r_\tau(p))$.
The first is negative exactly at $1\le p<t$.
Positive compensation forces \eqref{eq:sel}; under that condition
the negative anchor values give the lower ratio bounds and
the negative companion values give the upper bounds.
All other constraints have two nonnegative terms.
The upper set is nonempty because it contains $p=j$;
the lower bound is strictly positive. Adding the minimizer
constraint proves (a), including equality cases.
For $t=j$, Lemma~\ref{lem:record} is exactly (b).
For the first fan sequence,
$\tau=K+1\bmod k$ and $r_\tau(p')=(K+1)p'$
for $1\le p'<p$ once $\mathsf a_c<K+1$;
this also forces $p\le K+1$.
For the second fan sequence, $\tau=-K\bmod k$ and
$r_\tau(p')=k-Kp'$ for $1\le p'\le\ell$.
Taking the minimum gives (c).
\end{proof}
\begin{theorem}[Fan tongues]\label{thm:fan-div}
Let $p\ge2$ and $k=k_p=pK+1$.
\begin{enumerate}[label=\textup{(\alph*)}]
\item If there are $w_j\ge0$, $1\le j<k/2$, not all zero, with
$\sum_jw_j\Lambda_n(\omega_j)\ge0$ for $n=1,\dots,k-1$ and
$\sum_jw_j\cos\theta_j\ge0$, then there are $f\in\cF(\R^{2r+1})$, $r$ the number
of nonzero weights, and $x_0$ whose trajectory converges to a $k$-periodic
orbit; in particular $\rho\ge1$.
\item Let $p\le K+1$, and suppose $\pi/k<\theta_p<\pi p/k$ and
$\pi(p-1)/k<\theta_{p-1}<\pi(K+1)/k$.
Then $\Lambda_1(\omega_p)>0>\Lambda_{-K}(\omega_p)$ and
$\Lambda_{-K}(\omega_{p-1})>0>\Lambda_1(\omega_{p-1})$, and the lags $1$ and $-K$
are satisfied by some measure on $\{\omega_{p-1},\omega_p\}$ if and only if
\begin{multline*}
 \Delta_p:=\sin\tfrac{\pi p}k\,\sin\bigl(\tfrac{\pi p}k-\theta_p\bigr)\,
 \sin\tfrac{\pi(K+1)}k\,\sin\bigl(\tfrac{\pi(K+1)}k-\theta_{p-1}\bigr)\\
 -\sin\tfrac{\pi(p-1)}k\,\sin\bigl(\theta_{p-1}-\tfrac{\pi(p-1)}k\bigr)\,
 \sin\tfrac\pi k\,\sin\bigl(\theta_p-\tfrac\pi k\bigr)\ \ge0 .
\end{multline*}
If moreover $\theta_{p-1}\le\pi/2$, which is automatic for $p\ge3$, and the
lags $n\in\{2,\dots,k-1\}\setminus\{(p-1)K+1\}$ are satisfied for one ratio
$w_{p-1}/w_p$ of the admissible interval, then the conclusion of \textup{(a)}
holds.
\item For $p=2$, $k=2K+1$, and $\pi/k<\theta_1,\theta_2\le2\pi/k$, the remaining
lags are satisfied automatically, and the program of \textup{(a)} with the two
harmonics $\omega_1,\omega_2$ is feasible if and only if
\[
 2\cos\tfrac\pi k\cos\tfrac\pi{2k}\,\sin\bigl(\tfrac\pi k-\vartheta_2\bigr)\cos\bigl(\vartheta_1+\tfrac\pi{2k}\bigr)
 \ge\sin\tfrac\pi k\,\sin\vartheta_1\sin\vartheta_2,\quad
 \vartheta_j=\theta_j-\tfrac\pi k .
\]
\end{enumerate}
\end{theorem}

\begin{proof}
Part (a) is Proposition~\ref{prop:periodic}.
For (b), the four relevant lag values are
\[
\begin{array}{ll}
 \Lambda_1(\omega_p)=F_{k\theta_p/\pi}(p),&
 \Lambda_{-K}(\omega_p)=F_{k\theta_p/\pi}(1),\\
 \Lambda_1(\omega_{p-1})=F_{k\theta_{p-1}/\pi}(p-1),&
 \Lambda_{-K}(\omega_{p-1})=F_{k\theta_{p-1}/\pi}(K+1).
\end{array}
\]
The phase bounds give their opposite signs.
The two inequalities in the positive weight ratio are compatible
exactly when
$\Lambda_1(\omega_p)\Lambda_{-K}(\omega_{p-1})
 -\Lambda_1(\omega_{p-1})\Lambda_{-K}(\omega_p)\ge0$;
division by $4$ is the displayed $\Delta_p$.
The remaining residue conditions and phases at most $\pi/2$
give all other constraints.
For $p=2$, each atom violates only the residue with position
$1$, so all other residues are automatically nonnegative.
Use $\sin(2u)=2\sin u\cos u$ and
$\pi(K+1)/k=\pi/2+\pi/(2k)$ in $\Delta_2$ to obtain (c).
\end{proof}

\subsection{Dual certificates and pinning}
\begin{theorem}[Fan-lag certificates]\label{thm:fan-conv}
Assume \eqref{eq:quad} and let $K\ge2$. If there are finitely many weights
$\lambda_n\ge0$, $n\in\mathcal N_{\rm fan}(K)$, and $\gamma>0$ such that
\begin{multline}\label{eq:fanlag}
 \Xi(c)\Bigl[\sum_{m\ge0}\lambda_{mK+1}\bigl(1-T_{mK+1}(c)\bigr)+\sum_{m\ge1}\lambda_{-mK}\bigl(1-T_{mK}(c)\bigr)+\gamma\Bigr]\\
 <(1-c^2)\Bigl[\sum_{m\ge0}\lambda_{mK+1}U_{mK}(c)-\sum_{m\ge1}\lambda_{-mK}U_{mK-1}(c)\Bigr]\quad(-1\le c\le1),
\end{multline}
then $\rho<1$. The choice $\lambda_1=1$, $\lambda_{-K}=r\ge0$ gives \eqref{eq:twolag}.
\end{theorem}

\begin{proof}
Formula~\eqref{eq:Lambdapoly} identifies the symbol with
$\sin\Theta/\sin\omega$ times the left-minus-right difference.
It is negative on the open interval.
At both endpoints the lag terms are nonpositive and the minimizer
term equals $-\gamma<0$. Compactness gives a strict uniform
margin, and Theorem~\ref{thm:oneexp} gives convergence.
\end{proof}
\begin{lemma}[Pinning lags]\label{lem:pin}
Let $k=k_p$ and $\omega_j=2\pi j/k$. Then $-pK\equiv1$ and $(p-1)K+1\equiv-K\pmod k$,
so on every $k$-periodic spectrum the lags $1$ and $-pK$ give the same
constraint, and so do the lags $-K$ and $(p-1)K+1$. The differences
\[
 \mathsf D=\Lambda_{-pK}-\Lambda_1,\quad \mathsf E=\Lambda_{(p-1)K+1}-\Lambda_{-K}
\]
vanish at every frequency $\omega_j$, and there
\[
 \mathsf D'(\omega_j)=-k\sin(\omega_j-\theta_j),\quad \mathsf E'(\omega_j)=-k\sin(K\omega_j+\theta_j).
\]
\end{lemma}

\begin{proof}
Use the congruences of Lemma~\ref{lem:fan} and the derivative
formula of Lemma~\ref{lem:pinid}.
Transferring a weight between congruent lags preserves grid values
and adds the corresponding multiple of $\mathsf D$ or $\mathsf E$
to the off-grid symbol. Nonnegative transfers are exactly those
that keep the two affected weights nonnegative.
\end{proof}
If both contacts of an optimal two-lag symbol lie on $\Omega_L$
and conditions \textup{(E1)--(E4)} of Theorem~\ref{thm:exact}
hold, that symbol is optimal for the full program; the transfers
using $1\pm L,-K\pm L$ yield $L$-regularity.

\subsection{Exact arithmetic for the two rational examples}\label{app:farey-examples}
For an interval $[u,v]$, write
$P(u+(v-u)t)=\sum_{i=0}^d b_i\binom di t^i(1-t)^{d-i}$.
Nonnegative Bernstein basis functions sum to $1$, so
$\min b_i\le P\le\max b_i$~\cite{Farouki2012}.
The following sign checks use only this identity and rational Taylor
bounds for sine and cosine.
\begin{proof}[Proof of Proposition~\ref{prop:13}]

For the stated tuning,
\[
 \cL(c)=\frac{274486940640000c^2-89129341428596c+2155051415249}
 {110124036000000}.
\]
Its vertex lies in $(0,(\sqrt5-1)/4)$, has negative value,
and both endpoint values on that interval are positive.
Hence the tuning lies in $I_4$.
Its angle quadratic is
\[
 \Xi(c)=\frac{88897198775653}{91052250190902}
       +\frac{242832333719}{122711927481}c
       -\frac{26156560000}{8676600933}c^2.
\]
On $[-1,-1/2]$ and $[127/128,1]$, $-\Xi$ has positive
Bernstein coefficients. On each interval with consecutive endpoints
\[
 -\tfrac12,\ 0,\ \tfrac12,\ \tfrac34,\ \tfrac78,\
 \tfrac{15}{16},\ \tfrac{31}{32},\ \tfrac{63}{64},\ \tfrac{127}{128},
\]
the polynomial $3(1-c^2)-\Xi(c)^2$ has positive Bernstein
coefficients. Thus $\Theta>\pi/6$ throughout.
This excludes all denominators $k\ge6$; the four primitive
frequencies with $3\le k\le5$ are excluded by the same
quadratic signs and the $I_4$ window.
An irrational rotation also fails the single-atom constraints,
since its lag phases are dense modulo $2\pi$.

Put $c_j=\cos(2\pi j/13)$ and $s_j=\sin(2\pi j/13)$.
After division by the common positive factor $2(1-\beta)/B$,
the raw lag and minimizer slacks are
\[
 \sum_{j=2,3}v_j\bigl[\Xi(c_j)(1-c_{nj})-s_js_{nj}\bigr],
 \qquad\sum_{j=2,3}v_j\Xi(c_j),
 \quad (v_2,v_3)=(149,648)/797,
\]
with indices reduced modulo $13$.
All thirteen displayed quantities (the minimizer and
$n=1,\ldots,12$) are strictly positive.
For a reproducible rational check, enclose $\pi$ between
$314159265358979323846/10^{20}$ and
$314159265358979323847/10^{20}$, reduce angles to $[-\pi,\pi]$,
and use the Taylor polynomials through degree $32$ with their
Lagrange remainder bounds. Substitution into these finite formulas
gives strict positive interval lower endpoints.
Periodicity gives every signed lag; multiples of $13$ vanish.
Proposition~\ref{prop:periodic} realizes the stated two harmonics
in $\R^5$.

\end{proof}
\begin{proof}[Proof of Proposition~\ref{prop:conv-gap}]

Here
\[
 \Xi(c)=-\frac{620000}{208791}c^2
         +\frac{14967}{7733}c+\frac{203570}{208791},
 \qquad\cL(c)=\frac{620000c^2-195318c+5221}{250000}.
\]
The same vertex and endpoint test gives $I_4$.
Take $r=13/100$, $\gamma=1/1000$ in \eqref{eq:twolag}.
The right-minus-left difference is
\[
\begin{aligned}
 P(c)={}&-\tfrac{644800}{208791}c^6+\tfrac{23608}{7733}c^5
 +\tfrac{389324}{94905}c^4-\tfrac{34149683}{5219775}c^3\\
 &+\tfrac{1007042}{347985}c^2
 -\tfrac{92371789}{208791000}c+\tfrac{45613}{1898100}.
\end{aligned}
\]
Its degree-six Bernstein coefficients are strictly positive on
each interval $[-1+j/8,-1+(j+1)/8]$, $0\le j<16$.
Therefore $P>0$ on $[-1,1]$, and
Theorem~\ref{thm:fan-conv} gives $\rho<1$.

\end{proof}

Figure~\ref{fig:lattice} displays the lattice curve of Proposition~\ref{prop:13} and its endpoint fan.
\begin{figure}[tbp]
\centering
\includegraphics[width=\linewidth]{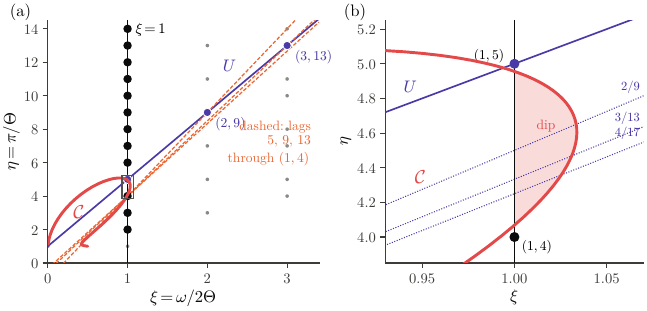}
\caption{The example of Proposition~\ref{prop:13} in lattice coordinates.
The line $U:\eta=4\xi+1$ carries the fan $p/(4p+1)$; dashed lines
correspond to its lower parent lags $5,9,13$. The right panel resolves
the dip and the fan directions $2/9,3/13,4/17$.}
\label{fig:lattice}
\end{figure}

\section{Boundary structure and finite certificates}\label{app:boundary}
For shifted curves, write $\Lambda^{q,s}_n=\Lambda_n[\Theta_{q,s}]$,
$\theta_\omega(q,s)=\Theta_{q,s}(\omega)$ and $\tau_q=-1/\Xi_q(1)>0$.
The lag functions are jointly real analytic through the endpoints, with
$\Lambda_n(0)=\Lambda_n'(0)=0$ and
$\Lambda_n''(0)=-n(n+2\tau_q)$. The calculation used throughout is
\begin{equation}\label{eq:shiftid}
 \frac{\sum_n\lambda_n\Lambda^{q,s}_n+\gamma\cos\Theta_{q,s}}{\sin\Theta_{q,s}}
 =\frac{\sum_n\lambda_n\Lambda^{q,s'}_n+\gamma\cos\Theta_{q,s'}}{\sin\Theta_{q,s'}}
 +(s-s')\Bigl[\sum_n\lambda_n(1-\cos n\omega)+\gamma\Bigr]
\end{equation}
on $(0,\pi)$. It follows directly from $\cot\Theta_{q,s}=h_q+s$.
In particular, if $S=\sum_n\lambda_n\Lambda^{q,s}_n$ and
$P=\sum_n\lambda_n(1-\cos n\omega)$, then
\begin{equation}\label{eq:ds}
 \partial_sS=(\partial_s\log\sin\Theta_{q,s})S+\sin\Theta_{q,s}P.
\end{equation}

\subsection{Regularity and continuation}\label{app:boundary-regularity}
\begin{definition}[Regular periodic point]\label{def:regular}
Let $L\ge3$. A pair $(q_*,s_*)$ is $L$-regular if there are a finite set
$S\subset\Omega_L\cap(0,\pi)$, a finite set $N\subset\Z\setminus\{0\}$ and weights $w^*\in\R^S_{>0}$,
$\lambda^*\in\R^N_{>0}$ with $\sum_\omega w^*_\omega=\sum_n\lambda^*_n=1$, such that, with $N_0=N\cap L\Z$,
$\mathcal A=(N\setminus N_0)\bmod L$ and $S^*=\sum_{n\in N}\lambda^*_n\Lambda^{q_*,s_*}_n$:
\begin{enumerate}[label=\textup{(R\arabic*)}]
\item primal: $w^*$ is feasible for $\PP_L(\Theta_{q_*,s_*})$; the constraints of the residues in
$\mathcal A$ are tight, the constraints of all other residues $\varrho\not\equiv0$ and the
minimizer constraint are strict; $|\mathcal A|=|S|$, and
$M_*=\bigl(\Lambda^{q_*,s_*}_\varrho(\omega)\bigr)_{\varrho\in\mathcal A,\ \omega\in S}$ has rank $|S|-1$;
\item dual: $S^*<0$ on $(0,\pi]\setminus S$, $\partial^2_\omega S^*<0$ on $S$,
$\sum_n\lambda^*_nn(n+2\tau_{q_*})>0$, and $P^*(\omega)=\sum_n\lambda^*_n(1-\cos n\omega)>0$ on $S$;
\item pinning rank: every class $N_\varrho=N\cap(\varrho+L\Z)$, $\varrho\in\mathcal A$, has at least
two elements, and the vectors $v_\varrho=\bigl(\sin(\varrho\omega-\theta_\omega(q_*,s_*))\bigr)_{\omega\in S}$,
$\varrho\in\mathcal A$, form a basis of $\R^S$ (for one fixed representative $\varrho$ of each class).
\end{enumerate}
A tuning $q_*$ is $L$-regular if $(q_*,\sigma(q_*))$ is.
\end{definition}

The values and first derivatives of $S^*$ vanish on $S$:
positivity of $w^*$ and complementarity give the values, and the global
sign gives the derivatives. The moment inequality is precisely
$S^{*\prime\prime}(0)<0$, controlling curvature at the endpoint.
If $|S|\ge2$, the condition $P^*>0$ on $S$ already follows from \textup{(R1)}.

\begin{proof}[Proof of Theorem~\ref{thm:modelock}]
Let $M(q,s)=(\Lambda^{q,s}_\varrho(\omega))_{\varrho\in\mathcal A,\omega\in S}$.
Its entries are real analytic by Lemma~\ref{lem:ends}. At the base point,
$M_*w^*=0$ and $\operatorname{rank}M_*=|S|-1$. Put
$\mu^*_\varrho=\sum_{n\equiv\varrho}\lambda_n^*>0$.
Since the zero-class lag functions vanish on $\Omega_L$,
$\mu^{*\top}M_*=(S^*(\omega))_{\omega\in S}$.
Every entry is nonpositive by (R2), while its product with the positive
vector $w^*$ is zero. Hence $S^*=0$ on $S$ and
$\mu^{*\top}M_*=0$. The adjugate of $M_*$ is therefore
$c\,w^*\mu^{*\top}$ for some $c\ne0$. Jacobi's formula and
\eqref{eq:ds} give
\[
 \partial_s\det M(q_*,s_*)
 =c\sum_{\omega\in S}w^*_\omega\sin\theta_\omega(q_*,s_*)P^*(\omega)\ne0.
\]
The analytic implicit function theorem
\cite[Section~2.3]{KrantzParks2002} gives an analytic $s(q)$ with
$\det M(q,s(q))=0$. A fixed nonzero minor keeps the rank equal to
$|S|-1$ after the neighborhood is reduced. Nonzero adjugate columns
and rows give analytic positive kernel vectors $w(q),\mu(q)$, normalized
to agree with $w^*,\mu^*$ at the base point. The inactive residue
constraints and minimizer constraint remain strict by continuity.
Thus $w(q)$ solves $\FP$ at shift $s(q)$ and
$\sigma(q)\le\sigma_L(q)\le s(q)$.

It remains to continue a dual symbol with fixed contacts. For each
nonzero class write
$\mu_\varrho=\sum_{n\equiv\varrho}\lambda_n$ and
$m_\varrho=\sum_{n\equiv\varrho}\lambda_n(n-\varrho)/L$.
Lemma~\ref{lem:pinid} gives
\begin{equation}\label{eq:slope}
 S'(\omega)=\sum_{\varrho\in\mathcal A}\mu_\varrho\Lambda_\varrho'(\omega)
 +L\sum_{\varrho\in\mathcal A}m_\varrho\sin(\varrho\omega-\Theta(\omega))
 +\sum_{n\in N_0}\lambda_n\Lambda_n'(\omega),\qquad\omega\in S.
\end{equation}
Keep the zero-class weights fixed at their original values and use
the continued $\mu(q)$. By (R3), the equations $S'|_S=0$ uniquely
determine analytic moments $m(q)$ near the base point. Each class has
two distinct lags, so its linear map from lag weights to its mass and
moment is onto. A fixed affine right inverse through the original
weights lifts $(\mu(q),m(q))$ to analytic positive weights $\lambda(q)$.
Normalize their sum. The resulting $S_q$ has zero values and slopes on
$S$.

Choose disjoint fixed neighborhoods of $S$ and $0$. On them the
negative second derivative of $S^*$ has a uniform margin, by (R2)
and the endpoint moment condition. On their compact complement,
$S^*$ has a strictly negative maximum. Uniform $C^2$ continuity
preserves both estimates. Since the continued values and slopes at
$S\cup\{0\}$ are zero, Taylor's theorem gives
$S_q\le0$ with precisely that zero set; also $P_q>0$ near $S$.
For $s<s(q)$ sufficiently close, \eqref{eq:shiftid} gives
\[
 \frac{\sum_n\lambda_n(q)\Lambda^{q,s}_n+\epsilon\cos\Theta_{q,s}}
 {\sin\Theta_{q,s}}
 =\frac{S_q}{\sin\Theta_{q,s(q)}}-(s(q)-s)P_q
       +\epsilon(h_q+s).
\]
Near $0$, $h_q+s<0$. Near $S$ the second term has a strict negative
margin; on the remaining compact set the first term does. The
positive part of $h_q+s$ is bounded there. Choosing $\epsilon>0$
smaller than these margins makes this expression negative throughout
$(0,\pi)$; the unscaled symbol equals $-\epsilon$ at $0$ and is at
most $-\epsilon$ at $\pi$. It is therefore strict. Monotonicity
excludes every smaller shift, proving
$\sigma(q)=\sigma_L(q)=s(q)$.

For any threshold measure $v$, feasibility and the dual sign give
$0\le\int S_q\,dv\le0$. Positivity of the dual weights forces every
lag in $N$ to be tight and $v$ to be carried by $S\cup\{0\}$.
The rank condition gives $v|_S=cw(q)$; the minimizer constraint excludes
$c=0$ and is exactly
$0\le v(\{0\})\le c\int\cos\Theta_{q,s(q)}\,dw(q)$.
Conversely every measure satisfying this description is feasible.
This proves (c).

Differentiate the determinant equation. Because the parameter
$(q,s)$ enters each column only through $h_q(\omega)+s$,
\[
 \nabla s(q)=
 -\frac{\sum_{\omega\in S}w_\omega\sin\theta_\omega P_q(\omega)
                         \nabla_qh_q(\omega)}
        {\sum_{\omega\in S}w_\omega\sin\theta_\omega P_q(\omega)}.
\]
All weights in this quotient are positive. Lemma~\ref{lem:kappa}
therefore gives $\partial_\kappa s<0$. The strict conditions (R1)--(R3)
persist, proving (e). A final application of the implicit function
theorem to $s(\kappa,\alpha L,\beta)=0$ gives the analytic boundary
graph and completes the proof.
\end{proof}

\begin{corollary}[Open plateaus]\label{cor:plateau}
The set of regular tunings is open. Its intersection with $\partial\mathcal D$ is relatively
open in $\partial\mathcal D$; on it $\kappa_\partial$ is real analytic, and on each connected component the
support $S$ of the threshold extremal, its least common denominator $L$ and the tight classes $\mathcal A$ are
constant, and $\kappa_\partial$ solves $\det\bigl(\Lambda_\varrho(\omega)\bigr)_{\varrho\in\mathcal A,\,\omega\in S}=0$ at the shift $0$. A plateau is such a component.
If the support has exactly one atom in the window $(2\pi/(K+1),2\pi/K)$, then on the
part of the plateau in $I_K$ this atom is the unique anchor.
\end{corollary}

\begin{proof}
Persistence gives openness and local constancy of the unique support
and tight classes, hence constancy on each component. The least period
is intrinsic by Theorem~\ref{thm:tightclass}(a). The determinant equation
and analyticity follow from Theorem~\ref{thm:modelock}. Every extremal
charges the dip, so a single support point in the gap window is its
unique anchor.
\end{proof}

\begin{corollary}[Periods on a plateau]\label{cor:periods}
Let $q\in\partial\mathcal D$ be regular with threshold extremal carried by $S$, and let $L$ be the
least common multiple of the denominators of the points of $S/2\pi$.
\begin{enumerate}[label=\textup{(\alph*)}]
\item $q\in \mathcal T_{L'}$ if and only if $L\mid L'$.
\item If $q\in I_K$, then the fan program of Theorem~\ref{thm:fan-div}\textup{(a)} of some period
$pK+1$ is feasible at $q$ if and only if $\gcd(L,K)=1$.
\end{enumerate}
\end{corollary}
\begin{proof}[Proof of Corollary~\ref{cor:periods}]
A multiple of $L$ contains the support grid. Conversely every feasible
periodic measure at threshold has the same positive support by
Theorem~\ref{thm:modelock}(c), so its period is divisible by $L$.
The congruence $pK\equiv-1\pmod L$ is solvable exactly when
$\gcd(K,L)=1$; a solution can be increased to have $p\ge2$.
\end{proof}

\subsection{Boundary points and heavy atoms}\label{app:boundary-heavy}
\begin{theorem}[Structure of a boundary point]\label{thm:bdry}
Let $q\in\mathcal B_K$.
\begin{enumerate}[label=\textup{(\alph*)}]
\item \textup{(Limit symbol.)} There are $\lambda_n\ge0$, $n\ne0$, and $\gamma\ge0$ with $\sum_n\lambda_n+\gamma=1$
such that $S_q=\sum_n\lambda_n\Lambda_n+\gamma\cos\Theta_q$ is continuous, $S_q\le0$ on $[0,\pi]$, every
extremal is carried by $Z_q=\{S_q=0\}$, $\int\Lambda_n\,dw=0$ for every extremal $w$ whenever
$\lambda_n>0$, and $\int\cos\Theta_q\,dw=0$ for every extremal if $\gamma>0$.
\item \textup{(Dip.)} Every extremal charges the closed dip $\bar\Delta\subset(2\pi/(K+1),2\pi/K)$;
hence $Z_q\cap\bar\Delta\ne\emptyset$.
\item \textup{(Periodic points.)} $q\in\mathcal P$ if and only if some extremal is carried by some
$\Omega_L$; then each such extremal has an anchor $2\pi j/k$ with $1/(K+1)<j/k<1/K$, $k\mid L$
and $k\ge2K+1$.
\item \textup{(Non-periodic points.)} If $q\notin\mathcal P$, then every extremal has an atom at an
irrational multiple of $2\pi$, or infinitely many atoms, or a non-atomic part; and
$\sigma_L(q)\to0$ as $L\to\infty$.
\item \textup{(Closures of plateaus.)} Let $\Pi$ be a plateau in $I_K$ with support $S_\Pi\subset\Omega_L$.
Then $\sigma_L=0$ on $\cl\Pi\cap\mathcal B_K$, and at every $q\in\cl\Pi\cap\mathcal B_K$ some extremal is
carried by $S_\Pi$. Hence the relative boundary of every plateau lies in $\mathcal P\setminus\mathcal R$.
\end{enumerate}
\end{theorem}

\begin{proof}
The limit symbol and its complementary slackness are
Theorem~\ref{thm:limitsymbol}; dip carriage is Lemma~\ref{lem:dip}.
A grid-carried extremal is equivalent to $\sigma_L=0$, and
Proposition~\ref{prop:farey} gives its anchor. Finitely many frequencies
at rational multiples of $2\pi$ lie on one grid; the absence of such a
support, after removing a possible zero atom, gives the stated measure alternatives. The sampling argument of
Theorem~\ref{thm:sampling}, applied to a measure having no such carrier,
then works for every sufficiently large period.

On a plateau $\sigma=\sigma_L=0$. Continuity preserves both equalities
at its closure, and compactness of probability vectors on its fixed
finite support gives an extremal there. A regular limit point has a
neighborhood with the same support inside its own open plateau; if it
is approached by the given plateau, the components coincide. Hence a
relative endpoint is periodic and nonregular.
\end{proof}

\begin{proof}[Proof of Theorem~\ref{thm:class}]
Strict decrease of $\sigma$ in $\kappa$ gives the three signs. The finite
alternative and exact threshold theorem give (i); sampling a negative
threshold gives (ii), and grid inclusion gives all multiples of its
period. At equality use Theorem~\ref{thm:bdry}. Every feasible grid
measure has the Farey anchor of Proposition~\ref{prop:farey}.
\end{proof}

\begin{proposition}[Heavy atoms move continuously]\label{prop:heavy}
Let $C\subset\mathcal B_K$ be connected and $m_0>1/2$, and suppose that at every $q\in C$ every
extremal $w$ with $w(\{0\})=0$ and $w((0,\pi])=1$ has an atom of mass at least $m_0$. Then:
\begin{enumerate}[label=\textup{(\alph*)}]
\item at each $q\in C$ all such extremals have their heavy atom at the same point $r(q)\in(0,\pi]$;
\item $r:C\to(0,\pi]$ is continuous;
\item if $q\in\mathcal P$, then $r(q)\in2\pi\Q$;
\item if $r$ is not constant on $C$, then $C\setminus\mathcal P$ is uncountable.
\end{enumerate}
\end{proposition}

\begin{proof}
Normalize by removing the zero atom and setting the remaining mass to
one. If two extremals had their atoms of mass at least $m_0>1/2$ at
different points, their convex interpolation would either have no such
atom or would have two, of total mass greater than one. Thus the
location is unique. The minimizer constraint bounds
$\cos\Theta_q(r(q))\ge-(1-m_0)/m_0>-1$, keeping it locally away from
$0$. Weak compactness and the closed program constraints pass a heavy
atom to every subsequential limit; removing a possible zero atom
preserves its normalized mass lower bound. Uniqueness identifies the
limit, proving continuity. At a periodic point the common location is
rational. A nonconstant continuous image of a connected set is an
interval, whose irrational points give uncountably many nonperiodic
preimages.
\end{proof}

\subsection{Finite decisions and localization}
\begin{proposition}[Polynomial certificates]\label{prop:polycert}
Let $q$ be a tuning, $N\subset\Z\setminus\{0\}$ finite, $\lambda_n\ge0$ $(n\in N)$ and $\gamma\ge0$. The symbol
$\sum_{n\in N}\lambda_n\Lambda_n+\gamma\cos\Theta_q$ is strict if and only if
\[
 \Xi_q(c)\Bigl[\sum_{n\in N}\lambda_n\bigl(1-T_{|n|}(c)\bigr)+\gamma\Bigr]<(1-c^2)\sum_{n\in N}\sgn(n)\,\lambda_n\,U_{|n|-1}(c)\qquad(-1\le c\le1).
\]
\end{proposition}

\begin{proof}
Put $\mathcal P=(1-c^2)\sum_n\operatorname{sgn}(n)\lambda_nU_{|n|-1}
-\Xi_q[\sum_n\lambda_n(1-T_{|n|})+\gamma]$. The Chebyshev identities give
$S=-(\sin\Theta_q/\sin\omega)\mathcal P(\cos\omega)$ in the interior.
At the endpoints $S(0)=-\gamma$, $S(\pi)=-\sum_n\lambda_n(1-(-1)^n)-\gamma$,
and $\Xi_q(\pm1)<0$. These identities give the claimed equivalence
including strictness at both ends; compactness makes the negative
margin uniform.
\end{proof}

\begin{theorem}[Exact relaxations near a regular point]\label{thm:tightclass}
Let $(q_*,s_*)$ be $L$-regular with data $S,N,w^*,\lambda^*$, and let $U$ and $s$ be as in Theorem~\ref{thm:modelock}.
\begin{enumerate}[label=\textup{(\alph*)}]
\item $L$ is the least common denominator of the points of $S/2\pi$.
\item For every $q\in U$ and every set of lags $N'\supset N$, $\sigma_{N'}(q)=\sigma_L(q)=\sigma(q)=s(q)$, and the
solutions of $\FP[N'](\Theta_{q,\sigma(q)})$ are those of $\FP(\Theta_{q,\sigma(q)})$.
\item For $q\in U$ the following are equivalent: $\rho(q)<1$; $\FP[N](\Theta_q)$ is infeasible; $\PP_L(\Theta_q)$ is
infeasible; $s(q)>0$.
\end{enumerate}
\end{theorem}

\begin{proof}
Let $L_0\mid L$ be the least support period. If $L=mL_0$ with $m>1$,
translation by $L_0$ preserves all tight classes and their pinning
vectors. Independence forces $m=2$ and a single tight class $L_0$;
then every certificate lag is a multiple of $L_0$ and $P^*=0$ on the
support, a contradiction. Thus $L=L_0$.
The continued symbol and the strict lower-shift symbols in the proof
of Theorem~\ref{thm:modelock} use only $N$. They give the same threshold
for every $N'\supset N$. Complementarity again gives exactly the same
extremals, and shift monotonicity together with
Theorem~\ref{thm:exactthr} gives (c).
\end{proof}

\begin{theorem}[Quantitative locking]\label{thm:qlock}
Let $(q,s_*)$ be $L$-regular with data $S,N,w^*,\lambda^*$, let $\Theta=\Theta_{q,s_*}$,
$S^*=\sum_n\lambda^*_n\Lambda^{q,s_*}_n$, $P^*=\sum_n\lambda^*_n(1-\cos n\omega)$, and $R^*=-S^*/(\sin\Theta\,P^*)$
on $(0,\pi)$.
\begin{enumerate}[label=\textup{(\alph*)}]
\item $R^*$ is continuous and positive on $(0,\pi)\setminus S$, vanishes on $S$, tends to $+\infty$ at
$0$ and $\pi$, and $R^*(\omega+h)=\frac{|\partial^2_\omega S^*(\omega)|}{2\sin\Theta(\omega)P^*(\omega)}h^2(1+o(1))$ at
$\omega\in S$.
\item If $w$ solves $\FP(\Theta_{q,s})$, then $w$ charges $(0,\pi)$ and
\[
 s-s_*\ \ge\ \frac{\int_{(0,\pi)}\sin\Theta_{q,s}\,P^*R^*\,dw}{\int_{(0,\pi)}\sin\Theta_{q,s}\,P^*\,dw}\ \ge\ \inf_{\supp w\cap(0,\pi)}R^* .
\]
\item For every $L'\ge3$ with $\Omega_{L'}\cap S=\emptyset$, with
$d(\Omega_{L'},S)=\min\{|\omega-\omega'|:\ \omega\in\Omega_{L'},\ \omega'\in S\}$,
\[
 \sigma_{L'}(q)-s_*\ \ge\ c_*\,d(\Omega_{L'},S)^2\bigl(1-o(1)\bigr)\quad\text{as }d(\Omega_{L'},S)\to0,
\]
where $c_*=\min_{\omega\in S}|\partial^2_\omega S^*(\omega)|/(2\sin\Theta(\omega)P^*(\omega))$; and
$d(\Omega_{L'},S)\ge2\pi/\operatorname{lcm}(L,L')$.
\end{enumerate}
\end{theorem}

\begin{proof}
If $P^*$ vanished in $(0,\pi]$, all lag phases there would be multiples
of $2\pi$, hence $S^*=0$, contradicting either strictness off $S$ or
positivity of $P^*$ on $S$. Taylor's formula at $S$, and the endpoint
expansions of the lag functions, prove (a). Integrate
\eqref{eq:shiftid} against a feasible finite measure. Its endpoint
contributions are nonpositive; it must charge $(0,\pi)$, since a
measure carried by the endpoints violates either this symbol or the
minimizer constraint. Dividing gives (b). Apply it to a periodic
threshold optimizer and use the quadratic minima in (a). Distinct
fractions of denominators $L,L'$ have distance at least
$1/\operatorname{lcm}(L,L')$, proving (c).
\end{proof}

\begin{theorem}[Localization]\label{thm:local}
Let $q$ be a tuning, $N\subset\Z\setminus\{0\}$ finite, and $\sigma_N=\sigma_N(q)$.
\begin{enumerate}[label=\textup{(\alph*)}]
\item There are $\lambda_n\ge0$ $(n\in N)$ and $\gamma\ge0$ with $\sum_n\lambda_n+\gamma=1$ such that
$T=\sum_{n\in N}\lambda_n\Lambda^{q,\sigma_N}_n+\gamma\cos\Theta_{q,\sigma_N}\le0$ on $[0,\pi]$.
\item Let $s\ge\sigma_N$ and let $w\ge0$ satisfy the constraints of $\FP[N](\Theta_{q,s})$. Then, with
$P=\sum_n\lambda_n(1-\cos n\omega)$ and $T_s(\pi)=\sum_n\lambda_n((-1)^n-1)-\gamma\le0$,
\begin{multline}\label{eq:local}
 \int_{(0,\pi)}\frac{\sin\Theta_{q,s}}{\sin\Theta_{q,\sigma_N}}\,(-T)\,dw+\gamma\,w(\{0\})+|T_s(\pi)|\,w(\{\pi\})\\
 \le\ (s-\sigma_N)\int_{(0,\pi)}\sin\Theta_{q,s}\,(P+\gamma)\,dw .
\end{multline}
\item If, in addition, $w$ has an atom of mass $w_a$ at $\omega_a\in(0,\pi)$, then
\[
 -T(\omega_a)\ \le\ \frac{2(s-\sigma_N)\,w([0,\pi])}{w_a}\cdot\frac{\sin\Theta_{q,\sigma_N}(\omega_a)}{\sin\Theta_{q,s}(\omega_a)} .
\]
\end{enumerate}
\end{theorem}

\begin{proof}
Normalized finite-dimensional separating vectors at shifts increasing
to $\sigma_N$ have a convergent subsequence in the simplex. Their
limit gives $T\le0$. Integrating \eqref{eq:shiftid} gives
\eqref{eq:local}, including the two displayed endpoint masses.
Keeping one atomic contribution on its left and using
$\sin\Theta(P+\gamma)\le2$ on the right proves (c).
\end{proof}

\begin{corollary}[Quadratic localization]\label{cor:local}
In the setting of Theorem~\ref{thm:local}, assume that $T(\pi)<0$, that $T$ has finitely many
zeros $z_1,\dots,z_r$ in $(0,\pi)$, all with $T''(z_i)<0$, and that $T''(0)<0$ if $\gamma=0$. Then there
is $c>0$ such that every $w$ as in Theorem~\ref{thm:local}(b) satisfies
\[
 \int_{[0,\pi]}\operatorname{dist}\bigl(\omega,\{0,z_1,\dots,z_r\}\bigr)^2\,dw(\omega)\ \le\ \frac{2(s-\sigma_N)}{c}\,w([0,\pi]).
\]
\end{corollary}

\begin{proof}
Nondegenerate zeros and compactness give
$-T\ge c'\operatorname{dist}(\omega,\{0,z_1,\ldots,z_r\})^2$.
On the shift interval $[\sigma_N,\sigma_N+1]$, the sine ratio in
\eqref{eq:local} has a positive lower bound, including its endpoint
limits of one. The negative value at $\pi$ controls its endpoint mass.
This proves the estimate there; decreasing $c$ so that
$c\le2/\pi^2$ covers all larger shifts by the trivial distance bound.
\end{proof}

At a fixed shift, separate the frequencies according to which of the
two lag constraints they satisfy: write
$A_+=\{\Lambda_1>0>\Lambda_{-K}\}$,
$A_-=\{\Lambda_{-K}>0>\Lambda_1\}$, both in $(0,\pi]$, and
$R_+=\Lambda_1/(-\Lambda_{-K})$, $R_-=\Lambda_{-K}/(-\Lambda_1)$.

\begin{proposition}[Decoupling]\label{prop:decouple}
Assume that no point of $(0,\pi]$ has $\Lambda_1\ge0$ and $\Lambda_{-K}\ge0$ at once, and that
$\hat R_-:=\sup_{A_-}R_-$ is attained at some $u\in A_-$ with $\cos\Theta(u)>0$. Then $\hat R_+:=\sup_{A_+}R_+$
is attained if $A_+\ne\emptyset$, and $\hat R_+:=0$ otherwise, and:
\begin{enumerate}[label=\textup{(\alph*)}]
\item $\FP[\{1,-K\}]$ is feasible if and only if $\hat R_+\hat R_-\ge1$; for $\omega_1\in A_+$, $\omega_2\in A_-$ the
measure $\delta_{\omega_1}+r\delta_{\omega_2}$ satisfies the two lag constraints exactly for
$-\Lambda_{-K}(\omega_1)/\Lambda_{-K}(\omega_2)\le r\le\Lambda_1(\omega_1)/(-\Lambda_1(\omega_2))$, an interval that is nonempty
if and only if $R_+(\omega_1)R_-(\omega_2)\ge1$;
\item if $\FP[\{1,-K\}]$ is feasible and $T=\lambda_1\Lambda_1+\lambda_{-K}\Lambda_{-K}+\gamma\cos\Theta\le0$ with
$\lambda_1,\lambda_{-K},\gamma\ge0$ not all zero, then $\gamma=0$, $\lambda_1,\lambda_{-K}>0$, $\hat R_+=\lambda_{-K}/\lambda_1$,
$\hat R_-=\lambda_1/\lambda_{-K}$, and the zeros of $T$ in $(0,\pi]$ are exactly the maximizers of $R_+$ on
$A_+$ and of $R_-$ on $A_-$.
\end{enumerate}
\end{proposition}

\begin{proof}
On the complement of $A_+\cup A_-$ both lags are nonpositive.
Every feasible measure charges both sign sets. Integrating the two
ratio bounds gives $\hat R_+\hat R_-\ge1$. Conversely the stated
interval for $r$ enforces both lag inequalities; choosing a maximizer
of $R_-$ with positive minimizer value also enforces that constraint,
since $\cos\Theta>0$ on $A_+$. The maximum of $R_+$ exists because it
extends by zero to the boundary of its compact interior sign set.
For (b), equality against this two-atom feasible measure forces
$\gamma=0$, both lag weights positive, and the reciprocal ratio bounds.
Their equality sets are exactly the zero set away from $0$.
\end{proof}

\begin{theorem}[Exact two-lag points]\label{thm:exact}
Let $q$ be a tuning, $K\ge2$, $s_2\in\R$, and write $\Lambda_n=\Lambda^{q,s_2}_n$, $\Theta=\Theta_{q,s_2}$.
Assume:
\begin{enumerate}[label=\textup{(E\arabic*)}]
\item there are $\lambda_1,\lambda_{-K}>0$, $\lambda_1+\lambda_{-K}=1$, such that $T=\lambda_1\Lambda_1+\lambda_{-K}\Lambda_{-K}\le0$
on $[0,\pi]$, $T''(0)<0$, and $T$ vanishes in $(0,\pi)$ exactly at two points $u_1,u_2$ with
$T''(u_1),T''(u_2)<0$, where $\Lambda_1(u_1)>0>\Lambda_{-K}(u_1)$ and $\Lambda_{-K}(u_2)>0>\Lambda_1(u_2)$;
\item $u_1,u_2\in\Omega_L$ for some $L\ge K+2$;
\item with $w_1,w_2>0$ defined up to scaling by $w_1\Lambda_1(u_1)+w_2\Lambda_1(u_2)=0$, one has
$w_1\Lambda_\varrho(u_1)+w_2\Lambda_\varrho(u_2)>0$ for every residue $\varrho\not\equiv0,1,-K\pmod L$, and
$w_1\cos\Theta(u_1)+w_2\cos\Theta(u_2)>0$;
\item $\det\bigl(\sin(\varrho u_i-\Theta(u_i))\bigr)_{i=1,2;\ \varrho=1,-K}\ne0$.
\end{enumerate}
Then $(q,s_2)$ is $L$-regular with $S=\{u_1,u_2\}$ and $N=\{1,1\pm L,-K,-K\pm L\}$.
Consequently $\sigma(q)=\sigma_L(q)=\sigma_{\{1,-K\}}(q)=s_2$, the threshold extremal is carried by
$\{u_1,u_2\}$ and unique up to scaling and an atom at $0$, and the conclusions of
Theorem~\ref{thm:modelock} hold near $q$.
\end{theorem}

\begin{proof}
The contact matrix has rank one: the nonzero dual vector annihilates
it, and the signs in (E1) give a positive right kernel. Condition (E3)
gives the remaining strict primal slacks. Replace each original lag
$\varrho\in\{1,-K\}$ by the three lags $\varrho,\varrho\pm L$, placing
weights $\lambda_\varrho-2\varepsilon,\varepsilon,\varepsilon$ on them.
Their added symbol and first derivative vanish on $\Omega_L$.
For sufficiently small $\varepsilon>0$, $C^2$ persistence preserves
the strict sign and curvatures of (E1); the lag $1$ keeps $P>0$.
Condition (E4) supplies pinning rank. Theorem~\ref{thm:modelock}
applies. The unperturbed two-lag symbol and \eqref{eq:shiftid}, with a
small minimizer term, exclude all lower shifts using only $1,-K$;
hence the relaxation threshold is equal as well.
\end{proof}

\begin{corollary}[Exact points on the boundary]\label{cor:exactdD}
In Theorem~\ref{thm:exact} assume moreover $s_2=0$ and $u_1/2\pi\in(1/(K+1),1/K)$. Then
$q\in\mathcal B_K$, $q$ lies on a plateau, and on this plateau $\kappa_\partial$ is the analytic solution of
\[
 \Lambda_1(u_1)\Lambda_{-K}(u_2)-\Lambda_1(u_2)\Lambda_{-K}(u_1)=0
\]
at the shift $0$, an equation in the values of $\Theta$ at the two fixed frequencies $u_1$
and $u_2$. If also $u_2/2\pi<1/(K+1)$, then $u_1/2\pi$ is the unique anchor at every point of this
plateau that lies in $I_K$.
\end{corollary}

\begin{proof}
The symbol is strictly negative at $2\pi/K$ and $2\pi/(K+1)$, so both
points lie outside the dip, while $u_1$ lies inside it. The dip is an
interval, giving $q\in I_K$ and $\sigma(q)=0$. The determinant equation
holds near $q$; analyticity and connectedness of the plateau extend it
throughout its projection. If $u_2$ lies below the window, only $u_1$
can be an anchor.
\end{proof}

For $K\ge2$ and $r>0$ put
\[
 W_K(c)=\frac{1-T_K(c)}{1-c},\qquad
 \Psi_r^{(K)}(c)=\frac{(1+c)(1-rU_{K-1}(c))}{1+rW_K(c)}.
\]
The polynomial $W_K$ is nonnegative on $[-1,1]$ and $W_K(1)=K^2$.

\begin{proposition}[Two-lag symbols as double tangencies]\label{prop:twolagtangency}
Let $K\ge2$, $0<u_2<u_1<\pi$, $c_i=\cos u_i$, let $q$ be a tuning, and let $T=\lambda_1\Lambda_1+\lambda_{-K}\Lambda_{-K}$ be a symbol of
$\Theta_q$ with $\lambda_1,\lambda_{-K}>0$ and $r=\lambda_{-K}/\lambda_1$.
\begin{enumerate}[label=\textup{(\alph*)}]
\item On $(0,\pi)$, $T=\lambda_1\,\dfrac{\sin\Theta_q\,(1-c)(1+rW_K(c))}{\sin\omega}\,\bigl(\Xi_q(c)-\Psi^{(K)}_r(c)\bigr)$ with $c=\cos\omega$, and $T(\pi)<0$.
\item $T\le0$ on $[0,\pi]$ with zero set $\{0,u_1,u_2\}$, $T''(u_1),T''(u_2)<0$ and $T''(0)<0$ hold if and only if
$\Xi_q<\Psi^{(K)}_r$ on $[-1,1]\setminus\{c_1,c_2\}$, $\Xi_q(c_i)=\Psi^{(K)}_r(c_i)$ and $(\Psi^{(K)}_r-\Xi_q)''(c_i)>0$ $(i=1,2)$, where the value at
$c=1$ is $\Psi^{(K)}_r(1)=2(1-rK)/(1+rK^2)$.
\item In that case $\Xi_q$ is the quadratic
with values $\Xi_q(c_i)=\Psi^{(K)}_r(c_i)$ and slopes $\Xi_q'(c_i)=\Psi^{(K)\prime}_r(c_i)$, and $r$ is a root of
\[
 G_K(r)=\Psi^{(K)\prime}_r(c_1)+\Psi^{(K)\prime}_r(c_2)-2\,\frac{\Psi^{(K)}_r(c_1)-\Psi^{(K)}_r(c_2)}{c_1-c_2} .
\]
\item Let $\Xi_q(c_i)=\Psi^{(K)}_r(c_i)$ and $p_i=\sin\frac{Ku_i}2\sin\frac{(K+1)u_i}2$. Then $\Lambda_1(u_i)>0>\Lambda_{-K}(u_i)$ if and only if $p_i<0$,
and $\Lambda_{-K}(u_i)>0>\Lambda_1(u_i)$ if and only if $p_i>0$. In particular
both sign conditions of \textup{(E1)} hold when $1/(K+1)<u_1/2\pi<1/K$ and $u_2/2\pi<1/(K+1)$.
\end{enumerate}
\end{proposition}

\begin{proof}
Factor the two-lag symbol using
$1-T_K=(1-c)W_K$ and $\sin K\omega=\sin\omega U_{K-1}(c)$.
The factor is positive in the interior. Taylor's formula at the
contacts gives the curvature equivalence; at $0$, its endpoint
condition is $\Xi_q(1)<2(1-rK)/(1+rK^2)$. A quadratic with the prescribed values and first derivatives exists
exactly when the average of its two prescribed slopes equals the
secant slope, namely $G_K(r)=0$.
Finally
\[
 U_{K-1}(\cos u)+W_K(\cos u)
 =\frac{\sin(Ku/2)\sin((K+1)u/2)}{\sin^2(u/2)\cos(u/2)},
\]
which gives the sign assertions and both stated frequency windows.
\end{proof}

\subsection{Verification of the mixed-spectrum examples}\label{app:boundary-examples}
\begin{proof}[Proof of Theorem~\ref{thm:plateaus}]
We give finite polynomial certificates for the two rows. Set
\[
 G_n(c)=\Xi_q(c)(1-T_{|n|}(c))
       -\operatorname{sgn}(n)(1-c^2)U_{|n|-1}(c).
\]
For the row of period $18$ take the lag pairs $(1,19),(-4,14)$;
for the row of period $21$ take $(1,-20),(-12,9),(-4,-25)$.
In each pair denote the first and second entries by $r_i,t_i$, and
put $c_j=\cos(2\pi j/L)$ for the support numerators in
Table~\ref{tab:plateaus}. Form $M(\kappa)=(G_{r_i}(c_j))_{i,j}$.
The determinant has opposite strict signs at the two endpoints of
$[\hat\kappa-10^{-7},\hat\kappa+10^{-7}]$, where
\[
 \hat\kappa_{18}=100.55082118856557,\qquad
 \hat\kappa_{21}=106.80551194382778.
\]
Both intervals are strictly contained in the displayed open intervals.
At any determinant root, the $(0,0)$ cofactor is nonzero. Normalize
its adjugate column and row to obtain right and left kernels
$\rho,\mu$ with sums one. To construct the dual weights, solve
\[
 \sum_i\bigl[G_{t_i}'(c_j)-G_{r_i}'(c_j)\bigr]v_i
       =-\sum_i\mu_iG_{r_i}'(c_j),\qquad
 \lambda_{t_i}=v_i,\quad\lambda_{r_i}=\mu_i-v_i.
\]
Then $\sum_n\lambda_nG_n$ and its first derivative vanish at every
contact. The determinants of this system lie in $(-580,-579)$ and
$(-46714,-46713)$, respectively. Its columns are
nonzero scalar multiples of the pinning columns in (R3), so (R3)
follows as well.

Here are rational lower bounds valid throughout the respective
brackets; $\rho,\mu$ are the polynomial kernels just defined.
\[
\begin{array}{c|ccccc}
 L&\min\rho_j&\min\lambda_n&
 \min_{r\not\equiv0,r_i}\sum_j\rho_jG_r(c_j)&
 \sum_j\rho_j\Xi_q(c_j)&|\operatorname{cof}_{00}M|\\ \hline
 18& 6/25&7/10^6&3/50&6/5&17/50\\
 21& 9/100&5/10^5&1/25&11/10&1/2
\end{array}
\]
The left kernels are positive too. All these inequalities are strict.
Since
$\Lambda_n=(\sin\Theta_q/\sin\omega)G_n$, the masses
$w_j=\rho_j\sin\omega_j/\sin\Theta_q(\omega_j)$ give the required
positive primal measure; its minimizer slack is
$\sum_j\rho_j\Xi_q(c_j)$. Thus the displayed bounds verify all inactive
primal constraints.

For the global dual sign, factor
\[
 \sum_n\lambda_nG_n(c)=(1-c)H(c),\qquad
 H(c)=\sum_n\lambda_n\left[
 \Xi_q(c)\frac{1-T_{|n|}(c)}{1-c}
       -\operatorname{sgn}(n)(1+c)U_{|n|-1}(c)\right].
\]
Enclose each $c_j$ by the rational interval
$[c_j^-,c_j^+]$ obtained by rounding outward to $14$ decimal places
and extending one further unit in the last place. On each contact
interval $[c_j^--1/200,c_j^++1/200]$, the upper Bernstein coefficient
bound for $H''$ is less than $-7/10$ for $L=18$ and $-7/25$ for
$L=21$. On each intervening closed interval, the upper Bernstein
coefficient bound for $H$ is less than $-9\cdot10^{-6}$ and
$-3\cdot10^{-6}$, respectively. The contact values and slopes are
exactly zero, so Taylor's theorem proves the required strict sign
and negative contact curvatures. The bound at $c=1$ is negative and,
as $H(1)=\Xi_q(1)\sum_n\lambda_nn(n+2\tau_q)$,
it also gives the endpoint moment condition.

For clarity, the coefficient test used here is the following standard
one~\cite[Section~5.1]{Farouki2012}. If a polynomial on $[a,b]$ is
written as $P(a+(b-a)x)=\sum_{j=0}^d p_jx^j$, its Bernstein
coefficients are
$b_i=\sum_{j\le i}p_j\binom{i}{j}/\binom{d}{j}$.
Their minimum and maximum bound $P$ on that interval. Applying this
formula to the displayed polynomials, with rational interval
operations, gives the stated strict bounds. The trigonometric
constants are enclosed by argument reduction and alternating Taylor
bounds. The same substitution in $\Xi_q$ gives strict gap endpoint
inequalities and a point with $\Theta_q<\omega/2$, hence membership
in $I_4$.

Consequently (R1)--(R3) hold at every bracketed determinant root.
Theorem~\ref{thm:modelock} gives an analytic plateau, and strict
decrease of $\sigma$ in $\kappa$ makes the boundary value unique.
The fixed supports give the unique anchors. Their least periods are
$18$ and $21$ by Theorem~\ref{thm:tightclass};
Corollary~\ref{cor:periods} excludes periods $4p+1$ in the first row,
and strict positivity at all three support points proves the final
assertion for the second row.
\end{proof}

\begin{theorem}[An exact two-lag point of $3/13$]\label{thm:exact313}
There is a tuning $q^\star=(\kappa^\star,\alpha L^\star,\beta^\star)\in I_4$ with
\[
 |\kappa^\star-83.5343|<5\cdot10^{-5},\quad
 |\alpha L^\star-2.218909|<5\cdot10^{-7},\quad
 |\beta^\star-0.622768|<5\cdot10^{-7},
\]
for which, at shift $0$, $u_1=6\pi/13$ and $u_2=4\pi/13$ satisfy
\[
 T(u_2)=T'(u_1)=T'(u_2)=0,\qquad
 T=-\Lambda_{-4}(u_1)\Lambda_1+\Lambda_1(u_1)\Lambda_{-4}.
\]
The Jacobian of these three equations with respect to
$(\kappa,\alpha L,\beta)$ is nonsingular at $q^\star$, which is their
unique solution in an open neighborhood. Conditions \textup{(E1)--(E4)}
of Theorem~\ref{thm:exact} hold. Consequently
\[
 \sigma(q^\star)=\sigma_{13}(q^\star)=\sigma_{\{1,-4\}}(q^\star)=0,
\]
and $q^\star$ lies on an analytic plateau with unique anchor $3/13$ and
unique threshold spectrum $\{2,3\}\cdot2\pi/13$, up to scaling and an
admissible atom at $0$.
\end{theorem}
\begin{proof}[Proof of Theorem~\ref{thm:exact313}]
For $K=4$ write $u=6\pi/13$, $v=4\pi/13$, $c_1=\cos u$, $c_2=\cos v$,
and let $Q_r(c)=d+bc+ac^2$ be the quadratic matching the contact values
and slopes in Proposition~\ref{prop:twolagtangency}. The scalar compatibility equation
has a root
\[
 0.136418311438750<r<0.136418311438754.
\]
Rational interval bounds give $a<0$, $Q_r(\pm1)<0$, and
$b+1+2a(c_1+c_2)<0$, so $Q_r$ is the quadratic of a tuning.
The global inequality requires only the factorization
\[
 (1+rW_4(c))Q_r(c)-(1+c)(1-rU_3(c))
 =(c-c_1)^2(c-c_2)^2\,8r
    [ac+b+a+1+2a(c_1+c_2)].
\]
For $c\ge-1$ the last factor is negative. Thus the only interior
zeros are the two contacts, and both curvatures are strictly negative.
The endpoint condition, ten inactive residue slacks, minimizer slack,
and pinning determinant are strict interval inequalities on the
root enclosure. This proves \textup{(E1)--(E4)} and membership in $I_4$.

For isolation, let $X_C(c)=C_0+C_1c+C_2c^2$ and put
$A_C=1+c_1-X_C(c_1)$,
$B_C=X_C(c_1)W_4(c_1)+(1+c_1)U_3(c_1)$.
Eliminating $r=A_C/B_C$ from the first contact gives the polynomial
profile
\[
 \mathcal H_C(c)=B_C[X_C(c)-1-c]
       +A_C[X_C(c)W_4(c)+(1+c)U_3(c)].
\]
The three-dimensional system is
$(\mathcal H_C(c_2),\mathcal H_C'(c_1),\mathcal H_C'(c_2))=0$.
A preconditioned interval-Jacobian inclusion gives a unique root and
a nonsingular derivative by the Krawczyk criterion
\cite[Theorem~13.3]{Rump2010}. Replacing the polynomial profile by the trigonometric symbol multiplies
the contact equations by a triangular matrix with nonzero diagonal;
the change from tuning parameters to the coefficients of $\Xi_q$ has
nonsingular Jacobian for $\beta>0$.
Consequently the original three equations have nonsingular Jacobian
and the same open local uniqueness. Inverting the chart gives the
three displayed rational parameter enclosures. Theorem~\ref{thm:exact}
and Corollary~\ref{cor:exactdD} finish the proof.
\end{proof}

\section{Single-harmonic certificates and fan asymptotics}\label{app:plateau}
Use the notation of Subsections~\ref{sec:farey} and~\ref{sec:plateau}.

\subsection{The edge criterion}
\begin{lemma}[Edge identities]\label{lem:edgeidjk}
Let $\Theta$ be a curve and $\omega=\omega_0+\delta\in[0,\pi]$. For $n\in b+k\Z$ and $m\in\Z\setminus\{0\}$,
\[
 \Lambda_n(\omega)=2\sin\bigl(c'+\tfrac{n\delta}2\bigr)\sin\bigl(c'+\tfrac{n\delta}2-\Theta(\omega)\bigr),\qquad
 \Lambda_{mk}(\omega)=2\sin\tfrac{mk\delta}2\,\sin\bigl(\tfrac{mk\delta}2-\Theta(\omega)\bigr),
\]
and $n\omega\equiv2c'+n\delta\pmod{2\pi}$ for $n\in b+k\Z$. If $\Theta(\omega_0)=c'$ and $\Theta$ is twice differentiable at
$\omega_0$ with $\bar n=2\Theta'(\omega_0)$, then for $n\in b+k\Z$
\[
 \Lambda_n'(\omega_0)=\sin c'\,(n-\bar n),\qquad \Lambda_n''(\omega_0)=\cos c'\,n(n-\bar n)-2\sin c'\,\Theta''(\omega_0),
\]
and for $n=mk\ne0$
\[
 \Lambda_n'(\omega_0)=-\sin c'\,n,\qquad \Lambda_n''(\omega_0)=\cos c'\,n(n-\bar n).
\]
\end{lemma}

\begin{proof}
Reduce $n\omega_0$ modulo $2\pi$ in
$\Lambda_n=2\sin(n\omega/2)\sin(n\omega/2-\Theta)$, then differentiate
twice. Only the stated two derivatives of $\Theta$ are required.
\end{proof}

\begin{theorem}[Edge criterion at a Farey fraction]\label{thm:edgejk}
Let $q$ be a tuning with $\bar n(q)\notin b+k\Z$.
\begin{enumerate}[label=\textup{(\alph*)}]
\item $F_q(0)=F_q'(0)=F_q(\omega_0)=F_q'(\omega_0)=0$,
$F_q''(0)=\tfrac12\bigl[(1-\vartheta)n_1(n_1+2\tau_q)+\vartheta n_2(n_2+2\tau_q)\bigr]$, and
\[
 F_q''(\omega_0)=\sin c'\,\Theta_*''(\omega_0)-\tfrac12\cos c'\,k^2\vartheta(1-\vartheta).
\]
\item If $F_q\ge0$ on $[0,\pi]$, then $\sigma(q)=\sigma_k(q)=s_*(q)$. If $F_q>0$ on $(0,\pi]\setminus\{\omega_0\}$, every
solution of $\FP(\Theta_*)$ is $b'\delta_{\omega_0}+a'\delta_0$ with $b'>0$ and $0\le a'\le b'\cos c'$. If moreover
$F_q''(0)>0$ and $F_q''(\omega_0)>0$, then $(q,s_*(q))$ is $k$-regular with $S=\{\omega_0\}$ and $N=\{n_1,n_2\}$,
and $\sigma=\sigma_k=s_*$ on a neighborhood of $q$.
\item If $F_q''(\omega_0)\le0$, then every symbol $T=\sum_n\lambda_n\Lambda_n\ne0$ of $\Theta_*$ with $\lambda_n\ge0$, lags in
$\mathcal N_{j/k}$ and $T'(\omega_0)=0$ has $T''(\omega_0)\ge0$. Consequently $(q,s_*(q))$ is not
$L$-regular with support $\{\omega_0\}$, for any $L$.
\end{enumerate}
\end{theorem}
The proof is given together with Lemma~\ref{lem:certjk} below.

\begin{lemma}[Single-harmonic certificates]\label{lem:certjk}
Let $q$ be a tuning, let a finite set $N\subset\mathcal N_{j/k}$ contain at least two lags of $b+k\Z$, and let $\lambda_n>0$, $n\in N$,
$\sum_n\lambda_n=1$. Let $T=\sum_{n\in N}\lambda_n\Lambda_n$ be formed with the lag functions of $\Theta_*$, and assume
$T<0$ on $(0,\pi]\setminus\{\omega_0\}$, $T''(\omega_0)<0$ and $\sum_n\lambda_nn(n+2\tau_q)>0$. Then $(q,s_*(q))$ is $k$-regular
with $S=\{\omega_0\}$ and the lag set $N$. There is a neighborhood $U$ of $q$ on which $\sigma=\sigma_k=s_*$ and
every solution of $\FP(\Theta_{q',\sigma(q')})$, $q'\in U$, is $b'\delta_{\omega_0}+a'\delta_0$ with $b'>0$ and $0\le a'\le b'\cos c'$.
\end{lemma}

\begin{proof}[Proofs of Lemma~\ref{lem:certjk} and Theorem~\ref{thm:edgejk}]
At the edge the only tight nonzero residue is $b+k\Z$ (position $1$);
its one-by-one primal matrix is zero, and its pinning entry is $\sin c'>0$.
The assumed sign and curvatures give all remaining regularity
conditions. Theorem~\ref{thm:modelock} then continues the zero
$\Lambda_b(\omega_0)=0$, whose unique shift is $s_*$. This proves the
lemma, including its precise classification of all finite extremals.

The edge identities and
$(1-\vartheta)n_1+\vartheta n_2=\bar n$ give part (a), with variance
$k^2\vartheta(1-\vartheta)$. For (b),
$P=(1-\vartheta)(1-\cos n_1\omega)+\vartheta(1-\cos n_2\omega)>0$
on $(0,\pi]$, since $\gcd(n_1,n_2)=1$.
Equation~\eqref{eq:shiftid} makes the symbol negative there at every
lower shift; a measure carried only by $0$ violates the minimizer
constraint. At the threshold, complementarity restricts the support to
$\{0,\omega_0\}$ and gives the asserted mass bound. The additional curvature conditions
allow application of the lemma.

For (c), pair positive and negative slope contributions in any finite
nonnegative combination with zero slope. For a position-$1$ lag the
ratio of second derivative to first derivative is
\[
 \rho_n=n\cot c'-\frac{2\Theta_*''(\omega_0)}{n-\bar n};
\]
for a nonzero position-$0$ lag it is $(\bar n-n)\cot c'$.
Here $\Theta_*''(\omega_0)>0$ by Theorem~\ref{thm:convex}.
Thus the first ratio is strictly increasing on either side of
$\bar n$: among positive-slope position-$1$ lags its minimum is
$\rho_{n_2}$, and among negative-slope ones its maximum is
$\rho_{n_1}$. Position-$0$ lags with positive slope have $n\le-k$
and ratio at least $(\bar n+k)\cot c'>\rho_{n_2}$; those with
negative slope have $n\ge k$ and ratio at most
$(\bar n-k)\cot c'<\rho_{n_1}$. Therefore the difference between
the least positive-slope ratio and the greatest negative-slope
ratio is
\[
 \rho_{n_2}-\rho_{n_1}
 =k\cot c'-\frac{2\Theta_*''(\omega_0)}{k\vartheta(1-\vartheta)},
\]
which is nonnegative exactly when $F_q''(\omega_0)\le0$.
A zero-slope nonnegative combination is a sum of two-lag combinations
with opposite slopes. Each has second derivative equal to the
product of its positive slope magnitudes and this nonnegative ratio
difference. Its curvature is therefore nonnegative. A regular singleton certificate
would have all its lags tight at $\omega_0$, hence in
$\mathcal N_{j/k}$, and would contradict this conclusion.
\end{proof}

\begin{corollary}[Single-harmonic plateaus]\label{cor:onejk}
Let $\Pi$ be a plateau whose threshold extremal is carried by $\{\omega_0\}$. Then $\Theta_q(\omega_0)=c'$ at
every $q\in\Pi$, and $\kappa_\partial=\kappa_{j/k}$ on the projection of $\Pi$ to the $(\alpha L,\beta)$ plane.
If $j/k\in(1/(K+1),1/K)$, then $j/k$ is the unique anchor at every point of $\Pi\cap I_K$.
\end{corollary}

\begin{proof}
A singleton regular primal has exactly one tight nonzero class.
Its position cannot be $0$, since that would give $P^*(\omega_0)=0$.
The one-atom feasibility condition and its tight position force phase
index $1$. Hence $\Theta_q(\omega_0)=c'$, and the tongue formula gives
$\kappa_\partial=\kappa_{j/k}$. The dip condition gives the anchor.
\end{proof}

\subsection{Asymptotics of two-frequency periodic measures}\label{app:plateau-fans}
\begin{lemma}\label{lem:Ppoly}
Let $k\ge3$ and $c'=\pi/k$. The trigonometric polynomial
\[
 \Upsilon_k(\theta)=k\cos c'-(k-1)\cos(\theta-c')-\cos\bigl((k-1)\theta+c'\bigr)
\]
is nonnegative on $\R$, vanishes exactly at $\theta\equiv0$ and $\theta\equiv2c'\pmod{2\pi}$, and
$\Upsilon_k''(0)=\Upsilon_k''(2c')=k(k-1)\cos c'>0$.
\end{lemma}

\begin{proof}
Put $m=k-1$ and $\varphi=\theta-c'$. Maximizing
$m\cos\varphi-\cos m\varphi$ gives the stationary factors
$2m\cos((m+1)\varphi/2)\sin((m-1)\varphi/2)$.
The first yields maximum $(m+1)\cos c'$ exactly at
$\varphi\equiv\pm c'$; the second yields at most $m-1$, which is
smaller because $(m+1)\cos(\pi/(m+1))>m-1$. Twice differentiating gives
the stated curvatures.
\end{proof}

\begin{proof}[Proof of Theorem~\ref{thm:fanjk}]
Write $s_c=\sin c'$, $c_c=\cos c'$,
$\mathsf u=k(k-1)c_c/4$, $\mathsf v=s_c/2$.
At $s=s_*-e\delta^2$ set $\Theta_a=c'+A$, $\Theta_c=c'+C$.
Joint analyticity gives
\begin{equation}\label{eq:ACjk}
 \begin{aligned}
 A&=\tfrac12\bar n\delta+
       (\tfrac12\Theta_*''+es_c^2)\delta^2+O(|\delta|^3),\\
 C&=-\tfrac12(k-1)\bar n\delta+
       (\tfrac12(k-1)^2\Theta_*''+es_c^2)\delta^2+O(|\delta|^3).
 \end{aligned}
\end{equation}
For masses $w_a=k-1+z\delta$, $w_c=1$, every integer lag constraint
is $\Phi(\theta)\ge0$ on $\theta\in(2\pi/L)\Z$, where
\[
 \Phi(\theta)=w_a[\cos\Theta_a-\cos(\theta-\Theta_a)]
             +\cos\Theta_c-\cos((k-1)\theta+\Theta_c).
\]
Indeed $\gcd(J,L)=1$ and
$(k-1)\omega_a+\omega_c=2\pi j$. The two exact local identities are
\begin{align}
 \tfrac12\Phi(u)&=w_a\sin\tfrac u2\sin(\tfrac u2-\Theta_a)
  +\sin\tfrac{(k-1)u}2\sin(\tfrac{(k-1)u}2+\Theta_c),\label{eq:Phi0jk}\\
 \tfrac12\Phi(2c'+v)&=w_a\sin(c'+\tfrac v2)\sin(\tfrac v2-A)
  -\sin(c'-\tfrac{(k-1)v}2)\sin(\tfrac{(k-1)v}2+C).\label{eq:Phi2cjk}
\end{align}
Near $0$ the grid is $u\in k|\delta|\Z$, and near $2c'$ it is
$v=y\delta$, $y\in b+k\Z$. Taylor expansion gives
$\Phi(u)/2=f_1u+f_2u^2+O(|u|^3)$ with
$f_1=-\delta(\mathsf u\bar n+\mathsf vz)+O(\delta^2)$,
$f_2=\mathsf u+O(|\delta|)$, and
\begin{equation}\label{eq:Qyjk}
 \frac{\Phi(2c'+y\delta)}{2\delta^2}
 =Q(y)+O\bigl(|\delta|(1+|y|^3)\bigr),\qquad
 Q(y)=\mathsf u(y^2-\bar n y)+\mathsf vz(y-\bar n)-\mathsf vE,
 \quad E=k(k-1)\Theta_*''+2ks_c^2e.
\end{equation}
The two nearest lattice values balance at
$\mathsf vz_*=-\mathsf u[\bar n+k(1-2\vartheta)]$. Convexity gives
\begin{equation}\label{eq:latminjk}
 \min_{y\in b+k\Z}Q(y)\big|_{z=z_*}
 =\max_z\min\{Q(n_1),Q(n_2)\}
 =\mathsf uk^2\vartheta(1-\vartheta)-\mathsf vE
 =k\sin^3c'(e_*-e).
\end{equation}
For $e<e_*$ this is positive. Away from fixed neighborhoods of
$0,2c'$, $\Phi$ converges uniformly to $\Upsilon_k>0$.
For the grid points $u\in k|\delta|\Z$ near $0$, put
$\nu=1-|1-2\vartheta|>0$. Since
$|\mathsf u\bar n+\mathsf vz_*|=(1-\nu)\mathsf uk$ and
$|u|\ge k|\delta|$ for $u\ne0$, the Taylor expansion gives
\[
 \tfrac12\Phi(u)\ge\nu\mathsf uu^2
                  -O(|\delta|u^2)-O(|u|^3)>0
\]
after one fixed neighborhood and then $|\delta|$ are made sufficiently
small. The point $u=0$ is an exact zero. Near $2c'$, first choose
$Y_0$ so that $Q(y)\ge\mathsf uy^2/4$ for $|y|>Y_0$.
For $|y|\le Y_0$, the fixed positive minimum in
\eqref{eq:latminjk} absorbs the remainder. For $|y|>Y_0$ and
$v=y\delta$ in a sufficiently small fixed neighborhood, the
unscaled remainder is
$O(|v|^3+|\delta|^3)=O(|v|^3)$ and is smaller than
$\mathsf uv^2/8$. Thus every grid point is controlled, including
lags increasing with $L$. The minimizer slack tends to
$k\cos c'>0$.

Conversely a nonzero feasible measure on the pair has both masses
positive for large $L$: each singleton has reduced denominator at
least $L/(k-1)$, incompatible with its limiting angle $c'>0$.
Normalize $w_c=1$. Let $u=2\pi/L$. The inequalities $\Phi(u),\Phi(-u)\ge0$
imply
\[
 |w_a\sin u\sin\Theta_a-\sin((k-1)u)\sin\Theta_c|
 \le2w_a\sin^2(u/2)\cos\Theta_a
       +2\sin^2((k-1)u/2)\cos\Theta_c.
\]
Since both angles tend to $c'\in(0,\pi/2)$, division by $u$
first bounds $w_a$ uniformly and then gives
$|w_a\sin\Theta_a-(k-1)\sin\Theta_c|=O(u)$.
Consequently $w_a=k-1+z\delta$ with $z$ bounded independently
of $L$. For $e>e_*$, \eqref{eq:latminjk} then forces one of
the two fixed lattice constraints $n_1,n_2$ to be negative for
all sufficiently large $L$. Shift
monotonicity now yields the attained pair threshold and its
$\delta^2$ expansion. If $F_q''<0$, choose $0<e<e_*$ to obtain
$\sigma<s_*$. Under the strict regular conditions, the potential
$2F_q/(\sin\Theta_*P)$ has coefficient
$F_q''/(2\sin^3c')$ at $\omega_0$; every missing grid point is at
distance at least $2\pi/(kL)$, giving the lower periodic bound.
The fan pair itself gives the upper bound.

On any compact tuning set where $n_1<\bar n<n_1+k$ for fixed $n_1$,
$\vartheta$ stays away from $0,1$, all Taylor bounds above are uniform,
and the margin for $e=e_*\mp\eta_0$ is the fixed number
$k\sin^3c'\eta_0$. The same neighborhood sizes, mass bounds and
$L_0$ therefore work for both orientations and all tunings in the
compact set. In particular the $o(\delta^2)$ in the pair threshold
is uniform on such compact sets.
\end{proof}

\subsection{The exact plateau domain}\label{app:plateau-domain}
Put $C_0=h_0-1-c_0>0$.
\begin{lemma}[The edge in the gap]\label{lem:edgedomain}
The tuning $q(a,D)$ is defined and lies in $I_K$ if and only if $a>a_E(D)$, where $a_E(D)$ is the largest of
\[
 \begin{gathered}
 \frac{h_0+(1+D)(1-c_0)}{(1-c_0)^2},\quad \frac{h_0-(1+D)(1+c_0)}{(1+c_0)^2},\\
 \frac{C_0+D(c_K-c_0)}{(c_K-c_0)^2},\quad \frac{C_0+D(c_{K+1}-c_0)}{(c_{K+1}-c_0)^2}.
 \end{gathered}
\]
The function $a_E$ is convex, piecewise affine and positive. The second term never exceeds the larger of the last two.
\end{lemma}

\begin{proof}
The first two affine inequalities are $\Xi^{a,D}(\pm1)<0$.
The last two say that the convex polynomial
$1+c-\Xi^{a,D}(c)=a(c-c_0)^2-D(c-c_0)-C_0$ is positive at both gap
endpoints, hence has its roots inside them. Their maximum is positive,
convex and piecewise affine. The negative-endpoint condition follows
already from the lower gap root lying above $-1$, proving redundancy
of the second term.
\end{proof}

\begin{definition}[Admissible symbols]\label{def:admissible}
Let $D\in\R$. An admissible symbol for $D$ is a finitely supported $\lambda=(\lambda_n)_{n\in\mathcal N_{j/k}}\ge0$ with
$\sum_n\lambda_n=1$, with positive weight on at least two lags of $b+k\Z$, whose support has greatest common divisor $1$, and with
\[
 \sum_{n\in b+k\Z}\lambda_n\bigl(n-\bar n(D)\bigr)=\sum_{n\in k\Z}\lambda_nn .
\]
For such $\lambda$ and $c=\cos\omega$ put
\[
 \Psi_\lambda(c)=\frac{\sin\omega\sum_n\lambda_n\sin n\omega}{\sum_n\lambda_n(1-\cos n\omega)}\quad(0<\omega\le\pi),\qquad
 \Psi_\lambda(1)=\frac{2\sum_n\lambda_nn}{\sum_n\lambda_nn^2},
\]
$R_\lambda(c)=\bigl(h_0+(1+D)(c-c_0)-\Psi_\lambda(c)\bigr)/(c-c_0)^2$ for $c\ne c_0$, and $R_\lambda(c_0)=-\frac12\Psi_\lambda''(c_0)$. Finally
\[
 a_F(D)=\inf_\lambda\ \max_{c\in[-1,1]}R_\lambda(c).
\]
When $\bar n(D)\notin b+k\Z$, the weights $1-\vartheta$, $\vartheta$ on the lags $n_1$, $n_2$ of Definition~\ref{def:edgesymbol} form the
edge symbol $\lambda^{\rm e}$, which is admissible because $\gcd(n_1,n_1+k)=\gcd(b,k)=1$, and we put $a_2(D)=R_{\lambda^{\rm e}}(c_0)$.
\end{definition}

\begin{lemma}[The certificate criterion along the edge]\label{lem:ray}
Let $\lambda$ be admissible for $D$. Then $R_\lambda$ is continuous on $[-1,1]$, and $q(a,D)$ is defined for all large $a$. Let $q(a,D)$ be defined and let
$T_\lambda=\sum_n\lambda_n\Lambda_n$ be formed with the lag functions of $\Theta_{q(a,D)}$. Then $T_\lambda<0$ on $(0,\pi]\setminus\{\omega_0\}$,
$T_\lambda''(\omega_0)<0$ and $\sum_n\lambda_nn(n+2\tau_q)>0$ if and only if $a>\max_{[-1,1]}R_\lambda$; in that case $q(a,D)$ is
$k$-regular with $S=\{\omega_0\}$ and the lag set $\supp\lambda$. If $\bar n(D)\notin b+k\Z$, then $R_\lambda(c_0)\ge a_2(D)$, with equality only for
$\lambda=\lambda^{\rm e}$, and
\[
 F_{q(a,D)}''(\omega_0)=2s_0\sin^3c'\,\bigl(a-a_2(D)\bigr),
\]
\begin{multline*}
 a_2(D)=\frac{\cot c'\,(1+c_0^2)}{2s_0^3}+\frac{(1+D)\,c_0}{2s_0^2}\\
 -\frac{\sin c'\cos c'\,\bigl(1+D+\cot c'\cot\omega_0\bigr)^2}{s_0}+\frac{\cos c'\,k^2\vartheta(1-\vartheta)}{4s_0\sin^3c'} .
\end{multline*}
\end{lemma}

\begin{proof}
Write $P_\lambda=\sum\lambda_n(1-\cos n\omega)$,
$Q_\lambda=\sum\lambda_n\sin n\omega$ and
$\mu=\sum_{n\in b+k\Z}\lambda_n>0$. The gcd condition gives
$P_\lambda>0$ on $(0,\pi]$. The Chebyshev form and the endpoint
limits make $\Psi_\lambda$ analytic through $[-1,1]$.
At $c_0$ its value is $h_0$; the slope constraint gives derivative
$1+D$. Thus $R_\lambda$ extends continuously. The master identity is
\begin{equation}\label{eq:TPsi}
 T_\lambda=\frac{\sin\Theta_qP_\lambda}{\sin\omega}
             (c-c_0)^2(R_\lambda(c)-a),\qquad c\ne c_0,
\end{equation}
and at the contact it gives
\begin{equation}\label{eq:Tpp}
 T_\lambda''(\omega_0)=4\sin^3c'\,s_0\mu(R_\lambda(c_0)-a).
\end{equation}
At $0$, the moment condition is exactly $\Xi^{a,D}(1)<\Psi_\lambda(1)$;
at $\pi$, both sides reduce to the strict tuning endpoint condition.
This proves the criterion and regularity.

Differentiating $\cot\Theta=\Xi/\sin\omega$ gives
\eqref{eq:nbaredge}; inserting its second derivative in the edge
curvature gives the displayed $a_2$. To prove sharpness, admissibility
and the edge identities give
\[
 \begin{aligned}
 4\sin^3c'\,s_0\mu\,[R_\lambda(c_0)-a_2(D)]
 =\cos c'\biggl[&\sum_{n\in b+k\Z}\lambda_n(n-n_1)(n-n_2)\\
 &+\sum_{m\ne0}\lambda_{mk}k^2m(m+1-2\vartheta)\biggr].
 \end{aligned}
\]
Every summand is nonnegative and the second group is strictly positive
on its support. Equality requires precisely the adjacent-lag weights.
\end{proof}

\begin{theorem}[Single-harmonic plateaus at large curvature]\label{thm:largea}
Let $D_{\max}\ge0$. Put
\[
 \begin{gathered}
 \bar N=2\sin^2c'\bigl(1+D_{\max}+\cot c'\,|\cot\omega_0|\bigr),\quad N_1=\bar N+\tfrac32k,\\
 s_h=\sin\frac\pi{2k},\quad H_A=1+\frac{(h_0+1)^2}{s_h^2},
 \end{gathered}
\]
and let $a_1=a_1(D_{\max})$ be the largest of
\[
 \frac{(h_0+2+2D_{\max})\,k^2N_1^2}{2s_h^2},\quad \frac{12\,k^2N_1^4}{s_h^2},\quad \frac{(1+D_{\max})^2}4,\quad
 \frac{\pi H_A\,(9k^2+3\pi k)}{16\,s_h^2}.
\]
If $|D|\le D_{\max}$, $a>a_1$ and $a>a_E(D)$, then $q(a,D)\in\mathcal B_K$, and the symbol with weights
$(\frac13-t,\frac13,\frac13+t)$ on the lags $n_0-k,n_0,n_0+k$ satisfies the conditions of Lemma~\ref{lem:ray}; here $n_0\in b+k\Z$ is a
member nearest to $\bar n(D)$ and $t=(\bar n(D)-n_0)/(2k)$. In particular, applied with $D_{\max}=|D|$, this gives
$a_F(D)\le\max\{a_1(|D|),a_E(D)\}<\infty$ for every $D$.
\end{theorem}

\begin{proof}
For the three stated lags, $|n_i|\le N_1$, each weight lies in
$[1/12,7/12]$, $\sum w_in_i=\bar n$, and their gcd is one.
Let $u_i=(n_i-\bar n)/2$,
$V=\sum w_iu_i^2$, $U_3=\sum w_i|u_i|^3$. Then
$k^2/24\le V\le9k^2/16$, $U_3\le27k^3/64$, and
$\sum w_in_i^2\ge k^2/6$.
Put $\delta_B=\pi/(kN_1)$, $J_B=[\omega_0-\delta_B,\omega_0+\delta_B]$,
and let $\varepsilon_C$ be the smaller cosine distance from its
endpoints to $c_0$. Since $\sin\omega\ge s_h$ on $J_B$,
$\varepsilon_C\ge2s_h/(kN_1)$, and the definition of $a_1$ gives
\begin{equation}\label{eq:aconds}
 a\varepsilon_C^2>2(h_0+2+2D_{\max}),\quad
 a\varepsilon_C^2>48N_1^2,\quad
 a\varepsilon_C^2>6\pi^2N_1/k^2,\quad
 a\varepsilon_C^2>24N_1/k^2.
\end{equation}
Outside $J_B$, $\Xi\le-a\varepsilon_C^2/2$. Write $\Theta=\pi-\varphi$;
then $\tan\varphi\le2\sin\omega/(a\varepsilon_C^2)$ and
$T\le-\cos\varphi(P-\tan\varphi|Q|)$. For
$\omega\le\pi/N_1$, use
$P\ge k^2\omega^2/(3\pi^2)$, $|Q|\le N_1\omega$.
For $\omega\ge\pi/N_1$, the consecutive-lag gcd gives
$P\ge1/(24N_1^2)$. Indeed, either the distance of $k\omega$ from
$2\pi\Z$ is at least $\pi/N_1$, forcing an adjacent phase distance
of at least $\pi/(2N_1)$, or $\omega$ is within $\pi/(kN_1)$ of a
nonzero $k$th root; coprimality then gives a phase distance of at
least $\pi/k$. In both cases use $1-\cos x\ge2x^2/\pi^2$ for
$|x|\le\pi$ and the weight lower bound. Also $|Q|\le1$.
Equation~\eqref{eq:aconds} makes both bounds strict and also gives
the positive endpoint moment at $0$.

On $J_A=\{\Theta\le\pi/2\}\subset J_B$, strict convexity gives
$\Theta''\ge as_h/H_A$. Set $\delta=\omega-\omega_0$,
$\ell=c'+\bar n\delta/2$, $g=\Theta-\ell$.
Then $g\ge as_h\delta^2/(2H_A)$, and Taylor's sine bound with
$\sum w_iu_i=0$ yields
\[
 T\le\delta^2\left[-\frac{2as_h^2}{\pi H_A}
                  +\frac98k^2+\frac{3\pi}8k\right]<0
 \quad(\delta\ne0).
\]
It also gives $T''(\omega_0)<0$. On $J_B\setminus J_A$, every
$e_i=c'+n_i\delta/2$ lies in $(0,\pi/2)$ while $\Theta>\pi/2$, so
$\Lambda_{n_i}=2\sin e_i\sin(e_i-\Theta)<0$.
The three zones prove the conditions of Lemma~\ref{lem:ray}, including
resonant slope indices. Taking $a$ down to the stated sufficient bound
gives the majorant for $a_F$.
\end{proof}

\begin{proof}[Proof of Theorem~\ref{thm:plateaujk}]
The certificate criterion shows every point of $\Pi_{j/k}$ is regular.
Conversely a singleton regular extremal has phase index $1$;
its tight lag set lies in $\mathcal N_{j/k}$ and satisfies the slope
condition. A common divisor larger than one would force a second
symbol zero on $(0,\pi]$, contradicting its strict negativity away
from the prescribed contact.
Thus its symbol is admissible and $a>a_F(D)$.
Its dip cannot contain an endpoint frequency $2\pi/K'$: that would
give a different single-harmonic extremal, contradicting uniqueness.
Hence it lies in $I_K$ and $a>a_E(D)$.
The continuous upper bound of Theorem~\ref{thm:largea} joins any two
vertical rays of the parameter domain, proving connectedness without
requiring continuity of $a_F$. Local constancy of support makes it
exactly one component of the regular boundary. The lower curvature
bound and its strict subcritical alternative follow from
Lemma~\ref{lem:ray} and Theorem~\ref{thm:fanjk}(d).
\end{proof}

\begin{proof}[Proof of Corollary~\ref{cor:everyfarey}]
The large-curvature certificate makes every $\Pi_{j/k}$ nonempty.
Every gap contains $2/(2K+1)$.
\end{proof}

\begin{proof}[Proof of Proposition~\ref{prop:location}]
Put $\epsilon=1-\beta$ and $t=a\epsilon^2$. Matching the quadratic
coefficients in the chart gives
$B=2\beta/(a\epsilon)$,
$A=2\beta(1+D+2ac_0)/(a(1+\beta))$, and
\[
 \frac{4\beta^2}{t}
 -t\left(1-\frac{4\beta c_0^2}{(1+\beta)^2}\right)-4\beta h_0
 +\frac{4\beta(1+D)c_0\epsilon^2}{(1+\beta)^2}
 -\frac{4\beta^2(1+D)^2}{a(1+\beta)^2}=0.
\]
The coefficient of $t$ lies in $[s_0^2,1]$, so this identity bounds
both $t$ and $1/t$ along the ray. Hence $\beta\to1$, and every limit
point solves $4/t-s_0^2t-4h_0=0$. Its unique positive root is
$2\tan(c'/2)/s_0$. Now $B\to0$, $A\to2c_0$, giving the other two
parameter limits and the stated limiting characteristic roots.
\end{proof}

\section{Tips, gap crossings, and junctions}\label{app:edges}
\subsection{Tongue tips and the lag-one boundary}\label{app:edges-tips}
For the fraction $1/K'$, put
$\chi=\Theta_*'(\omega_*)-1/2=(\bar n-1)/2$.
Strict convexity of $\Theta_*-\omega/2$ on its nonpositive sublevel
shows that its dip opens to the right of $\omega_*$ if $\chi<0$ and
to the left if $\chi>0$; at $\chi=0$ it is empty. For
$0<|\chi|<K'/2$ the adjacent lags are $1$ and
$1+\operatorname{sgn}(\chi)K'$, with respective weights
$1-2|\chi|/K'$ and $2|\chi|/K'$. Thus
\begin{equation}\label{eq:Fedge}
 F_q''(\omega_*)=\sin c'\,\Theta_*''(\omega_*)
                 -(K'-2|\chi|)|\chi|\cos c'.
\end{equation}
The fan pairs in rotation numbers are
$\{(m+1)/(K'm+K'-1),m/(K'm+K'-1)\}$ for $\varsigma=-1$, and
$\{m/(K'm+1),(m+1)/(K'm+1)\}$ for $\varsigma=1$, listing the nearer
point first. For large $m$ the upper point lies in the window of
$I_{K'-1}$ and the lower in that of $I_{K'}$. At the feasible shifts
of Theorem~\ref{thm:fanjk}(a), the anchor is the point on the side
where the dip opens.

\begin{proof}[Proof of Theorem~\ref{thm:tips}]
At the tip, $\bar n=1$ and
$\Lambda_1=-2\sin(\omega/2)\sin(\Theta-\omega/2)$ is strictly negative
away from $0,\omega_*$. Its second derivatives at those zeros are
$-(1+2\tau_q)<0$ and $-2\sin c'\,\Theta''(\omega_*)<0$.
The perturbation
\[
 T=(1-2\varepsilon)\Lambda_1+
       \varepsilon\Lambda_{1+K'}+\varepsilon\Lambda_{1-K'}
\]
has the same zero values and slopes. For sufficiently small
$\varepsilon>0$, strict curvature in fixed neighborhoods and compact
negativity elsewhere persist. Lemma~\ref{lem:certjk} gives one open
neighborhood with every claimed threshold, regularity, and extremal
conclusion. On it $\sigma\le0$ is exactly $s_*\le0$, the closed tongue;
both adjacent gaps have $s_*>0$.
\end{proof}

\begin{corollary}[A neighborhood of $\partial\Omega_1$]\label{cor:omega1}
Every point $q_0$ of $\partial\Omega_1$ has a neighborhood $V$ with $\mathcal D\cap V=\emptyset$ if $q_0$ is not
phase-locked, and $\mathcal D\cap V=\{q\in V:\Theta_q(2\pi/K')\le\pi/K'\}$ if $q_0$ is a tip of the
$K'$-tongue. In particular $\cl\mathcal B_K\cap\partial\Omega_1=\emptyset$ for every $K$: near $\partial\Omega_1$ the
phase diagram is that of level zero.
\end{corollary}

\begin{proof}
At a nonlocked tangency, Theorem~\ref{thm:tangency} gives a strict
finite symbol, whose sign persists uniformly in the tuning. The
locked points are the tips and have the neighborhoods above. A gap
contains neither the locked endpoint in its closed dip nor any
point of the local closed tongue; its boundary consequently cannot
accumulate on $\partial\Omega_1$.
\end{proof}

\subsection{Junctions}
For $K\ge3$ put $t_1=c_{K+1}$, $t_2=c_K$ and let
$q_\times(a)$ have quadratic
$\Xi(c)=1+c-a(c-t_1)(c-t_2)$, where $a>a_0(t_2,t_1)$.
Its closed dip is $[2\pi/(K+1),2\pi/K]$, so it lies on both tongue
edges and is approached from $I_K$.

\begin{proposition}[Junctions]\label{prop:corner}
Let $K\ge3$ and $K'=K+1$. At $q_\times(a)$, with respect to $\omega_*=2\pi/K'$, $\chi=-a(t_1-t_2)\sin^2c'$, and for
$|\chi|<K'/2$
\[
 F_q''(\omega_*)=|\chi|\cos c'\Bigl[4|\chi|-(K+3)+\frac{2(1-t_1)}{t_1-t_2}+\frac{t_1}{1+t_1}\Bigr].
\]
The interval
\[
 \mathcal J_K=\Bigl\{a:\ \frac{2}{(1-t_1)(1-t_2)}<a<\frac{K+3-2(1-t_1)/(t_1-t_2)-t_1/(1+t_1)}{4(t_1-t_2)\sin^2c'}\Bigr\}
\]
is nonempty and contained in $\{|\chi|<K'/2\}$, and for $a>a_0(t_2,t_1)$ with $\chi\notin(K'/2)\Z$ one has $F_q''(\omega_*)<0$
if and only if $a\in\mathcal J_K$.
Consequently, for $a\in\mathcal J_K$, $\sigma(q_\times(a))<0$, and all tunings of $I_K$ near $q_\times(a)$ lie in the interior of $\mathcal D$.
\end{proposition}
The proof follows Lemmas~\ref{lem:junction} and~\ref{lem:cornerineq}.

\begin{lemma}[The curve at a junction]\label{lem:junction}
Let $m\ge4$, $\omega_*=2\pi/m$, $c'=\pi/m$, $t_1=\cos\omega_*$, $t_2<t_1$, $a>0$, and let
$\Theta=\operatorname{arccot}\bigl(\Xi(\cos\omega)/\sin\omega\bigr)$ on $(0,\pi)$, where $\Xi(c)=1+c-a(c-t_1)(c-t_2)$. Then
$\Theta(\omega_*)=c'$, $\Theta'(\omega_*)-\tfrac12=-a(t_1-t_2)\sin^2c'$, and
\[
 \Theta''(\omega_*)=-\tfrac12\cot c'+\tfrac a2\tan c'\,(2-t_1^2-t_1t_2)+2\Theta'(\omega_*)^2\cot c' .
\]
\end{lemma}

\begin{proof}
Differentiate
$\cot\Theta=\cot(\omega/2)-a(\cos\omega-t_1)(\cos\omega-t_2)/\sin\omega$
at $\omega_*$. The perturbation has value zero, first derivative
$a(t_1-t_2)$ and second derivative
$a(t_1^2+t_1t_2-2)/\sin\omega_*$. The identities
$\Theta'=-\mathsf f'\sin^2\Theta$ and
$\Theta''=-\mathsf f''\sin^2\Theta+2\Theta'^2\cot\Theta$ for
$\mathsf f=\cot\Theta$ give both formulas.
\end{proof}

\begin{lemma}\label{lem:cornerineq}
For $m\ge4$ put $u_m=1-\sin^2(\pi/m)/\sin^2(\pi/(m-1))$ and $t=\cos(2\pi/m)$. Then
$4u_m+2/u_m+t/(1+t)<m+4$.
\end{lemma}

\begin{proof}
For $m=4,5$ the left sides are $22/3$ and $4\sqrt5-1$.
For $m\ge6$, the alternating sine bounds through degrees five and
three, together with monotonicity of $\sin x/x$, give
\[
 \frac{2m^2-m-10.57}{m^3}\le u_m\le\frac2m.
\]
Therefore the required left side minus $m$ is bounded above by
\[
 \frac8m+\frac{m^2+10.57m}{2m^2-m-10.57}+\frac12.
\]
This expression decreases for $m\ge6$ and is less than $3.63$ at
$m=6$, proving the strict inequality.
\end{proof}

\begin{proof}[Proof of Proposition~\ref{prop:corner}]
The parameterization of Proposition~\ref{prop:chart} gives a tuning
with the stated closed dip, and
$(a,t_2+\epsilon,t_1-\epsilon)$ approaches it from the gap.
Lemma~\ref{lem:junction}, substituted in \eqref{eq:Fedge}, gives the
displayed curvature. The lower endpoint of $\mathcal J_K$ is exactly
$a_0(t_2,t_1)$. Its being smaller than the upper endpoint is
Lemma~\ref{lem:cornerineq}, after using
$1-t_1=2\sin^2c'$ and
$u_m=1-\sin^2c'/\sin^2(\pi/(m-1))$, $m=K+1$.
On this interval $0<|\chi|<(m+2)/4\le m/2$.
Beyond the first slope strip, put $g=|\chi|>m/2$; the adjacent-lag
curvature formula and $\vartheta(1-\vartheta)\le1/4$ give
\[
 F_q''(\omega_*)\ge\cos c'[2g(g-1)-m^2/8]>0.
\]
Thus, at every nonresonant slope, negative curvature holds precisely
on $\mathcal J_K$. Theorem~\ref{thm:fanjk}(d) gives $\sigma<0$ there;
continuity gives an open divergent neighborhood reaching into $I_K$.
\end{proof}

\end{document}